\documentclass[a4paper,reqno]{amsart}

\usepackage{amssymb}
\usepackage{enumitem}
\usepackage{mathrsfs}
\usepackage{mathtools}
\usepackage{amsmath}
\usepackage{xcolor}
\usepackage[colorlinks=true]{hyperref}
\hypersetup{linkcolor=violet,urlcolor=cyan,citecolor=cyan}

\setlist[enumerate]{label=\emph{(\roman*)}}

\usepackage{environ}

\newtheorem{theorem}{Theorem}[section]
\newtheorem{corollary}[theorem]{Corollary}
\newtheorem{lemma}[theorem]{Lemma}
\newtheorem{proposition}[theorem]{Proposition}

\theoremstyle{definition}

\newtheorem{remark}[theorem]{Remark}
\numberwithin{equation}{section}
\newcommand{\R}{\mathbb{R}}

\def \C {{\mathbb{C}}}
\def \d {{\rm{d}}}
\def \del {\partial}
\def \pt{\partial_{t}}

\begin{document}
	
	\parindent=0pt
	
	\title[Large data global existence for wave systems]
	{Large data global existence for coupled massive-massless  wave-type systems}
	
	\author[Y.~Cai]{Yuan Cai}
	\address{Fudan University, School of Mathematical Sciences,  P.R. China.}
	\email{caiy@fudan.edu.cn}
	
	\author[S.~Dong]{Shijie Dong}
	\address{Southern University of Science and Technology, SUSTech International Center for Mathematics, and Department of Mathematics, 518055 Shenzhen, P.R. China.}
	\email{dongsj@sustech.edu.cn, shijiedong1991@hotmail.com}
	
	\author[K.~Li]{Kuijie Li}
	\address{Nankai University, School of Mathematical Sciences and LPMC, Tianjin, 300071, P.R. China.}
	\email{kuijiel@nankai.edu.cn}

	\author[J. Zhao]{Jingya Zhao}
	\address{Southern University of Science and Technology, and Department of Mathematics, 518055 Shenzhen, P.R. China.}
	\email{12231282@mail.sustech.edu.cn}
	\begin{abstract}
		We consider 3D  Klein-Gordon-Zakharov (KGZ) and Dirac-Klein-Gordon (DKG) systems, where a common feature is that there exist both massless and massive fields in each system. We establish global existence and asymptotic behavior for both systems with a class of large data. More precisely, in the KGZ system, we allow the massless field to be large, while in the DKG system we allow the massive field to be large. 
		
	\end{abstract}
	\maketitle
	
	\tableofcontents
	
	\section{Introduction}
	\subsection{Model problems and main results}

	We are interested in coupled massive-massless  wave-type equations arising from mathematical physics, and more precisely the Klein-Gordon--Zakharov(KGZ) model and the Dirac--Klein-Gordon(DKG) model. The models are of both hyperbolic and dispersive types, and have been actively studied in the past few decades; see for instance \cite{OTT95, OTT99, Kata, Dong2006, Bache, Boura00, DFS-07, Guo-N-W, WangX, Bejenaru-Herr, DongZoeDKG, DLMY} and the references therein.  In particular, we aim to investigate the global existence as well as asymptotic behavior for these models with some classes of large data.

	\smallskip

	\smallskip
	$\bullet$ 
	\textbf{Model I}: Klein-Gordon--Zakharov equations with a large massless field in $\R^{1+3}$.
	
	The KGZ system models the interaction between Langmuir waves and ion sound waves in plasma; see \cite{Za72,De90}. The model equations for the KGZ system read
	\begin{equation}\label{eq:KGZ-L}
		\left\{\begin{aligned}
			&-\Box E + E  = -nE,
			\\
			&-\Box n  =\Delta|E|^2,
		\end{aligned}\right.
	\end{equation}
	where $E:\R^+ \times \R^3 \to \R^3$ and $n:\R^+ \times \R^3 \to \R$. 
	The initial data are prescribed at $t=t_0=0$
	\begin{equation}\label{eq:KGZ-ID}
		\left(E, \del_t E,n,\del_t n\right)(t_0)=\left(E_{0},E_1,n_{0},n_{1}\right).
	\end{equation}
	In \eqref{eq:KGZ-L}, $\Delta = \partial_{x_1}^2 +\partial_{x_2}^2 +\partial_{x_3}^2 $ denotes the Laplace operator, and $\Box = -\partial_t^2 + \Delta$ denotes the d'Alembert operator.
	The size of $(n_0, n_1)$ is allowed to be large in a sense to be specified later.
	
	There exist  several small data results concerning the KGZ model in various spacetime dimensions; see for instance \cite{OTT95, Kata, DW-JDE, Dong2006}. In this paper, we investigate the KGZ system in the large data setting.
	
	Our main results on the KGZ system are as follows.
	
	\begin{theorem}[Global existence for KGZ with a large massless field I]\label{thm:KGZ-L}
		Consider the KGZ equations \eqref{eq:KGZ-L} and let $N\geq 10$ be an integer. For any  $K_0>1$, there exists an $\epsilon_0>0$, which depends on $K_0$ exponentially, such that for all initial data satisfying the conditions
		\begin{equation}\label{initialdata}
			\begin{aligned}
				\sum_{\substack{0\leq i\leq 10\\0\leq i+j\leq N}}\|\langle x\rangle^i\nabla^{i+j}n_0\|
				+\sum_{\substack{0\leq i\leq 9\\0\leq i+j\leq N-1}}\|\langle x\rangle^{i+1}\nabla^{i+j}n_1\|
				< K_0,\\
				\sum_{0\leq i\leq N+1}\|\langle x\rangle^{12}\nabla^i E_0\|+\sum_{0\leq i\leq N}\|\langle x\rangle^{12}\nabla^i E_1\| < \epsilon\leq\epsilon_0,
			\end{aligned}
		\end{equation}
		where $\|\cdot\|$ denotes the $L_x^2$ norm,  the Cauchy problem \eqref{eq:KGZ-L}--\eqref{eq:KGZ-ID} admits a global solution $(E, n)$, which decays as
		\begin{equation}\label{est:pointwise1}
			|E(t,x)|\leq C \epsilon^{1/2}\langle t+|x|\rangle^{-\frac{3}{2}},\quad |n(t,x)|\leq CK_0\langle t+|x|\rangle^{-1}\langle t-|x|\rangle^{-\frac{1}{2}},
		\end{equation}
		for some constant $C$.
		In addition, the solution $(E,n)$ scatters linearly.
	\end{theorem}

	\begin{theorem}[Global existence for KGZ with a large massless field II]\label{thm:KGZ-L2}
		Consider the KGZ equations \eqref{eq:KGZ-L}, let $N\geq 10$ be an integer, and let $C_{KS}$ be the constant in the Klainerman-Sobolev inequality stated in \eqref{est:Sobo}. For any  $K_0 > 1$, there exists an $\epsilon_0>0$, which depends on $K_0$ polynomially, such that for all initial data satisfying the conditions
		\begin{equation}\label{initialdata2}
			\begin{aligned}
				\sum_{\substack{0\leq i\leq 10\\0\leq i+j\leq N}}\|\langle x\rangle^i\nabla^{i+j}n_0\|
				+\sum_{\substack{0\leq i\leq 9\\0\leq i+j\leq N-1}}\|\langle x\rangle^{i+1}\nabla^{i+j}n_1\|
				< K_0,\\
				\sum_{0\leq i\leq N+1}\|\langle x\rangle^{12}\nabla^i E_0\|+\sum_{0\leq i\leq N}\|\langle x\rangle^{12}\nabla^i E_1\|<\epsilon\leq\epsilon_0,
			\end{aligned}
		\end{equation}
		and additionally 
		\begin{equation}\label{initialdata3}
			\begin{aligned}
				&n_0\geq 0,\quad n_1\geq|\nabla n_0|,\\
				&\sum_{i=0}^2 \big\| \langle x\rangle^i \nabla^{i+2} n_0 \big\|
				+		\sum_{i=0}^2 \big\| \langle x\rangle^i \nabla^{i+1} n_1 \big\|
				\leq {1\over 71^2\times174 \times C_{KS}^3 K_0^2},
			\end{aligned}
		\end{equation}
		the Cauchy problem \eqref{eq:KGZ-L}--\eqref{eq:KGZ-ID} admits a global solution $(E, n)$, which decays as
		\begin{equation}\label{est:pointwise2}
			|E(t,x)|\leq C \epsilon^{1/2} \langle t+|x|\rangle^{-\frac{3}{2}},\quad |n(t,x)|\leq CK_0\langle t+|x|\rangle^{-1}\langle t-|x|\rangle^{-\frac{1}{2}},
		\end{equation}
		for some constant C.
		In addition, the solution $(E,n)$ scatters linearly.
	\end{theorem}

	\begin{remark}	
		In Theorems \ref{thm:KGZ-L}--\ref{thm:KGZ-L2}, we prove global existence for the KGZ system with a large wave field and a small Klein-Gordon(KG) field.
		We reformulate \eqref{eq:KGZ-L}--\eqref{eq:KGZ-ID} to be
		\begin{equation}\label{eq:KGZ-L2}
			\left\{\begin{aligned}
				&-\Box E + E  = -n^0E - \Delta n^1E,
				\\
				&-\Box n^0 = 0,
				\\
				&-\Box n^1 =|E|^2,
			\end{aligned}\right.
		\end{equation}
		with initial data
		\begin{equation}\label{eq:KGZ-ID2}
			\left(E, \del_t E,n^0,\del_t n^0, n^1,\del_t n^1\right) (t_0)=\left(E_0, E_1,n_{0},n_{1}, 0, 0\right).
		\end{equation}
		We notice that $(E=0, n=n^0)$ is a non-zero solution to the KGZ system \eqref{eq:KGZ-L}. Thus Theorems \ref{thm:KGZ-L}--\ref{thm:KGZ-L2} can also be regarded as stability results for the KGZ system around nontrivial large solutions $(E=0, n=n^0)$ (as a contrast, we call $(E=0, n=0)$ a trivial solution to the KGZ system).
	\end{remark}

	\begin{remark}
		In Theorems \ref{thm:KGZ-L}--\ref{thm:KGZ-L2}  the wave field can be arbitrarily large of size $K_0 >1$, but the smallness requirement on the KG field of size $\epsilon$ depends on $K_0$. The dependence of $\epsilon$ on $K_0$ is also of interest to us. Under the assumptions in \eqref{initialdata}, we prove global existence for the KGZ system in Theorem \ref{thm:KGZ-L} with $\epsilon_0$ depending on $K_0$ exponentially, and more precisely $\epsilon_0\lesssim K_0^{-\frac{3}{4}(N+1)}\exp(-K_0^2)$. If we additionally assume \eqref{initialdata3} in Theorem \ref{thm:KGZ-L2}, then the dependence of $\epsilon_0$ on $K_0$ can be reduced to be polynomially, and which reads $\epsilon_0\lesssim K_0^{-\frac{3}{4}N-48}$.
	\end{remark}
	
	\begin{remark}
		We rewrite \eqref{eq:KGZ-L} 
		\begin{equation}\label{eq:KGZ-L3}
			\left\{\begin{aligned}
				&-\Box E + (1+n)E  = 0,
				\\
				&-\Box n  =\Delta|E|^2.
			\end{aligned}\right.
		\end{equation}
		In this way, we regard the component $E$ to satisfy a Klein-Gordon equation with varying mass $1+n$. Interestingly, the mass $1+n$ is allowed to be negative/zero in some spacetime domains/points since we allow the field $n$ to be large, e.g. we set $n_0(x) = -2$ for $|x|\leq 1$. Of course, the varying mass $1+n$ will be close to $1$ when time $t$ is large thanks to the decay property of the wave field $n$.
	\end{remark}
	
	\begin{remark}
		In the proof, we rely on the decay property of the fields $n, E$. It is well-known that solutions to wave equations or Klein-Gordon equations decay faster in higher dimensional spacetime.
		Therefore, our proof also applies to higher dimensional cases, i.e., $\mathbb{R}^{1+d}$ with $d\geq 4$.
	\end{remark}
	\begin{remark}
	 The constant appearing in the upper bound of \eqref{initialdata3} is due to some technical reasons and is not optimal; see Section \ref{S:proofTh1.4}. This is not the focus of the paper and we will not discuss further. 	
	\end{remark}
	
	\smallskip
	$\bullet$  
	{\textbf{Model II}}: Dirac--Klein-Gordon equations with a large massive field in $\R^{1+3}$.
	
	The DKG system,  a basic model in particle physics, describes interactions between a scalar field and a Dirac field. 
	The model equations read
	\begin{equation}\label{eq:DKG-L}
		\left\{\begin{aligned}
			&-i\gamma^{\mu}\partial_{\mu}\psi  = v \psi,
			\\
			&-\Box v + v =\psi^{*}\gamma^0 \psi,
		\end{aligned}\right.
	\end{equation}
	where $\psi(t,x):\R^+\times \R^3 \to \C^4$ denotes the Dirac field and  $v(t,x):\R^+\times \R^3 \to \R$ denotes scalar field.
	The initial data are prescribed at $t=t_0=0$
	\begin{equation}\label{eq:DKG-ID}
		\left(\psi,v,\del_t v\right)(t_0)=\left(\psi_{0},v_{0},v_{1}\right),
	\end{equation}
	and the size of $(v_0, v_1)$ is allowed to be large in a sense to be specified later.  In \eqref{eq:DKG-L}, $\left\{\gamma^{0},\gamma^{1},\gamma^{2}, \gamma^3\right\}$ represent the Dirac matrices and satisfy the identities
	\begin{equation}\label{equ:gamma}
		\left\{\gamma^{\mu},\gamma^{\nu} \right\}:=\gamma^\mu \gamma^\nu + \gamma^\nu \gamma^\mu = -2\eta_{\mu\nu}I_{4},\quad 
		(\gamma^{\mu})^{*}=-\eta_{\mu\nu}\gamma^{\nu},
	\end{equation}
	in which $A^{*}$ is the conjugate transpose of the matrix $A$, $I_{4}$ is the $4\times 4$ identity matrix, and $\eta=\rm{diag}(-1,1,1,1)$ denotes the Minkowski metric in $\R^{1+3}$. We mention that Einstein summation convention over repeated indices is adopted throughout the article.

	\begin{theorem}[Global existence for DKG with a large massive field]\label{thm:DKG-L}
		Consider the DKG equations \eqref{eq:DKG-L} and let $N\geq 13$ be an integer. For any $K_0>1$, there exists an $\epsilon_0>0$, which depends on $K_0$ polynomially, such that for all initial data satisfying the conditions
		\begin{align}
			\sum_{|I|\leq N+1}\|\langle x \rangle^{N+2} \nabla^{I} v_0\| + \sum_{|I|\leq N}\|\langle x \rangle^{N+1} \nabla^{I} v_1\| & \leq K_0,  \label{initikg} \\
			\sum_{|I|\leq N} \|\langle x \rangle^{N+1} \nabla^I \psi_0\| \leq \epsilon<\epsilon_0,  \label{initidirac}
		\end{align}
		the Cauchy problem \eqref{eq:DKG-L}--\eqref{eq:DKG-ID} admits a global solution $(\psi, v)$, which decays as
		\begin{align}
			|v(t,x)| \leq C K_0 \langle t +|x| \rangle^{-\frac 32}, \ \ \ |\psi(t,x)| \leq C\epsilon^{1/2} \langle t + |x|\rangle^{-1},
		\end{align}
		for some constant $C$. Moreover, the solution $(v, \psi)$ scatters linearly. 
	\end{theorem}
	
	\begin{remark}
		In Theorem \ref{thm:DKG-L}, we illustrate global existence for a DKG system with a large KG field and a small Dirac field.
		We reformulate \eqref{eq:DKG-L}--\eqref{eq:DKG-ID} to be
		\begin{equation}\label{eq:DKG-L2}
			\left\{\begin{aligned}
				&-i\gamma^{\mu}\partial_{\mu}\psi  = V^0 \psi + V^1 \psi,
				\\
				&-\Box V^0 + V^0 = 0,
				\\
				&-\Box V^1 + V^1 =\psi^{*}\gamma^0 \psi,
			\end{aligned}\right.
		\end{equation}
		with initial data
		\begin{equation}\label{eq:DKG-ID2}
			\left(\psi,V^0,\del_t V^0, V^1,\del_t V^1\right) (t_0)=\left(\psi_{0},v_{0},v_{1}, 0, 0\right).
		\end{equation}
		We note $(\psi = 0, v=V^0)$ is a non-zero solution to the DKG system \eqref{eq:DKG-L}. Therefore, Theorem \ref{thm:DKG-L} provides a stability result for the DKG system around nontrivial large solutions $(\psi = 0, v=V^0)$  (as a contrast, we call $(\psi = 0, v=0)$ a trivial solution to the DKG system).
	\end{remark}
	
	\begin{remark}
		In Theorem \ref{thm:DKG-L}, the size of $\epsilon_0$ depends polynomially with respect to $K_0$,  and more precisely, we have 
		$$
		\epsilon_0 \lesssim K_0^{-6N^2};
		$$
		see \eqref{choiceep} in Section \ref{S:proofdkg}. 
	\end{remark}
	
	\begin{remark}
		In Theorem \ref{thm:DKG-L}, we require quite high weights on the initial data, which can be improved using the argument treating Theorems \ref{thm:KGZ-L}--\ref{thm:KGZ-L2}. The main reason why  we choose high weight here is to make the proof neater.
	\end{remark}

	\begin{remark}
		In Theorem \ref{thm:DKG-L}, we treat the Dirac-Klein-Gordon system in the case of a massless Dirac field and a massive Klein-Gordon field. In fact, our proof also applies to the case where both Dirac and Klein-Gordon fields are massive (i.e., replacing $-i\gamma^\mu \partial_\mu \psi$ with $-i\gamma^\mu \partial_\mu \psi+ \psi$ in \eqref{eq:DKG-L}). 
		
		We can also treat higher dimensional cases, i.e., $\mathbb{R}^{1+d}$ with $d\geq 4$.
	\end{remark}

	\smallskip

	\subsection{Brief discussion on related results}
The study of nonlinear wave-type equations has been an active research field for decades. The foundational works of Christodoulou \cite{Christodoulou86} and Klainerman \cite{Kl86, KlainWave} on massless wave equations, along with Klainerman \cite{Klainerman85} and Shatah \cite{Shatah} on massive Klein-Gordon equations, have significantly contributed to the field. Inspired by these pioneering efforts, researchers have focused on establishing global existence with pointwise decay results for coupled massless-massive wave-type equations with small data, encompassing both physical models and those of mathematical interest. In this context, the first notable results were achieved by Bachelot \cite{Bache} on certain Dirac-Klein-Gordon equations and by Georgiev \cite{Geor90} on a coupled wave-Klein-Gordon system. We will not provide an exhaustive list but only recall part of the literatures: the Maxwell-Dirac model by Georgiev \cite{Geor91}, the KGZ model by Ozawa-Tsutaya-Tsutsumi \cite{OTT95}, the Maxwell-KG model by Psarelli \cite{PsarelliA} and Klainerman-Wang-Yang \cite{Klainerman-WY}, the Einstein-KG model by LeFloch-Ma \cite{PLF-YM-cmp}, Wang \cite{WangQ}, and Ionescu-Pausader \cite{IoPa-WKG, IoPa-EKG}, the Kaluza-Klein model by Huneau-Stingo-Wyatt \cite{Zoe23}, among others. Without smallness assumptions on the data, less is known.

	In the remarkable works \cite{Klainerman-WY} by Klainerman-Wang-Yang and \cite{FaWaYa} by Fang-Wang-Yang, the authors established global existence with pointwise decay for the  Maxwell-KG equations with a large Maxwell field and a small KG field. Our work in the present paper is motivated by these results.  Other  large data global existence results for wave-type equations include but not limited to: a class of spherically symmetric data with large bounded variation norms on Einstein-scalar model by Luk-Oh-Yang \cite{LuOhYa} (see also \cite{LuOh}); short pulse type large data, introduced by Christodoulou \cite{Christodoulou09}, on null wave equations by  Miao-Pei-Yu \cite{MiPeYu}; global dynamics of Maxwell-KG by Yang \cite{Ya18}, Yang-Yu \cite{YaYu19}  and Yang-Mills-Higgs by Wei-Yang-Yu \cite{WeYaYu} with large data;
	a class of data with large $L^2$ norm (dispersed large data) on the cubic Dirac model by Dong-Li-Zhao \cite{DoLiZh}; global existence for 2D Dirac-Klein-Gordon system with large data by Gr\"unrock-Pecher \cite{GrPe}.

	\subsection{Key strategies and ideas}
	
	We will rely on Klainerman's  vector-field method to prove the main theorems. In addition to the difficulty caused by the fact that the systems are not consistent with the scaling vector field $L_0 = t\partial_t + x^a\partial_a$, the primary difficulties include the insufficient decay rates of the solutions (e.g. the massless wave-type components decay at $t^{-1}$ or an even worse rate $t^{-3/4}$ due to the use of Sobolev inequality \eqref{est:globalSobo}) and the presence of the large data. Therefore, to prove large data global existence results, we need to balance the decay rates and the largeness of the fields. 
	
	We first discuss the Klein-Gordon-Zakharov equations \eqref{eq:KGZ-L}. In this case, we allow the wave component $n$ to be large, while the Klein-Gordon field $E$ needs to be small. In the equation of $E$, the nonlinearity $-nE$ contains an undifferentiated wave component. Recall that in the natural energy of wave equations, only the differentiated solution appears, i.e.,
	$$
	\int_{\mathbb{R}^3} |\partial n|^2 \, dx.
	$$
	Thus, it is generally more difficult to estimate wave components without any partial derivatives, and consequently it is hard to bound the term $-nE$. To address this problem, we use the reformulation \eqref{eq:KGZ-L2} used in \cite{Kata} by Katayama to decompose $n = n^0 + \Delta n^1$, in which $n^0$ is a free wave that can be easily estimated and $n^1$ is the solution to a wave equation. Specifically, 
	\begin{align*}
		& -\Box n^0 =0, \ \ \qquad \	  n^0(0,x) = n_0, \ \    \partial_t n^0(0,x) = n_1,\\
		&-\Box n^1 =|E|^2, \ \ \qquad n^1(0,x) = 0, \ \    \partial_t n^1(0,x) = 0.
	\end{align*}
	Our bootstrap assumption is imposed on the ghost weight energy and the $L^{\infty}$ norm of $E$ which are expected to be small(of order $\epsilon$) due to the restriction on the initial data.   To refine the ghost weight energy of Klein-Gordon solution $E$, we observe  
	\begin{align} \label{ineq:examghost}
		\mathcal{E}_{gst,1}(t, E) \lesssim \mathcal{E}_{gst,1}(0, E) + \int_0^t\int_{\R^3}|(n^0E+\Delta n^1 E)\partial_t E|\, \d x\d s.
	\end{align}
	Since $n^0$
  is a linear wave solution with large initial data, its size is large. Therefore, the problematic term in the above estimate is 	\begin{align*}
		\int_0^t \int_{\R^3} |n^0 E \partial_t E|\, \d x\d s & \lesssim C_{\eta}\int_0^t\|\langle s -r \rangle^{\frac 12 +\delta} n^0\|^2_{L^{\infty}}\|\partial_t E\|^2 \d s + \eta \mathcal{E}_{gst,1}(t, E) \\
		& \lesssim C_{\eta}\int_0^t\|\langle s -r \rangle^{\frac 12 +\delta} n^0\|^2_{L^{\infty}}\mathcal{E}_{gst,1}(s, E)  \d s + \eta \mathcal{E}_{gst,1}(t, E),
	\end{align*}
	where $\eta$ can be chosen to be sufficiently small. Given the equation of $n^1$, we can expect the size of $n^1$ to be of order $\epsilon^2$.  Thus, we can roughly estimate 
	\begin{align*}
		\mathcal{E}_{gst,1}(t, E) & \lesssim \mathcal{E}_{gst,1}(0, E) + {\rm smaller \  term(order\ \epsilon^2) }  \\
		& \qquad  + \int_0^t\|\langle s -r \rangle^{\frac 12 +\delta} n^0\|^2_{L^{\infty}}\mathcal{E}_{gst,1}(s, E) \,   \d s.
	\end{align*}
	Then Gronwall inequality implies 
	\begin{align*}
		\mathcal{E}_{gst,1}(t, E) & \lesssim (\mathcal{E}_{gst,1}(0, E) + {\rm smaller \  term(order\ \epsilon^2) } )\exp{\int_0^t}\|\langle s -r \rangle^{\frac 12 +\delta} n^0\|^2_{L^{\infty}} \d s,
	\end{align*}
	which makes it possible to improve the bootstrap assumption. In Theorem \ref{thm:KGZ-L}, the dependence of 
$\epsilon_0$ on $K_0$
  is exponential due to the application of the Gronwall inequality.   To relax the dependence of $\epsilon_0$ on $K_0$ to  a polynomial one, we need to rely on more detailed information about the solution and a modified version of the ghost weight energy estimate, and we thus need to impose further restrictions on the initial data $(n_0, n_1)$. In this case, the weighted energy of $E$ takes the following formula
	\begin{equation*}
		\begin{aligned}
			&\frac{1}{2}\int_{\R^{3}}e^q\left(|\partial E|^2+| E|^2+n^0|E|^2\right)\d x(t)\\
			&+\frac{1}{2}\int_{0}^{t}\int_{\R^{3}}e^q\left(\frac{n^0|E|^2}{\langle r-s\rangle^{1+2\delta}}+\Big(\frac{1}{\langle r-s\rangle^{1+2\delta}} -\partial_t n^0 \Big)|E|^2 +\sum_{a=1}^{3}\frac{|G_a E|^2}{\langle r-s\rangle^{1+2\delta}}\right)\d x\d s\\
			&=\frac{1}{2}\int_{\R^{3}}e^q\left(|\partial E|^2+|E|^2+n^0|E|^2\right)\d x(0)
			-\int_{0}^{t}\int_{\R^{3}}e^q(\Delta n^1E)\partial_t E\d x\d s.
		\end{aligned}
	\end{equation*}

	Our observation is that if 
	\begin{align}  \label{enerestri}
		n^0 \geq 0, \ \ \  \frac{1}{\langle r-s\rangle^{1+2\delta}} -\partial_t n^0 \geq 0,
	\end{align}
	then we have 
	\begin{align*}
		\mathcal{E}_{gst,1}(t,E) \lesssim \mathcal{E}_{gst,1}(0, E) + \|n_0\|_{L^{\infty}}\|E_0\|^2+\int_{0}^{t}\int_{\R^{3}}|(\Delta n^1E)\partial_t E|\,\d x\d s. 
	\end{align*}
	Compared with \eqref{ineq:examghost}, the problematic term $n^0 E \partial_t E$ disappears, allowing us to establish a polynomial dependence  between $\epsilon_0$ and $K_0$. Motivated by this, we identify a sufficient condition  \eqref{initialdata3} to ensure the validity of \eqref{enerestri}.

	Next, we consider the Dirac-Klein-Gordon system \eqref{eq:DKG-L}. In this case, we take the initial data of the massive Klein-Gordon component to be large, while the data of the massless scalar field remain small. Heuristically, we expect the size of the data to be preserved as the system evolves, provided the smallness condition plays the dominant role.  
	
	Since our proof of global existence result is based on a continuity argument, the first difficulty is to establish appropriate bootstrap assumptions that can be refined. For this purpose, let us first observe the ghost energy estimate of Dirac equation (the definition of $\mathcal{E}_{D}$ can be found in Section \ref{S:pre}), for $|I|\geq 1$,
	\begin{align*}
		\mathcal{E}_{D}(t, \widehat{\Gamma}^I \psi) &\lesssim \mathcal{E}_{D}(0, \widehat{\Gamma}^I \psi) + \int_0^t \|(\widehat{\Gamma}^I \psi)^* \gamma^0 \widehat{\Gamma}^I(v\psi)  - (\widehat{\Gamma}^I(v\psi))^* \gamma^0 \widehat{\Gamma}^I \psi\|_{L^1} \, \d s \\
		&  \lesssim \mathcal{E}_{D}(0, \widehat{\Gamma}^I \psi) + \sum_{\substack{|I_1|+|I_2|\leq |I|\\|I_2|<|I|}} \int_0^t \big\|(\Gamma^{I_1}v) (\widehat{\Gamma}^I\psi)^* \gamma^0 \widehat{\Gamma}^{I_2} \psi  \big\|_{L^1} \,\d s.
	\end{align*}
	One can note that there is a cancellation at the top order derivative of $\psi$ and thus producing a lower order term $\widehat{\Gamma}^{I_2}\psi$ with $|I_2|<|I|$, which motivates us to introduce a bound depending on index $I$ so that there is room to improve the estimate, and more precisely, we set one of the assumptions as 
	\begin{align}
		\mathcal{E}_{D}^{\frac 12}(t,\widehat{\Gamma}^I\psi) \leq C_1 \epsilon^{1-\delta|I|}.
	\end{align}
	We point that the following version ghost energy estimate for Dirac equation has been extensively used in \cite{DLMY, DongZoeDKG}, 
	\begin{align}
		\mathcal{E}_{D}(t, \widehat{\Gamma}^I \psi) &\lesssim \mathcal{E}_{D}(0, \widehat{\Gamma}^I \psi) + \int_0^t \|(\widehat{\Gamma}^I \psi)^* \gamma^0 \widehat{\Gamma}^I(v\psi) \|_{L^1} \, \d s. \nonumber
	\end{align}	
	In the present article, the aforementioned cancellation property plays a key role, especially in obtaining the polynomial dependence between $\epsilon_0$ and $K_0$. 
	
	Another difficulty lies in the slow decay of the nonlinearities. We want to handle this issue by applying the ghost energy estimate, and in this way, an extra $t-r$ decay for Dirac field $\psi$ is required (see Lemmas \ref{lem:ghostKG}, \ref{le:EneryDirac}). Note that the global Sobolev inequality \eqref{est:globalSobo} only provides a slow $\langle t+r \rangle^{-\frac 34}$ decay rate. In order to gain the $\langle t- r\rangle$ decay for $\psi$, we adopt an idea due to Bournaveas \cite{Boura00}, introducing a wave solution $\Psi$ which solves
	\begin{align}
		-\Box \Psi = v\psi.  \label{psiwaverep}
	\end{align} 
	By setting appropriate initial data for $\Psi$, we can derive $i \gamma^{\mu}\partial_{\mu} \Psi = \psi$. Then  $\langle t -r \rangle$ decay   can be obtained by using the  following inequality 
	\begin{align}
		|\psi| \lesssim |\partial \Psi| \lesssim \langle t- r\rangle^{-1} \big(|L_0 \Psi| + |\Gamma \Psi| \big).
	\end{align} 
	To achieve our goal, it suffices to establish the decay of $L_0 \Psi$ and $\Gamma \Psi$. We aim to bound the conformal energy of $\Psi$ and then apply the Sobolev inequality to deduce the pointwise decay.  However,  severe difficulties arise due to the slow decay of the nonlinearity in \eqref{psiwaverep}.  To address this problem, we perform a nonlinear transformation by introducing another unknown $\widetilde{\Psi}:= \Psi + v\psi$,  which satisfies a wave equation with improved nonlinearities (quadratic null terms and cubic terms), i.e.
	\begin{align}  \label{nonlipsi}
		-\Box \widetilde{\Psi} = (\psi^* \gamma^0 \psi) \psi + i \gamma^{\mu} v  \partial_{\mu}(v\psi) + 2Q_0(v,\psi).
	\end{align}
	Compared with \eqref{psiwaverep}, the nonlinearities in \eqref{nonlipsi} exhibit faster decay,  
	making it possible to obtain a good bound on the conformal energy of $\widetilde{\Psi}$. In fact, a significant portion of Section \ref{S:proofdkg} focuses on establishing an appropriate bound for the conformal energy of $\widetilde{\Psi}$. This includes carefully handling the quadratic null term  $Q_0(v,\psi)$, where we utilize techniques such as dividing the integration domain, exploiting the extra $\langle t -r \rangle$ decay for $\partial \psi$ (see Lemma \ref{le:Diracdecay}) inside the light cone and using the additional decay of the Klein-Gordon solution(see Lemma \ref{est:weightKG}).

	\subsection{Some open problems}
	
	In the course of proving Theorems \ref{thm:KGZ-L}, \ref{thm:KGZ-L2}, and \ref{thm:DKG-L}, several problems remain unknown to us, which are listed below.
	
	\
	
	\smallskip
	$\bullet$  What is the best constant in Klainerman-Sobolev inequality \eqref{est:Sobo}? Determining this constant would allow us to state Theorem \ref{thm:KGZ-L2} more concretely.
	
	\smallskip
	$\bullet$ Can one prove the Sobolev inequality 
	\begin{align}\label{eq:Sobolev-open}
		|u(t,x)|\lesssim\langle t+r\rangle^{-1}\sum_{|I|\leq3}\|\Gamma^Iu(t,x)\|,
	\end{align}						
	whereas  in \eqref{est:globalSobo} the decay factor is $\langle t+r\rangle^{-3/4}$?
	If this stronger inequality could be established, many proofs would be simplified.		
	
	\smallskip
	$\bullet$ Can one prove a similar global existence result as in Theorem \ref{thm:DKG-L}, allowing the massless Dirac field to be large in a certain sense, specifically, belonging to the dispersed large  data set (large in $L^2$ but small when hit with derivatives) considered in \cite{DoLiZh}? If so, this would mean that both fields are permitted to be large.	
	
	\subsection{Organization} The article is organized as follows. 	
	In Section \ref{S:pre}, we prepare some preliminaries, including notation, various energy estimates for wave-type equations, commutator estimates, Sobolev type inequalities, etc.
	In Sections \ref{S:proofTh1.3} and \ref{S:proofTh1.4},  we prove large data global existence for the KGZ system in Theorems \ref{thm:KGZ-L} and \ref{thm:KGZ-L2}, respectively.
	Finally, we present the proof of large data global existence for the DKG system in Theorem \ref{thm:DKG-L} in Section \ref{S:proofdkg}.

	\section{Preliminaries}\label{S:pre}
	\subsection{Notation and conventions}\label{SS:Notation}
	
	Our problem is in (1+3) dimensional spacetime $\R^{1+3}$. A spacetime point in $\R^{1+3}$ is represented by 
	$(t,x)=(x_{0},x_{1},x_{2},x_{3})$, and its spatial radius is denoted by $r=|x|=\sqrt{x_{1}^{2}+x_{2}^{2}+x_{3}^{2}}$. Set $\omega_{a}=\frac{x_{a}}{r}$ for $a=1,2,3$ and $x=(x_{1},x_{2},x_{3})\in \R^{3}$. We denote spacetime indices by Greek letters $\left\{\alpha,\beta,\dots\right\}$ and spatial indices by Roman indices $\left\{a,b,\dots\right\}$. Unless otherwise stated, we will always adopt the Einstein summation convention for repeated upper and lower indices.
	
	Following Klainerman~\cite{Kl86}, we introduce the following vector fields
	\begin{enumerate}
		\item [(1)] Translations: $\partial_{\alpha}=\partial_{x_{\alpha}}$, for $\alpha=0,1,2,3$.
		\item [(2)] Rotations: $\Omega_{ab}=x_{a}\partial_{b}-x_{b}\partial_{a}$, for $1\leq a<b\leq 3$.
		\item [(3)] Scaling vector field: $L_0=t\partial_t +x^{a}\partial_{a}$.
		\item [(4)] Lorentz boosts: $L_{a}=t\partial_{a}+x_{a}\partial_{t}$, for $a=1,2,3$.
	\end{enumerate}
	To treat Dirac equations, we also introduce the modified rotations and Lorentz boosts introduced by Bachelot~\cite{Bache}
	\begin{equation*}
		\widehat{\Omega}_{ab}=\Omega_{ab}-\frac{1}{2}\gamma^{a}\gamma^{b}\quad \mbox{and}\quad 
		\widehat{L}_{a}=L_{a}-\frac{1}{2}\gamma^{0}\gamma^{a}.
	\end{equation*}
	For convenience, we denote
	\begin{equation*}
		\begin{aligned}
			\widehat{\Omega}=(\widehat{\Omega}_{12},\widehat{\Omega}_{13},\widehat{\Omega}_{23}),\quad
			\widehat{L}=(\widehat{L}_{1},\widehat{L}_{2},\widehat{L}_{3}).
		\end{aligned}
	\end{equation*}
	We define the following four ordered sets of vector fields,
	\begin{equation*}
		\begin{aligned}
			&\Gamma=\left(\Gamma_{1},\dots,\Gamma_{10}\right)=\left(\partial,\Omega,L\right),\quad \ \ \ \
			\widehat{\Gamma}=\left(\widehat{\Gamma}_{1},\dots,\widehat{\Gamma}_{10}\right)=\left(\partial,\widehat{\Omega},\widehat{L}\right),\\
			&Z=\left(Z_{1},\dots,Z_{11}\right)=\left(L_0,\partial,\Omega,L\right),\ \
			\widehat{Z}=\left(\widehat{Z}_{1},\dots,\widehat{Z}_{11}\right)=\left(L_0,\partial,\widehat{\Omega},\widehat{L}\right),\\
		\end{aligned}
	\end{equation*}
	where 
	\begin{equation*}
		\begin{aligned}
			\partial=(\partial_0,\partial_1,\partial_2,\partial_3)=(\partial_t,\nabla),\ \ \Omega=(\Omega_{12},\Omega_{13},\Omega_{23}), \ \  L=(L_1,L_2,L_3).
		\end{aligned}
	\end{equation*}
	
	Moreover, for all multi-index $I=(I_{1},\dots,I_{10})\in \mathbb{N}^{10}, I'=(I'_{1},\dots,I'_{11})\in \mathbb{N}^{11}$, we denote
	\begin{equation*}
		\Gamma^{I}=\prod_{k=1}^{10}\Gamma_{k}^{I_{k}},\quad 
		\widehat{\Gamma}^{I}=\prod_{k=1}^{10}\widehat{\Gamma}_{k}^{I_{k}},\quad
		Z^{I'}=\prod_{k=1}^{11}Z_{k}^{I'_{k}},\quad 
		\widehat{Z}^{I'}=\prod_{k=1}^{11}\widehat{Z}_{k}^{I'_{k}}.
	\end{equation*}
	
	Also, for any $\R$-valued, or vector-valued function $f$, we have 
	\begin{equation*}
		\begin{aligned}
			\left|\Gamma f\right|=\left(\sum_{k=1}^{10}\left|\Gamma_{k}f\right|^{2}\right)^{\frac{1}{2}},\quad  \big|\widehat{{\Gamma}}f\big|=\left(\sum_{k=1}^{10}\big|\widehat{\Gamma}_{k}f\big|^{2}\right)^{\frac{1}{2}},\\
			\left|Z f\right|=\left(\sum_{k=1}^{11}\left|Z_{k}f\right|^{2}\right)^{\frac{1}{2}},\quad 
			\big|\widehat{{Z}}f\big|=\left(\sum_{k=1}^{11}\big|\widehat{Z}_{k}f\big|^{2}\right)^{\frac{1}{2}}.
		\end{aligned}
	\end{equation*}
	
	%To state the global Sobolev inequality, we also define 
	%\begin{equation*}
	%\left(\Lambda_{1},\Lambda_{2}\right)=\left(\partial_{r},\Omega\right)\quad \mbox{and}\quad
	%\Lambda^{I}=\Lambda_{1}^{i_{1}}\Lambda_{2}^{i_{2}},\ \  \mbox{for}\ I=(i_{1},i_{2})\in \mathbb{N}^{2}.
	%\end{equation*}
	Additionally, we will use the good derivatives 
	
	$$G_{a}=\partial_{a}+\omega_{a}\partial_{t}, \quad \mbox{for} \quad a=1,2,3.$$
	
	Following Dong-Wyatt~\cite{DongZoeDKG}, we introduce the following notation in order to explain the hidden structure of the nonlinear term of~\eqref{eq:DKG-L}: for all $\phi=\phi(t,x):\R^{1+3}\to \C^{4}$, we set 
	\begin{equation}\label{def:phi-}
		\begin{aligned}
			\left[\phi\right]_{+}(t,x)&:=\phi(t,x)+\omega_{a}\gamma^{0}\gamma^{a}\phi(t,x),\\
			\left[\phi\right]_{-}(t,x)&:=\phi(t,x)-\omega_{a}\gamma^{0}\gamma^{a}\phi(t,x).
		\end{aligned}
	\end{equation}
	
	We also need to introduce a suitable cut-off $C^{1}$ function $\chi:\R\to \R$ such that 
	\begin{equation}\label{def:chi}
		\chi'\ge 0\quad \mbox{and}\quad  \chi=\left\{\begin{aligned}
			&0,\quad \mbox{for}\ x\le 1,\\
			&1,\quad \mbox{for}\ x\ge 2.
		\end{aligned}\right.
	\end{equation}

	Let $\left\{\zeta_{k}\right\}_{k=0}^{\infty}$ be a Littlewood-Paley partition of unity, i.e.,
	\begin{equation*}
		1=\sum_{k=0}^{\infty}\zeta_{k}(\rho),\ \rho\ge 0,\ \zeta_{k}\in C_{0}^{\infty}\left(\R\right),\quad \zeta_{k}\ge 0\quad \mbox{for all}\ k\ge 0,
	\end{equation*}
	and
	\begin{equation*}
		\mbox{supp}\,\zeta_{0}\cap [0,\infty)=[0,2],\quad 
		\mbox{supp}\,\zeta_{k}\subset \left[2^{k-1},2^{k+1}\right]\quad \mbox{for all}\ k\ge 1.
	\end{equation*}
	Finally, we use Japanese bracket to reprensent $\langle p\rangle=(1+|p|^{2})^{\frac{1}{2}}$ for $p$ being a scalar or a vector and $\|\cdot\|$ denotes the $L^2$ norm.
	We write $A \lesssim B$ to indicate $A\leq C_0 B$ with $C_0$ a universal constant. We use $[\rho]$ to represent the largest integer that is not bigger than $\rho\in\R$.
	
	%Last, for further reference, we recall the following standard Gronwall's inequality. Let $T>0$, $c\geq 0$ be a constant and $f$, $g\in C\left([0,T],\R\right)$ be nonnegative functions with the property that
	%\begin{equation*}
	%f(t)\le c+\int_{0}^{t}g(s)f(s)\d s,\quad \mbox{for all}\ t\in [0,T].
	%\end{equation*}
	%Then, we have 
	%\begin{equation}\label{est:Gron}
	%f(t)\le c\exp\left(\int_{0}^{t}g(s)\d s\right),\quad \mbox{for all}\  t\in [0,T].
	%\end{equation}

	\subsection{Estimates}
	In this subsection, we review some preliminary estimates related to vector fields and commutators. First, from the definitions of $\Gamma$ and $\widehat{{\Gamma}}$, we can easily see that the differences between them are constant matrices. Then there hold 
	\begin{equation}\label{est:hatGa}
		\sum_{|I|\le K}\big|\widehat{{\Gamma}}^{I}f\big|\lesssim \sum_{|I|\le K}\big|{\Gamma}^{I}f\big|\lesssim \sum_{|I|\le K}\big|\widehat{{\Gamma}}^{I}f\big|,
	\end{equation}
	\begin{equation}\label{est:hatparGa}
		\sum_{|I|\le K}\big|\partial\widehat{{\Gamma}}^{I}f\big|\lesssim \sum_{|I|\le K}\big|\partial{\Gamma}^{I}f\big|\lesssim \sum_{|I|\le K}\big|\partial\widehat{{\Gamma}}^{I}f\big|,
	\end{equation}
	for any smooth $\mathbb{R}$-valued or vector-valued function $f$ and $K\in \mathbb{N}^{+}$.
	
	Second, we recall some significant estimates on commutators. Denote the commutator by $[A,B]=AB-BA$. One can refer to~\cite{Sogge} for more details.
	\begin{lemma}[\cite{Sogge,Bache}]\label{lem:vectorfield}
		For any $k=1,\cdots,10$, the following estimates hold.
		\begin{itemize}
			\item[\rm (i)] $[-\Box, \Gamma_k]=0,\quad [-\Box,L_0]=-2\Box,$
			\item[\rm (ii)] $[-i\gamma^{\mu}\partial_{\mu},\widehat{\Gamma}_k]=0,\quad [-i\gamma^{\mu}\partial_{\mu}, L_0]=-i\gamma^{\mu}\partial_{\mu}.$
		\end{itemize}
	\end{lemma}

	\begin{lemma}[\cite{Sogge}]\label{lem:commutators} For any smooth function $u=u(t,x)$, we have
		\begin{equation}\label{est:commutators}
			\sum_{\alpha=0}^{3}\left|[\partial_\alpha,\Gamma^I]u\right|+\left|[L_0,\Gamma^I]u\right|\lesssim\sum_{|J|<|I|}\sum_{\beta=0}^{3}|\partial_\beta\Gamma^Ju|.
		\end{equation}
	\end{lemma}
	
	The following lemma gives us an estimate for $\partial u$, which provides an extra $\langle t-r\rangle^{-1}$ decay.
	\begin{lemma}[\cite{AlinhacBook,Sogge}]\label{lem:partial} Let $u=u(t,x)$ be any smooth function. Then we have
		\begin{equation}\label{est:partial}
			\langle t+r\rangle\sum_{a=1}^{3}|G_au|+\langle t-r\rangle|\partial u|\lesssim|L_0 u|+|\Gamma u|.
		\end{equation}
	\end{lemma}

	Next, we state the following Sobolev inequalities.
	%\begin{lemma}[Klainerman-Sobolev inequality]
	%Let $u=u(t,x)$ be a sufficiently regular function that vanishes when $|x|$ is large. Then there holds
	%\begin{equation}
	%\langle t+r\rangle \langle t-r\rangle^{\frac{1}{2}}|u(t,x)|\lesssim \sum_{|I|\leq2}\|Z^I u(t,x)\|_{L^2}.
	%\end{equation}
	%\end{lemma}
	%\begin{proof}
	%For Klainerman-Sobolev inequality~(\ref{est:Sobo}), one can refer to Theorem 1.3 in \cite{Sogge} for details of the proof.
	%\end{proof}
	\begin{lemma}\label{lem:Sobolev}
		Let $u=u(t,x)$ be a sufficiently regular function that vanishes when $|x|$ is large. Then the following estimates hold.
		\begin{enumerate}
			
			\item {\rm {Klainerman-Sobolev inequality}}. We have 
			\begin{equation}\label{est:Sobo}
				|u(t,x)|\lesssim \langle t+r\rangle^{-1} \langle t-r\rangle^{-\frac{1}{2}}\sum_{|I|\leq2}\|Z^I u(t,x)\|.
			\end{equation}
			
			\item {\rm {Standard Sobolev inequality}}. We have 
			\begin{equation}\label{est:standardSobo}
				|u(t,x)|\lesssim\langle r\rangle^{-1}\sum_{k \leq 1, |J|\leq 2}\|\partial_r^k \Omega^J u(t,x)\|,
			\end{equation}
			where $\partial_r := \sum_{a=1}^3 \omega_a \partial_{a}$ denotes the derivative with respect to the radial direction.
			
			\item{\rm{Estimate inside of a cone.}} For $r\leq\frac{t}{2}$, we have
			\begin{equation}\label{est:GeorSobo}
				|u(t,x)|\lesssim\langle t\rangle^{-\frac{3}{4}}\sum_{|I|\leq3}\|\Gamma^Iu(t,x)\|.
			\end{equation}
			\item{\rm{Global Sobolev inequality.}} We have
			\begin{equation}\label{est:globalSobo}
				|u(t,x)|\lesssim\langle t+r\rangle^{-\frac{3}{4}}\sum_{|I|\leq3}\|\Gamma^Iu(t,x)\|.
			\end{equation}
		\end{enumerate}
	\end{lemma}
	\begin{proof}
		For Klainerman-Sobolev inequality~(\ref{est:Sobo}), one can refer to Theorem 1.3 in \cite{Sogge} for details of the proof. For~(\ref{est:standardSobo}) and~(\ref{est:GeorSobo}), one can refer to~\cite[Proposition 1]{KlainWave} and~\cite[Lemma 2.4]{Geor}, respectively. The inequality~(\ref{est:globalSobo}) follows from~(\ref{est:standardSobo}) and~(\ref{est:GeorSobo}).
	\end{proof}
	To proceed, we introduce a special version of Klainerman-Sobolev inequality in \cite{KlainWave}, which can lead to a better decay in time without using the scaling vector field. The drawback is that more information on time interval $[0,2t]$ is needed so as to derive the decay at time $t$.
	\begin{lemma}[\cite{KlainWave}] \label{lem:mksdecay}
		Let $u(t,x):\R^{1+3}\to \R$ be a nice function with sufficient decay at space infinity, then we have 
		\begin{align}
			|u(t,x)| & \lesssim \langle t +|x| \rangle^{-1} \sup_{0\leq s \leq 2t,|I|\leq 3} \|\Gamma^I u(s)\|.
		\end{align}
	\end{lemma}

	We also recall the following estimates related to the null form $Q_{0}(u,v):= \partial_t u \partial_t v- \nabla u\cdot \nabla v$.
	\begin{lemma}
		For any smooth $\mathbb{C}$-valued or $\mathbb{C}^{4}$-valued functions $f$ and $g$, the following estimates hold.
		\begin{enumerate}
			\item {\rm {Estimate of $Q_{0}(f,g)$}}. We have 
			\begin{equation}\label{est:Q0fg}
				\left|Q_{0}(f,g)\right|\lesssim \langle t+r\rangle^{-1}\left(\left|L_0f\right|+\left|\Gamma f\right|\right)\left|\Gamma g\right|.
			\end{equation}

			\item {\rm{Estimate of $Q_{0}(f,g)$ in the interior region.}} We have 
			\begin{equation}\label{est:Q0inside}
				\begin{aligned}
					\left|Q_{0}(f,g)\right|
					&\lesssim \frac{\langle t-r\rangle}{\langle t+r\rangle}|\partial f||\partial g|+\langle t\rangle^{-1}|\Gamma f||\Gamma g|,\quad \mbox{for}\ r\le 3t+3.
				\end{aligned}
			\end{equation}
			
		\end{enumerate}
	\end{lemma}

	Recall that, from~\cite{HomanderBook,Sogge}, for any multi-index ${I}\in \mathbb{N}^{10}$, we also have 
	\begin{equation}\label{est:GamQ0}
		\Gamma^{I}Q_{0}(f,g)=\sum_{I_{1}+I_{2}=I}Q_{0}\left(\Gamma^{I_{1}}f,\Gamma^{I_{2}}g\right).
	\end{equation}
	
	\begin{lemma}[\cite{DW-JDE,DLMY}]\label{est:hatGfPhi}
		For any multi-index $I\in \mathbb{N}^{10}$, smooth $\mathbb{R}$-valued function $f_{1}$ and $\mathbb{C}^{4}$-valued function $f_{2}$, we have 
		\begin{align}
			\big|\widehat{\Gamma}^I(f_1f_2)  \big| & \lesssim \sum_{|I_1|+|I_2|\leq |I|}\left|\Gamma^{I_{1}}f_{1}\right|\big|\widehat{{\Gamma}}^{I_{2}}f_{2}\big|,  \\
			\big|\big[\widehat{{\Gamma}}^{I}(f_{1}f_{2})\big]_{-}\big| & \lesssim \sum_{|I_{1}|+|I_{2}|\le |I|}\left|\Gamma^{I_{1}}f_{1}\right|\big|\big[\widehat{{\Gamma}}^{I_{2}}f_{2}\big]_{-}\big|.
		\end{align}
	\end{lemma}

	Finally, we recall the Gronwall inequality for later use.
	\begin{lemma}[Gronwall inequality]\label{lem:gronwall}
		Let $ \xi(t)$ be a nonnegative function that satisfies the integral inequality
		\begin{equation*}
			\xi(t)\leq c+\int_{t_0}^{t}a(s)\xi(s)\d s,
		\end{equation*}
		where $c\geq0$ is a constant and $a(t)$ is a continuous nonnegative function for $t\geq t_0$. Then we have
		\begin{equation*}
			\xi(t)\leq c\exp\left({\int_{t_0}^{t}a(s)\d s}\right).
		\end{equation*}
	\end{lemma}
	
	\subsection{Estimates for the 3D linear wave equation}
	In this subsection, we recall some technical estimates for the 3D linear wave equation. First, we introduce a sufficient condition to guarantee the solution to the 3D homogeneous wave equation to have a positive sign.
	
	\begin{lemma}\label{lem:homo-positive}
		Let $u=u(t,x)$ be the solution to the Cauchy problem
		\begin{equation}\label{eq:homowave}
			\left\{\begin{aligned}
				-\Box u(t,x) &= 0,\\
				(u,\partial_t u)|_{t=t_0}&=(u_0,u_1).
			\end{aligned}\right.
		\end{equation}
		Then if $u_0\geq 0, u_1\geq |\nabla u_0|$, we have $u(t,x)\geq 0.$
	\end{lemma}
	\begin{proof}
		Recall that
		\begin{equation*}
			\tilde{u}(t,x)=\frac{\partial}{\partial t}\left(\frac{1}{4\pi t}\int_{|y-x|=t}\tilde{u}_0(y)\d S_y\right)+\frac{1}{4\pi t}\int_{|y-x|=t}\tilde{u}_1(y)\d S_y,
		\end{equation*}
		is the solution to the 3D linear  homogeneous wave equation
		$$\left\{\begin{aligned}
			-\Box \tilde{u}(t,x)&=0,\\
			(\tilde{u},\partial_t \tilde{u})|_{t=0}&=(\tilde{u}_0,\tilde{u}_1),
		\end{aligned}\right.$$
		where $\d S_y$ is the area element of the sphere $|y-x|=t$.
		
		Then we get
		\begin{equation*}
			\begin{aligned}
				\tilde{u}(t,x)&=\frac{\partial}{\partial t}\left(\frac{t}{4\pi}\int_{|\xi|=1}\tilde{u}_0(x+t\xi)\d \omega_\xi\right)+\frac{t}{4\pi}\int_{|\xi|=1}\tilde{u}_1(x+t\xi)\d \omega_\xi\\
				&=\frac{1}{4\pi}\int_{|\xi|=1}\tilde{u}_0(x+t\xi)\d \omega_\xi+\frac{t}{4\pi}\int_{|\xi|=1}\nabla \tilde{u}_0(x+t\xi)\cdot \xi\d \omega_\xi\\
				&+\frac{t}{4\pi}\int_{|\xi|=1}\tilde{u}_1(x+t\xi)\d \omega_\xi,
			\end{aligned}
		\end{equation*}
		where $\d \omega_\xi$ is the area element of the unit sphere $S^2$, and $\xi=(\xi_1,\xi_2,\xi_3)$.
		
		Thus, the solution to~(\ref{eq:homowave}) is
		\begin{equation*}
			\begin{aligned}
				u(t,x)&=\frac{\partial}{\partial t}\left(\frac{t-t_0}{4\pi}\int_{|\xi|=1}u_0(x+(t-t_0)\xi)\d \omega_\xi\right)
				+\frac{t-t_0}{4\pi}\int_{|\xi|=1}u_1(x+(t-t_0)\xi)\d \omega_\xi\\
				&=\frac{1}{4\pi}\int_{|\xi|=1}u_0(x+(t-t_0)\xi)\d \omega_\xi
				+\frac{t-t_0}{4\pi}\int_{|\xi|=1}\nabla u_0(x+(t-t_0)\xi)\cdot \xi\d \omega_\xi\\
				&+\frac{t-t_0}{4\pi}\int_{|\xi|=1}u_1(x+(t-t_0)\xi)\d \omega_\xi.
			\end{aligned}
		\end{equation*}
		%Then from spherical coordinate transformation, we get
		%\begin{equation*}
		%\begin{aligned}
		%u(t,x)&=\frac{\partial}{\partial t}\left(\frac{t}{4\pi}\int_{0}^{2\pi}\int_{0}^{\pi}u_0(x_1+t\sin\theta\cos\phi,x_2+t\sin\theta\sin\phi,x_3+t\cos\theta)\sin\theta\d \theta\d \phi\right)\\
		%&+\frac{t}{4\pi}\int_{0}^{2\pi}\int_{0}^{\pi}u_1(x_1+t\sin\theta\cos\phi,x_2+t\sin\theta\sin\phi,x_3+t\cos\theta)\sin\theta\d \theta\d \phi\\
		%&=\frac{1}{4\pi}\int_{0}^{2\pi}\int_{0}^{\pi}u_0(x_1+t\sin\theta\cos\phi,x_2+t\sin\theta\sin\phi,x_3+t\cos\theta)\sin\theta\d \theta\d \phi\\
		%&+\frac{t}{4\pi}\int_{0}^{2\pi}\int_{0}^{\pi}\nabla u_0\cdot(\sin\theta\cos\phi,\sin\theta\sin\phi,\cos\theta)\sin\theta\d \theta\d \phi\\
		%&+\frac{t}{4\pi}\int_{0}^{2\pi}\int_{0}^{\pi}u_1(x_1+t\sin\theta\cos\phi,x_2+t\sin\theta\sin\phi,x_3+t\cos\theta)\sin\theta\d \theta\d \phi.
		%\end{aligned}
		%\end{equation*}
		Therefore, when $u_0\geq 0, u_1\geq|\nabla u_0|$, we deduce $u(t,x)\geq 0$ from the above formula.
	\end{proof}
	
	\begin{lemma}
		Let $u=u(t,x)$ be the solution to the Cauchy problem
		$$\left\{\begin{aligned}
			-\Box u(t,x)&=0,\\
			(u,\partial_t u)|_{t=t_0}&=(u_0,u_1).
		\end{aligned}\right.$$
		Then we have the following $L^2$ estimate
		\begin{equation*}\label{est:homoL2}
			\|u(t,x)\|\lesssim \|u_0\|+\|u_1\|_{L^1}+\|u_1\|,
		\end{equation*}
		in which we recall that $\|\cdot\|$ denotes the $L^2$ norm.
	\end{lemma}
	%\begin{proof}
		%Taking the Fourier transform in the Cauchy problem with respect to $x$, we get
		%\begin{equation*}
			%\left\{\begin{aligned}
				%\partial_{tt} \hat{u}(t,\xi)+|\xi|^2\hat{u}(t,\xi)=0,\\
				%\hat{u}(t_0,\xi)=\hat{u}_0(\xi),\quad \partial_t\hat{u}(t_0,\xi)=\hat{u}_1(\xi).
			%\end{aligned}\right.
		%\end{equation*}
		
		%We solve the above second order ODE with respect to $t$ obtaining the expression of the solution $u$ in Fourier space
		%$$\hat{u}(t,\xi)=\cos((t-t_0)|\xi|)\hat{u}_0(\xi)+\frac{\sin((t-t_0)|\xi|)}{|\xi|}\hat{u}_1(\xi), $$
		%for $(t,\xi)\in [t_0,\infty)\times\mathbb{R}^3$.
		
		%From the Plancherel theorem and polar coordinate transformation, we deduce
		%\begin{equation*}
			%\begin{aligned}
				%&\|\cos((t-t_0)|\xi|)\hat{u}_0(\xi)\|_{L^2_\xi}\lesssim \|\hat{u}_0(\xi)\|_{L^2_\xi},\\
				%&\|\frac{\sin((t-t_0)|\xi|)}{|\xi|}\hat{u}_1(\xi)\|_{L^2_\xi}\\
				%&\lesssim\|\frac{\sin((t-t_0)|\xi|)}{|\xi|}\hat{u}_1(\xi)\textbf{1}_{|\xi|\geq 1}\|_{L^2_\xi}+\|\frac{\sin((t-t_0)|\xi|)}{|\xi|}\hat{u}_1(\xi)\textbf{1}_{|\xi|\leq 1}\|_{L^2_\xi}\\
				%&\lesssim\|\hat{u}_1(\xi)\|_{L^2_\xi}+\|\hat{u}_1(\xi)\|_{L^\infty_\xi}\left(\int_{|\omega|=1}\int_{0}^{1}\sin^2((t-t_0)r)\d r\d \omega\right)^{\frac{1}{2}}\\
				%&\lesssim\|\hat{u}_1(\xi)\|_{L^2_\xi}+\|\hat{u}_1(\xi)\|_{L^\infty_\xi}.
			%\end{aligned}
		%\end{equation*}
		%Using again the Plancherel theorem, we can obtain the desired result.
	%\end{proof}
	Next, we recall some energy estimates for the 3D linear wave equation. Before that, we introduce the natural energy and conformal energy for the 3D wave equation
	\begin{equation*}
		\begin{aligned}
			\mathcal{E}(t,u)&=\int_{\R^{3}}(|\partial_tu|^2+|\partial_1u|^2+|\partial_2u|^2+|\partial_3u|^2)\d x,\\
			\mathcal{E}_{con}(t,u)&=\int_{\R^{3}}(|L_0u|^2+u^2+\sum_{1\leq a<b\leq3}|\Omega_{ab}u|^2+\sum_{a=1,2,3}|L_au|^2)\d x.
		\end{aligned}
	\end{equation*}
	\begin{lemma} \label{lem:staconene}
		Let $u=u(t,x)$ be the solution to the Cauchy problem
		$$\left\{\begin{aligned}
			-\Box u(t,x)&=G(t,x),\\
			(u,\partial_t u)|_{t=t_0}&=(u_0,u_1),
		\end{aligned}\right.$$
		
		with $G(t,x)$ a sufficiently regular function. Then the following estimates hold.
		\begin{enumerate}
			
			\item\cite{Sogge} {\rm {Standard energy estimate}}. We have 
			\begin{equation}
				\begin{aligned}
					\mathcal{E}^{\frac{1}{2}}(t,u)\lesssim \mathcal{E}^\frac{1}{2}(t_0,u)+\int_{t_0}^{t}\|G(s,x)\| \d s.\label{est:eneryWave}
				\end{aligned}
			\end{equation}
			
			\item\cite{AlinhacBook} {\rm{Conformal energy estimate.}} We have
			\begin{equation}\label{est:Conformal}
				\begin{aligned}
					\mathcal{E}^{\frac{1}{2}}_{con}(t,u)\lesssim \mathcal{E}^\frac{1}{2}_{con}(t_0,u)+\int_{t_0}^{t}\|\langle s+r \rangle G(s,x)\|\,\d s. 
				\end{aligned}
			\end{equation}
			
		\end{enumerate}
		
	\end{lemma}
	
	Finally, we introduce the extra decay for Hessian of the 3D linear wave equation.
	\begin{lemma}\label{lem:Hessian}
		Suppose that $u=u(t,x)$ satisfies the wave equation
		\begin{equation*}
			-\Box u(t,x)=G(t,x).
		\end{equation*}
		Then we have
		\begin{equation}\label{est:Hessian}
			|\partial\partial u|\lesssim\frac{t}{\langle t-r\rangle}|G|+\frac{1}{\langle t-r\rangle}\sum_{|J|\leq1}|\partial\Gamma^Ju|,\quad \mbox{for}\quad r\leq 2t.\\
		\end{equation}
	\end{lemma}
	\begin{proof}
		\textbf{Case I.} Let $t\in [0,1]$. Then
		\begin{equation*}
			|\partial\partial u|\lesssim \sum_{|J|\leq1}|\partial\Gamma^Ju|\lesssim \frac{t}{\langle t-r\rangle}|G|+\frac{1}{\langle t-r\rangle}\sum_{|J|\leq1}|\partial\Gamma^Ju|,\quad \mbox{for}\quad r\leq 2t.
		\end{equation*}
		\textbf{Case II.} Let $t\in (1,\infty)$.
		According to the definition of vector fields, we can deduce
		\begin{equation*}
			-\Box u=\frac{(t+r)(t-r)}{t^2}\partial_t\partial_t u-\sum_{a=1}^{3}\left(\frac{1}{t}\partial_a L_a u-\frac{x_a}{t^2}\partial_t L_a u+\frac{x_a}{t^2}\partial_a u\right)+\frac{3}{t}\partial_t u.
		\end{equation*}
		Combining this formula with the equation, we have
		\begin{equation}\label{est:partial_tt}
			|\partial_t\partial_t u|\lesssim \frac{t^2}{\langle t+r\rangle\langle t-r\rangle}|G|+\frac{1}{\langle t-r\rangle}\sum_{|J|\leq1}|\partial\Gamma^Ju|.
		\end{equation}
		Based on $L_a=x_a\partial_t+t\partial_a$, we can obtain
		\begin{equation}\label{equ:partial_at}
			\partial_a\partial_t u=\frac{1}{t}\partial_tL_a u-\frac{1}{t}\partial_a u-\frac{x_a}{t}\partial_t\partial_t u.
		\end{equation}
		By~(\ref{est:partial_tt}) and~(\ref{equ:partial_at}), we get
		\begin{equation*}
			|\partial_a\partial_t u|\lesssim \frac{t^2}{\langle t+r\rangle\langle t-r\rangle}|G|+\frac{1}{\langle t-r\rangle}\sum_{|J|\leq1}|\partial\Gamma^Ju|,\quad \mbox{for}\quad r\leq 2t.
		\end{equation*}
		Similarly, we have
		\begin{equation*}
			\partial_a\partial_b u=\frac{1}{t}\partial_aL_b u-\frac{1}{t}\delta_{ab}\partial_t u-\frac{x_b}{t^2}\partial_tL_a u+\frac{x_b}{t^2}\partial_a u+\frac{x_ax_b}{t^2}\partial_t\partial_t u,
		\end{equation*}
		which implies
		\begin{equation*}
			|\partial_a\partial_b u|\lesssim \frac{t^2}{\langle t+r\rangle\langle t-r\rangle}|G|+\frac{1}{\langle t-r\rangle}\sum_{|J|\leq1}|\partial\Gamma^Ju|,\quad \mbox{for}\quad r\leq 2t.
		\end{equation*}
		Therefore, the proof of~(\ref{est:Hessian}) is completed.
	\end{proof}
	
	\subsection{Estimates for the 3D linear Klein-Gordon equation}
	In this subsection, we recall some estimates for the 3D linear Klein-Gordon equation, including ghost weight energy estimates and pointwise estimates.
	
	First, we introduce the ghost weight energy estimate for the 3D linear Klein-Gordon equation.
	\begin{lemma}\label{lem:ghostKG}
		Suppose that $u(t,x)$ is the solution to the Cauchy problem
		$$\left\{\begin{aligned}
			-\Box u +u&=G(t,x),\\
			(u,\partial_t u)|_{t=t_0}&=(u_0,u_1).
		\end{aligned}\right.$$
		Then for all $\delta>0$, the following estimates are true. 
		\begin{enumerate}
			\item {{\rm Ghost energy estimate.}}
			\begin{align}\label{est:ghost}
				\mathcal{E}_{gst,1}(t,u)& \lesssim_{\delta}\mathcal{E}_{gst,1}(t_0,u)+\int_{t_0}^{t}\int_{\R^{3}}|G(s,x)\partial_t u|\,\d x\d s, \\
				\mathcal{E}_{gst,1}^{\frac 12}(t,u)& \lesssim_{\delta}\mathcal{E}^{\frac 12}_{gst,1}(t_0,u)+\int_{t_0}^{t}\|G(s,x)\|\, \d s, \label{est:kgenergy}
			\end{align}
			in which
			\begin{equation*}
				\begin{aligned}
					\mathcal{E}_{gst,1}(t,u)=&\int_{\R^{3}}\left(|\partial u|^2+|u|^2\right)\d x+\int_{t_0}^{t}\int_{\R^{3}}\frac{|u|^2}{\langle s-r\rangle^{1+2\delta}}\d x\d s \\
					&+\sum_{a=1}^{3}\int_{t_0}^{t}\int_{\R^{3}}\frac{|G_au|^2}{\langle s-r\rangle^{1+2\delta}}\d x\d s.
				\end{aligned}
			\end{equation*}
			\item {\rm{Energy estimate on an exterior region.}}  It holds
			\begin{equation}\label{est:cutenergyv}
				\begin{aligned}
					&\left\|\langle r-t\rangle\chi(r-2t)u(t,x)\right\|+\left\|\langle r-t\rangle\chi(r-2t)\partial u(t,x)\right\|\\
					&\lesssim \left\|\langle r\rangle u_{0}(x)\right\|_{H_{x}^{1}}+ \| \langle r\rangle u_{1}(x) \| + \int_{0}^{t}\left\|\langle r-s\rangle\chi(r-2s)G(s,x)\right\|\d s.
				\end{aligned}
			\end{equation}
		\end{enumerate}
	\end{lemma}
	\begin{proof}
		Multiplying the equation by $e^q\partial_t u$, we can obtain
		\begin{equation} \label{equdiffer}
			\begin{aligned}
				\frac{1}{2}\partial_t&\left(e^q\left(|\partial u|^2+u^2\right)\right)-\partial^a\left(e^q\partial_t u\partial_a u\right)\\
				&+\frac{e^q}{2\langle r-t\rangle^{1+2\delta}}\left(\sum_{a=1}^{3}|G_a u|^2+u^2\right)=e^q\partial_t u G,
			\end{aligned}
		\end{equation}
		with 
		\begin{equation*}
			q(r,t)=\int_{-\infty}^{r-t}\langle s\rangle^{-1-2\delta}\d s,\quad \delta>0.
		\end{equation*}
		Then integrating with respect to $x$ and $t$, we get~(\ref{est:ghost}). Using (\ref{equdiffer}), one can also obtain \eqref{est:kgenergy}. Finally, for the proof of \eqref{est:cutenergyv}, one can refer to \cite{DLMY}.
	\end{proof}
	
	Second, we present the following pointwise estimate proved by Georgiev in~\cite[Theorem 1]{Geor}.
	\begin{proposition}\label{pro:KGpointwise}
		Suppose that $u(t,x)$ is the solution to the Cauchy problem
		$$\left\{\begin{aligned}
			-\Box u +u&=G(t,x),\\
			(u,\partial_t u)|_{t=t_0}&=(u_0,u_1).
		\end{aligned}\right.$$
		Then we have
		\begin{equation}
			\begin{aligned}
				\langle t+r\rangle^{\frac{3}{2}}|u(t,x)|&\lesssim\sum_{k=0}^{\infty}\sum_{|I|\leq 5}\|\langle r\rangle^{\frac{3}{2}}\zeta_k(r)\Gamma^I u(t_0,x)\|\\
				&+\sum_{k=0}^{\infty}\sum_{|I|\leq 4}\max_{t_0\leq s\leq t}\zeta_k(s)\|\langle s+r\rangle\Gamma^I G(s,x)\|.
			\end{aligned}
		\end{equation}
	\end{proposition}

	Next, we give a corollary of Proposition~\ref{pro:KGpointwise}.
	\begin{corollary}\cite{DLMY}\label{cor:KGpoint}
		With the same settings as Proposition~\ref{pro:KGpointwise}, let $\delta_0>0$ and assume
		\begin{equation}\label{est:C_G}
			\sum_{|I|\leq 4}\max_{t_0\leq s\leq t}s^{\delta_0}\|\langle s+r\rangle\Gamma^I G(s,x)\|\leq C_G.
		\end{equation}
		Then we have the following estimate
		\begin{equation}\label{est:pointwiseKG}
			\langle t+r\rangle^{\frac{3}{2}}|u(t,x)|\lesssim\frac{C_G}{1-2^{-\delta_0}}+\sum_{|I|\leq 5}\|\langle r\rangle^{\frac{3}{2}}\log(2+r)\Gamma^I u(t_0,x)\|.
		\end{equation}
		
	\end{corollary}
	
	Finally, we recall the extra weighted estimate for the 3D linear Klein-Gordon equation.
	\begin{lemma}\cite{Klainerman93}\label{est:weightKG}
		Suppose that $u(t,x)$ satisfies the equation
		\begin{equation*}
			-\Box u +u=G(t,x).
		\end{equation*}
		Then we have
		\begin{equation}\label{est:KG3}
			\left|\frac{\langle t+r\rangle}{\langle t-r\rangle}u\right|\lesssim\sum_{|J|\leq1}|\partial\Gamma^Ju|+\frac{\langle t+r\rangle}{\langle t-r\rangle}|G|.
		\end{equation}
	\end{lemma}
	\begin{proof}
		%From the equation, we have $|u|\lesssim|\Box u|+|G|$.
		
		\textbf{Case I.} Let $r\geq 2t.$ From the equation of $u$, we can easily calculate that
		\begin{equation}\label{KGouter}
			|u|\lesssim|\Box u|+|G|\lesssim\frac{\langle t-r\rangle}{\langle t+r\rangle}\sum_{|J|\leq1}|\partial\Gamma^Ju|+|G|.
		\end{equation}
		\textbf{Case II.} Let $r\leq 2t$. By the definition of vector fields, we have
		\begin{equation*}
			-\Box u=\frac{(t+r)(t-r)}{t^2}\partial_t\partial_t u-\sum_{a=1}^{3}\left(\frac{1}{t}\partial_a L_a u-\frac{x_a}{t^2}\partial_t L_a u+\frac{x_a}{t^2}\partial_a u\right)+\frac{3}{t}\partial_t u,
		\end{equation*}
		which implies
		\begin{equation*}
			|\Box u|\lesssim\frac{\langle t-r\rangle}{\langle t+r\rangle}\sum_{|J|\leq1}|\partial\Gamma^Ju|.
		\end{equation*}
		Then using again the equation of $u$, we can deduce
		\begin{equation}\label{KGinner}
			|u|\lesssim|\Box u|+|G|\lesssim\frac{\langle t-r\rangle}{\langle t+r\rangle}\sum_{|J|\leq1}|\partial\Gamma^Ju|+|G|.
		\end{equation}
		We see that~(\ref{est:KG3}) follows from~(\ref{KGouter}) and~(\ref{KGinner}).
	\end{proof}
	
	\subsection{Estimates for the 3D linear  Dirac equation} In this subsection, we present some estimates related to the Dirac equation, including the ghost weight energy estimates and extra decay inside the light cone. 							
	We first recall the weighted energy estimates for the Dirac equation, which are inspired by Alinhac's ghost weight method. Such estimates have been proved in \cite{DLMY}. For convenience, we denote
	\begin{align*}
		\mathcal{E}_{D}(t,\phi)=\int_{\R^3}|\phi(t,x)|^2\d x + \int_{0}^t\int_{\R^3}\frac{|[\phi]_{-}|^2}{\langle s-r \rangle^{1+2\delta} } \d x \d s
	\end{align*}
	where $\delta >0$ is a small constant.
	
	\begin{lemma}\label{le:EneryDirac}
		Let $\phi=\phi(t,x)$ be the solution to the Cauchy problem of linear Dirac equation
		\begin{equation}\label{equ:Dirac}
			-i\gamma^{\mu}\partial_{\mu}\phi=F,\quad \quad \phi (0,x)=\phi_{0}(x),
		\end{equation} 
		then the following estimates are true.
		\begin{enumerate}
			\item {\rm{Ghost weight estimate.}} It holds 
			%\begin{equation}\label{est:Energyphi1}
			%\begin{aligned}
			%\mathcal{E}^{D}(t,\phi)^{\frac{1}{2}}&\lesssim \left(\int_{\R^{3}}|\phi_{0}(x)|^{2}\d x\right)^{\frac{1}{2}}+\int_{0}^{t}\|F(s,x)\|_{L_{x}^{2}} \d s, 
			%\end{aligned}
			%\end{equation}
			\begin{equation}\label{est:Energyphi2}
				\begin{aligned}
					\mathcal{E}_{D}(t,\phi)&\lesssim \int_{\R^{3}}|\phi_{0}(x)|^{2}\d x+\int_{0}^{t}\int_{\R^{3}}\left|\phi^{*}\gamma^{0}F(s,x)-F^*\gamma^0\phi(s,x) \right|\d x \d s.
				\end{aligned}
			\end{equation}
			
			\item {\rm{Energy estimate on an exterior region.}} It holds
			\begin{equation}\label{est:cutenergyphi}
				\begin{aligned}
					&\left\|\langle r-t\rangle\chi(r-2t)\phi(t,x)\right\|^2 \\
					&\lesssim \left\|\langle r\rangle\phi_{0}(x)\right\|^2+\int_{0}^{t} \int_{\R^3} \langle r-s\rangle^2\chi(r-2s)^2|\phi^*\gamma^0 F - F^*\gamma^0 \phi|\, \d x \d s.
				\end{aligned}
			\end{equation}
		\end{enumerate}
	\end{lemma}
	\begin{proof}
		We omit the details and one can refer to \cite{DLMY} for the proof.
	\end{proof}
	
	We will use the following extra decay in $\langle t-r \rangle $ for Dirac field $\phi$ inside the cone,  which was  proved in \cite{DongZoeDKG} in the two dimensional case.
	\begin{lemma}\label{le:Diracdecay}
		Suppose $\phi=\phi(t,x)$ is the solution to the equation
		\begin{equation}\label{equ:dirac}
			-i\gamma^{\mu}\partial_{\mu}\phi(t,x)=F(t,x),
		\end{equation}
		then we have 
		\begin{equation}\label{est:decaypphi}
			\left|\partial \phi\right|\lesssim\frac{1}{\langle t-r\rangle}\big(\big|\widehat{\Gamma}\phi\big|+\left|\phi\right|\big)+\frac{t}{\langle t-r\rangle}|F|,\quad \quad  \ r\le 3t+3.
		\end{equation}
	\end{lemma}

	\begin{proof} For $t\in [0,1]$, we can see  
		\begin{equation}\label{est:pphi1}
			\left|\partial \phi\right|\lesssim \big|\widehat{{\Gamma}} \phi\big|\lesssim\frac{1}{\langle t-r\rangle}\big|\widehat{\Gamma}\phi\big|+\frac{t}{\langle t-r\rangle}|F|,\quad \quad \mbox{for}\ r\le 3t+3.
		\end{equation}
		
		Let $t\in (1,\infty)$, using $\partial_a = t^{-1}L_{a} - (x_a/t) \partial_t$, we can rewrite~\eqref{equ:dirac} as 
		\begin{equation*}
			i\left(\gamma^{0}-\frac{x_{a}}{t}\gamma^{a} \right)\partial_t\phi=-i\gamma^{a}\frac{L_{a}}{t}\phi-F.
		\end{equation*}
		Multiplying $-i\left(\gamma^{0}-\frac{x_{a}}{t}\gamma^{a} \right)$ to the above identity and  using~\eqref{equ:gamma}, we get
		\begin{equation*}
			\frac{(t-r)(t+r)}{t^{2}}\pt \phi=-\left(\gamma^{0}-\frac{x_{a}}{t}\gamma^{a} \right)\gamma^{b}\frac{L_{b}}{t}\phi+i\left(\gamma^{0}-\frac{x_{a}}{t}\gamma^{a} \right)F,
		\end{equation*}
		This yields
		\begin{equation}\label{est:pphi2}
			|\pt \phi|\lesssim \frac{1}{\langle t-r\rangle}\big|\Gamma\phi\big|+\frac{t}{\langle t-r\rangle}|F|,\quad \quad \mbox{for}\ r\le 3t+3.
		\end{equation}
		Since  $\partial_a = t^{-1}L_{a} - (x_a/t) \partial_t$, we immediately have
		\begin{equation}\label{est:pphi3}
			\left|\partial_{a} \phi\right|\lesssim \frac{1}{t}|L_{a}\phi|+|\pt \phi|\lesssim \frac{1}{\langle t-r\rangle}\big|\Gamma\phi\big|+\frac{t}{\langle t-r\rangle}|F|,\quad \quad \mbox{for}\ r\le 3t+3.
		\end{equation}
		Noticing that the difference between $\Gamma$ and $\widehat{\Gamma}$ is a constant matrix, the desired result then follows.
	\end{proof}
	
	\subsection{Hidden structures within the Dirac--Klein-Gordon system}\label{Se:Hidden}
	
	Let $(\psi, v)$ be the solution to the Cauchy problem \eqref{eq:DKG-L} with data \eqref{eq:DKG-ID}. In order to gain extra $\langle t - r \rangle $ decay for the Dirac solution $\psi$, we adopt an idea due to Bournaveas \cite{Boura00} and  introduce a wave solution $\Psi$ such that 
	\begin{equation} \label{eq:waverepre}
		\left\{
		\begin{aligned}
			-\Box \Psi & = v\psi,   \\
			\Psi|_{t=0}& =0, \ \partial_t \Psi|_{t=0} = -i \gamma^0 \psi_0,
		\end{aligned}
		\right.
	\end{equation}
	from which one can derive that 
	\begin{align} \label{eq:diffwave}
		i \gamma^{\mu} \partial_{\mu} \Psi = \psi.
	\end{align}
	
	\begin{lemma} \label{lem:ghostrepre}
		Let $\Psi$ satisfy \eqref{eq:diffwave}, then for any multi-index $I$, we have 
		\begin{align}
			[\widehat{\Gamma}^I \psi]_{-} = i(I- \omega_b \gamma^0\gamma^b ) \gamma^a G_a \widehat{\Gamma}^I \Psi,
		\end{align}
		where $\omega_b = x_b /r$ and $G_a = \partial_a + \omega_a \partial_t$ denotes the good derivative.
	\end{lemma}
	\begin{proof}
		By \eqref{eq:diffwave}, we can rewrite $\psi$ as 
		\begin{align} \label{eq:repasghost}
			\widehat{\Gamma}^I \psi = i \gamma^0\partial_t \widehat{\Gamma}^I\Psi + i \gamma^a \partial_a \widehat{\Gamma}^I\Psi = i\gamma^a G_a\widehat{\Gamma}^I \Psi + i(\gamma^0 - \omega_a \gamma^a)\partial_t \widehat{\Gamma}^I\Psi.
		\end{align}
		Recall that $[\psi]_{-} = \psi - \omega_b \gamma^0 \gamma^b \psi$, one can derive 
		\begin{align}
			[\widehat{\Gamma}^I \psi]_{-} & = (I - \omega_b \gamma^0\gamma^b) \widehat{\Gamma}^I \psi \nonumber  \\
			& = i (I - \omega_b \gamma^0\gamma^b)\gamma^a G_a\widehat{\Gamma}^I \Psi + i (I - \omega_b \gamma^0\gamma^b)(\gamma^0 - \omega_a \gamma^a)\partial_t \widehat{\Gamma}^I\Psi \nonumber\\
			& = i (I - \omega_b \gamma^0\gamma^b)\gamma^a G_a\widehat{\Gamma}^I \Psi + i(\gamma^0-\omega_a \gamma^a +\omega_b\gamma^b +\omega_a \omega_b\gamma^0\gamma^b\gamma^a) \partial_t \widehat{\Gamma}^I\Psi  \nonumber\\
			& = i (I - \omega_b \gamma^0\gamma^b)\gamma^a G_a\widehat{\Gamma}^I \Psi,
		\end{align}
		in which we use \eqref{eq:repasghost} and \eqref{equ:gamma}. The proof is done.
	\end{proof}
	
	Roughly speaking,  we have $\psi \sim \partial \Psi$, by employing the following decay estimate, one can obtain the desired   $\langle t -r \rangle$ decay.

	As a consequence of \eqref{eq:diffwave} and Lemma \ref{lem:partial}, one can see 
	\begin{align}
		|\psi(t,x)| \lesssim |\partial \Psi| \lesssim \langle t - r \rangle^{-1}  \big(|L_0 \Psi| + |\Gamma \Psi| \big).
	\end{align}
	This motivates us to study the conformal energy for the wave solution $\Psi$. However, the nonlinearity in \eqref{eq:waverepre} decays slowly, which makes it difficult to get a good bound on $L_0 \Psi$. To treat this problem, we perform a nonlinear transform, which is formulated as a lemma.
	
	\begin{lemma} \label{transwavedi}
		Let $(\psi, v)$ be the solution to the Cauchy problem \eqref{eq:DKG-L}-\eqref{eq:DKG-ID} and $\Psi$ solve the wave equation \eqref{eq:waverepre}. We denote $ \widetilde{\Psi}= \Psi + v\psi $, then one can obtain 
		\begin{align}
			-\Box \widetilde{\Psi} = (\psi^* \gamma^0 \psi) \psi + i \gamma^{\mu} v  \partial_{\mu}(v\psi) + 2Q_0(v,\psi),
		\end{align}
		where we recall that $Q_0(v,\psi) = \partial_t v \partial_t \psi - \nabla v \cdot \nabla \psi$.
	\end{lemma}
	
	At the end of this part,  we recall a hidden structure for the nonlinear term $\psi^{*}\gamma^{0}\psi$; see for instance \cite{DLMY}. 
	\begin{lemma}\label{lem:hidden}
		Let $\Phi_{1}$ and $\Phi_{2}$ be two $\C^{4}$-valued functions on $\R^{1+3}$, then we have 
		\begin{equation}\label{equ:HiddenPhi}
			\Phi_{1}^{*}\gamma^{0}\Phi_{2}=\frac{1}{4}\left([\Phi_{1}]_{-}^{*}\gamma^{0}[\Phi_{2}]_{-}+[\Phi_{1}]_{-}^{*}\gamma^{0}[\Phi_{2}]_{+}+[\Phi_{1}]_{+}^{*}\gamma^{0}[\Phi_{2}]_{-}\right).
		\end{equation}
	\end{lemma}
	
	\begin{lemma} \label{est:derinull}
		For any multi-index $I\in \mathbb{N}^{10}$, and any smooth $\mathbb{C}^{4}$-valued functions $\Phi_{1}$ and $\Phi_{2}$, we have 
		\begin{equation}\label{est:hatGG}
			\left|\Gamma^{I}\left(\Phi_{1}^{*}\gamma^{0}\Phi_{2}\right)\right|\lesssim \sum_{|I_{1}|+|I_{2}|\le |I|}\big|\big(\widehat{{\Gamma}}^{I_{1}}\Phi_{1}\big)^{*}\gamma^{0}\big(\widehat{{\Gamma}}^{I_{2}}\Phi_{2}\big)\big|.
		\end{equation}
	\end{lemma}

	\subsection{Criteria for scattering} In this subsection, we introduce some lemmas which indicate linearly scattering of solutions to wave-type equations.
	\begin{lemma}\label{lem:scatterWave}
		Suppose that $u(t,x)$ is the solution to the Cauchy problem
		$$\left\{\begin{aligned}
			-\Box u(t,x) &=G(t,x),\\
			(u,\partial_t u)|_{t=t_0}&=(u_0,u_1).
		\end{aligned}\right.$$
		Suppose that 
		\begin{equation*}
			\int_{t_0}^{+\infty}\|G(\tau,x)\|_{\dot{H}^{s-1}}\d \tau<+\infty,
		\end{equation*}
		then there exist a pair $(u_0^+,u_1^+)\in \dot{H}^s(\mathbb{R}^3)\times \dot{H}^{s-1}(\mathbb{R}^3)$ and a constant $C$ such that
		\begin{equation*}
			\|(u,\partial_t u)^T-S_{W}(t-t_0)(u_0^+,u_1^+)^T\|_{\dot{H}^s\times \dot{H}^{s-1}}\leq C\int_{t}^{+\infty}\|G(\tau,x)\|_{\dot{H}^{s-1}}\d \tau.
		\end{equation*}
		In particular,
		\begin{equation*}
			\lim_{t\rightarrow+\infty}\|(u,\partial_t u)^T-S_{W}(t-t_0)(u_0^+,u_1^+)^T\|_{\dot{H}^s\times \dot{H}^{s-1}}=0.
		\end{equation*}
		In the above, $S_{W}(t)$ is denoted by
		\begin{equation*}
			\begin{pmatrix}
				\cos(t|\nabla|) & |\nabla|^{-1}\sin(t|\nabla|)\\
				-|\nabla|\sin(t|\nabla|) & \cos(t|\nabla|)
			\end{pmatrix}.
		\end{equation*}
		Here, the operators mean that
		\begin{equation*}
			\begin{aligned}
				\left(|\nabla|^{-1}f\right)(x)&=\mathcal{F}^{-1}\left(|\xi|^{-1}\hat{f}(\xi)\right)(x),\\
				\left(\cos(t|\nabla|)f\right)(x)&=\mathcal{F}^{-1}\left(\cos(t|\xi|)\hat{f}(\xi)\right)(x),
			\end{aligned}
		\end{equation*}
		and other operators are defined similarly.
		
		In addition, $S_{W}(t-t_0)(u_0^+,u_1^+)^T$ solves
		\begin{equation*}
			-\Box u=0, \quad(u,\partial_t u)|_{t=t_0}=(u_0^+,u_1^+).
		\end{equation*}
	\end{lemma}
	
	\begin{lemma}\label{lem:scatterKG}
		Suppose that $u(t,x)$ is the solution to the Cauchy problem
		$$\left\{\begin{aligned}
			-\Box u +u&=G(t,x),\\
			(u,\partial_t u)|_{t=t_0}&=(u_0,u_1).
		\end{aligned}\right.$$
		Then if
		\begin{equation*}
			\int_{t_0}^{+\infty}\|G(\tau,x)\|_{H^{s-1}}\d \tau<+\infty,
		\end{equation*}
		there exist a pair $(u_0^+,u_1^+)\in H^s(\mathbb{R}^3)\times H^{s-1}(\mathbb{R}^3)$ and a constant $C$ such that
		\begin{equation*}
			\|(u,\partial_t u)^T-S_{KG}(t-t_0)(u_0^+,u_1^+)^T\|_{H^s\times H^{s-1}}\leq C\int_{t}^{+\infty}\|G(\tau,x)\|_{H^{s-1}}\d \tau.
		\end{equation*}
		In particular,
		\begin{equation*}
			\lim_{t\rightarrow+\infty}\|(u,\partial_t u)^T-S_{KG}(t-t_0)(u_0^+,u_1^+)^T\|_{H^s\times H^{s-1}}=0.
		\end{equation*}
		In the above, $S_{KG}(t)$ is denoted by
		\begin{equation*}
			\begin{pmatrix}
				\cos(t\langle \nabla\rangle) & \langle \nabla\rangle^{-1}\sin(t\langle\nabla\rangle)\\
				-\langle\nabla\rangle\sin(t\langle\nabla\rangle) & \cos(t\langle\nabla\rangle)
			\end{pmatrix}.
		\end{equation*}
		Here, the operators mean that
		\begin{equation*}
			\begin{aligned}
				\left(\langle\nabla\rangle^{-1}f\right)(x)&=\mathcal{F}^{-1}\left(\langle\xi\rangle^{-1}\hat{f}(\xi)\right)(x),\\
				\left(\cos(t\langle\nabla\rangle)f\right)(x)&=\mathcal{F}^{-1}\left(\cos(t\langle\xi\rangle)\hat{f}(\xi)\right)(x),
			\end{aligned}
		\end{equation*}
		and other operators are defined similarly.
		
		In addition, $S_{KG}(t-t_0)(u_0^+,u_1^+)^T$ solves
		\begin{equation*}
			-\Box u+u=0, \quad(u,\partial_t u)|_{t=t_0}=(u_0^+,u_1^+).
		\end{equation*}
	\end{lemma}
	
	\begin{lemma} \label{lem:scatdir}
		Let $s\in \mathbb{N}$ and  $\phi$ be the solution to the following  Dirac equation 
		\begin{align}
			-i \gamma^{\mu}\partial_{\mu} \phi = F, \ \ \ \phi(t_0,x) = \phi_0.
		\end{align}
		Suppose that $F$ satisfies 
		\begin{align} \label{scacond}
			\int_{t_0}^{\infty} \|F(\tau)\|_{H^s} \, \d\tau <\infty.
		\end{align}
		Then $\phi$ scatters linearly, more precisely, there exists some $\phi_+ \in H^s$, such that 
		\begin{align}
			\|\phi - e^{-i(t-t_0)\mathcal{D}}\phi_{+}\|_{H^s} \lesssim \int_t^{\infty}\|F(\tau)\|_{H^s} \d\tau,
		\end{align}
		where $\mathcal{D} = - i \gamma^0 \gamma^{a} \partial_{a}$. Additionally, $e^{-i(t-t_0)\mathcal{D}}\phi_{+}$ solves the linear Dirac equation 
		\begin{equation*}
			-i \gamma^{\mu}\partial_{\mu} \phi = 0, \ \ \phi(t_0,x)=\phi_{+}.
		\end{equation*}
	\end{lemma}

	\section{Proof of Theorem~\ref{thm:KGZ-L}}\label{S:proofTh1.3}
	
	\subsection{Bootstrap assumption}  
	From now on, we shall set the initial time $t=t_0=0$.
	Let $N\in \mathbb{N} \ \mbox{with} \ N\geq 10$. Fix $0<\delta\ll 1$. To prove Theorem~\ref{thm:KGZ-L}, we introduce the following bootstrap assumption of $E$: for $C_1\gg 1$ and $0<\epsilon\ll C_1^{-1}$ to be chosen later, 
	\begin{equation}\label{est:BootE}
		\left\{\begin{aligned}
			\mathcal{E}^{\frac{1}{2}}_{gst,1}(t,\Gamma^I\partial^{J} E)&\leq C_1^{1+(|I|+|J|)\delta}\epsilon,\quad\quad \ \ \ \mbox{for}\quad |I|\leq 10, \ |I|+|J|\leq N,\\
			\sup_{x\in \R^3}\langle t+r\rangle^{\frac{3}{2}}|\Gamma^I \partial^{J}E(t,x)|&\leq C_1^{1+(|I|+|J|+13)\delta}\epsilon,\  \ \quad\mbox{for}\quad |I|\leq 5,\ |I|+|J|\leq N-5.
		\end{aligned}\right.
	\end{equation}

	For all initial data $(\vec{E}_{0},\vec{n}_{0})$ satisfying~(\ref{initialdata}), we set
	\begin{equation*}\label{def:T}
		T^{*}(\vec{E}_{0},\vec{n}_{0})=\sup\left\{ t\in [0,+\infty): E\ \mbox{satisfies}~\eqref{est:BootE}\ \mbox{on}\ [0,t]\right\}.
	\end{equation*} 
	Note that we denote the initial data $(E_{0},E_{1},n_{0},n_{1})$ by $(\vec{E}_{0},\vec{n}_{0})$ for simplicity of notation.
	
	The following proposition, which establishes global existence of the system~(\ref{eq:KGZ-L})--(\ref{eq:KGZ-ID}), is part of Theorem~\ref{thm:KGZ-L}.
	\begin{proposition}\label{pro:KGZ-L}
		For all initial data $(\vec{E}_{0},\vec{n}_{0})$ satisfying the conditions~(\ref{initialdata}) in Theorem~\ref{thm:KGZ-L}, we have $T^{*}(\vec{E}_{0},\vec{n}_{0})=+\infty$.
	\end{proposition}

	\subsection{Key estimates}
	In this subsection, we establish some estimates on solution $(E,n^0,n^1)$ that will be used in the proof of Proposition~\ref{pro:KGZ-L}. From now on, the implied constants in $\lesssim$ do not depend on the constants $C_1$ and $\epsilon$ appearing in the bootstrap assumption~(\ref{est:BootE}).
	
	Based on Lemma~\ref{lem:vectorfield} and (\ref{eq:KGZ-L2}), for any $I\in \mathbb{N}^{10}, J\in \mathbb{N}^{4}$, we see that
	\begin{equation}\label{eq:vectorfieldE}
		-\Box\Gamma^I\partial^JE+\Gamma^I\partial^JE=-\Gamma^I\partial^J(n^0E+\Delta n^1E),
	\end{equation}
	\begin{equation}\label{eq:vectorfieldn0}
		-\Box\Gamma^I\partial^Jn^0=0,
	\end{equation}
	\begin{equation}\label{6Jeq:vectorfieldn1}
		-\Box\Gamma^I\partial^Jn^1=\Gamma^I\partial^J|E|^2.
	\end{equation}
	
	First, we introduce estimates of the solution $n^0$ to the homogeneous wave equation in~(\ref{eq:KGZ-L2}).
	\begin{lemma}[Estimates of $n^0$]\label{lem:n0}
		We have the following estimates.
		
		\begin{enumerate}
			\item {\rm {$L^2$ estimate of $\Gamma^I\partial^J n^0$}}. For $|I|\leq 10,|I|+|J|\leq N$, we have
			\begin{equation}\label{est:Gamma_n0L^2}
				\|\Gamma^I\partial^J n^0\|\lesssim K_0.
			\end{equation}
			
			\item {\rm {Pointwise decay estimate of $\Gamma^I\partial^J n^0$}}. For $|I|\leq 8,|I|+|J|\leq N-2$, we have
			\begin{equation}\label{est:Gamma_n0point}
				|\Gamma^I\partial^J n^0|\lesssim \langle t+r\rangle^{-1}\langle t-r\rangle^{-\frac{1}{2}}K_0.
			\end{equation}
			
		\end{enumerate}
	\end{lemma}
	\begin{proof}
		\textbf{Proof of (i).} Let $|I|\leq 10,|I|+|J|\leq N$.
		If $|I|=|J|=0$, from the definition of the conformal energy, we have
		\begin{equation*}
			\|n^0\|\lesssim\mathcal{E}_{con}^{\frac{1}{2}}(t,n^0)\lesssim\mathcal{E}_{con}^{\frac{1}{2}}(0,n^0).
		\end{equation*}
		If $|J|=0,1\leq |I|\leq 10$, from the definition of natural energy and conformal energy, we have
		\begin{equation*}
			\begin{aligned}
				\|\Gamma^I n^0\|&\lesssim\sum_{\substack{|K|=|I|-1}}(\mathcal{E}_{con}^{\frac{1}{2}}(t,\Gamma^K n^0)+\mathcal{E}^{\frac{1}{2}}(t,\Gamma^K n^0))\\
				&\lesssim\sum_{\substack{|K|=|I|-1}}(\mathcal{E}_{con}^{\frac{1}{2}}(0,\Gamma^K n^0)+\mathcal{E}^{\frac{1}{2}}(0,\Gamma^K n^0)).
			\end{aligned}
		\end{equation*}
		If $0\leq |I|\leq 10,|J|\geq 1,|I|+|J|\leq N,$ we have
		\begin{equation*}
			\begin{aligned}
				\|\Gamma^I\partial^J n^0\|&\lesssim\sum_{|K_1|\leq |I|,|K_2|=|J|-1}\|\partial\Gamma^{K_1}\partial^{K_2}n^0\|\\
				&\lesssim\sum_{|K_1|\leq |I|,|K_2|=|J|-1}\mathcal{E}^{\frac{1}{2}}(t,\Gamma^{K_1}\partial^{K_2}n^0)\\
				&\lesssim\sum_{|K_1|\leq |I|,|K_2|=|J|-1}\mathcal{E}^{\frac{1}{2}}(0,\Gamma^{K_1}\partial^{K_2}n^0).
			\end{aligned}
		\end{equation*}
		Then the condition in~(\ref{initialdata}) implies~(\ref{est:Gamma_n0L^2}).
		
		\textbf{Proof of (ii).} Using the Klainerman-Sobolev inequality~(\ref{est:Sobo}), for $|I|\leq 8, |I|+|J|\leq N-2$, we can deduce
		\begin{equation*}
			|\Gamma^I\partial^J n^0|
			\lesssim\langle t+r\rangle^{-1}\langle t-r\rangle^{-\frac{1}{2}}\sum_{|K|\leq2}\|Z^K\Gamma^I\partial^J n^0\|.
		\end{equation*}
		Similarly, we can get~(\ref{est:Gamma_n0point}).
	\end{proof}

	Second, we introduce some energy estimates of the solution $n^1$ to the nonhomogeneous wave equation in~(\ref{eq:KGZ-L2}), including standard energy estimates and conformal energy estimates.
	\begin{lemma}\label{lem:energyn1}%[Estimates of $n^1$] 
		Let $|I|\leq 10, |I|+|J|\leq N.$ For all $t\in [0,T^*(\vec{E}_{0},\vec{n}_{0})),$ we have the following estimates.
		\begin{enumerate}
			\item {\rm {Standard energy estimate of $\Gamma^I\partial^J n^1$}}. We have
			\begin{equation}\label{est:energy_n1_1}
				\mathcal{E}^{\frac{1}{2}}(t,\Gamma^I\partial^J n^1)\lesssim \epsilon^2 K_0^{\max\{[\frac{|I|+|J|-1}{2}],0\}}+C_1^{2+(|I|+|J|+13)\delta}\epsilon^2,
			\end{equation}
			in which $[a]$ denotes the maximal integer less than or equal to $a$.
			
			\item {\rm {Standard energy estimate of $\partial\Gamma^I\partial^J n^1$}}. We have
			\begin{equation}\label{est:energy_n1_2}
				\mathcal{E}^{\frac{1}{2}}(t,\partial\Gamma^I\partial^J n^1)\lesssim \epsilon^2 K_0^{[\frac{|I|+|J|}{2}]}+C_1^{2+(|I|+|J|+14)\delta}\epsilon^2.
			\end{equation}
			
			\item{\rm{Conformal energy estimate of $\Gamma^I\partial^J n^1$.}} We have
			\begin{equation}\label{est:conformal_n1}
				\mathcal{E}^{\frac{1}{2}}_{con}(t,\Gamma^I\partial^J n^1)\lesssim \epsilon^2 K_0^{\max\{[\frac{|I|+|J|-1}{2}],0\}}+C_1^{2+(|I|+|J|+13)\delta}\epsilon^2\langle t\rangle^{\frac{1}{2}}.
			\end{equation}
		\end{enumerate}
	\end{lemma}
	
	\begin{proof}
		\textbf{Proof of (i).}
		Using~(\ref{initialdata}), (\ref{est:eneryWave}), (\ref{est:BootE}) and (\ref{6Jeq:vectorfieldn1}), for $|I|\leq 10, |I|+|J|\leq N, N\geq 10$, we can get
		\begin{equation*}
			\begin{aligned}
				&\mathcal{E}^{\frac{1}{2}}(t,\Gamma^I\partial^J n^1)\lesssim \mathcal{E}^{\frac{1}{2}}(0,\Gamma^I\partial^J n^1)+\int_{0}^{t}\|\Gamma^I\partial^J |E|^2\|\d s\\
				&\lesssim\epsilon^2 K_0^{\max\{[\frac{|I|+|J|-1}{2}],0\}}+\sum_{\substack{I_1+I_2= I,J_1+J_2=J\\|I_1|\leq 10,|I_1|+|J_1|\leq N\\|I_2|\leq 5,|I_2|+|J_2|\leq N-5}}\int_{0}^{t}\|\Gamma^{I_1}\partial^{J_1}E\|\|\Gamma^{I_2}\partial^{J_2}E\|_{L^\infty}\d s\\
				&\lesssim\epsilon^2 K_0^{\max\{[\frac{|I|+|J|-1}{2}],0\}}+C_1^{2+(|I|+|J|+13)\delta}\epsilon^2,
			\end{aligned}
		\end{equation*}
		which means~(\ref{est:energy_n1_1}).
		
		\textbf{Proof of (ii).} Similar to (i), for $|I|\leq 10, |I|+|J|\leq N, N\geq 10$, we can deduce
		\begin{equation*}
			\begin{aligned}
				&\mathcal{E}^{\frac{1}{2}}(t,\partial\Gamma^I\partial^J n^1)
				\lesssim\mathcal{E}^{\frac{1}{2}}(0,\partial\Gamma^I\partial^J n^1)+\int_{0}^{t}\|\partial\Gamma^I\partial^J |E|^2\|\d s\\
				&\lesssim\epsilon^2 K_0^{[\frac{|I|+|J|}{2}]}+\sum_{\substack{I_1+I_2=I,J_1+J_2=J\\|I_1|\leq 10,|I_1|+|J_1|\leq N\\|I_2|\leq 5,|I_2|+|J_2|\leq N-5}}\int_{0}^{t}\|\partial\Gamma^{I_1}\partial^{J_1}E\|\|\Gamma^{I_2}\partial^{J_2}E\|_{L^\infty}\d s\\
				&+\sum_{\substack{I_1+I_2=I,J_1+J_2=J\\|I_1|\leq 5,|I_1|+|J_1|\leq N-6\\|I_2|\leq 10,|I_2|+|J_2|\leq N}}\int_{0}^{t}\|\partial\Gamma^{I_1}\partial^{J_1}E\|_{L^\infty}\|\Gamma^{I_2}\partial^{J_2}E\|\d s\\
				&\lesssim\epsilon^2 K_0^{[\frac{|I|+|J|}{2}]}+C_1^{2+(|I|+|J|+14)\delta}\epsilon^2,
			\end{aligned}
		\end{equation*}
		which implies~(\ref{est:energy_n1_2}).
		
		\textbf{Proof of (iii).} Based on~(\ref{initialdata}), (\ref{est:Conformal}), (\ref{est:BootE}) and (\ref{6Jeq:vectorfieldn1}), for $|I|\leq 10, |I|+|J|\leq N, N\geq 10$, we have
		\begin{equation*}
			\begin{aligned}
				&\mathcal{E}^{\frac{1}{2}}_{con}(t,\Gamma^I\partial^J n^1)
				\lesssim\mathcal{E}^{\frac{1}{2}}_{con}(0,\Gamma^I\partial^J n^1)+\int_{0}^{t}\|\langle s+r\rangle \Gamma^I\partial^J |E|^2\| \d s\\
				&\lesssim\epsilon^2 K_0^{\max\{[\frac{|I|+|J|-1}{2}],0\}}+\sum_{\substack{I_1+I_2=I,J_1+J_2=J\\|I_1|\leq 10,|I_1|+|J_1|\leq N\\|I_2|\leq 5,|I_2|+|J_2|\leq N-5}}\int_{0}^{t}\|\Gamma^{I_1}\partial^{J_1}E\|\|\langle s+r\rangle\Gamma^{I_2}\partial^{J_2}E\|_{L^\infty}\d s\\
				&\lesssim\epsilon^2 K_0^{\max\{[\frac{|I|+|J|-1}{2}],0\}}+C_1^{2+(|I|+|J|+13)\delta}\epsilon^2\langle t\rangle^{\frac{1}{2}},
			\end{aligned}
		\end{equation*}
		which yields~(\ref{est:conformal_n1}).
	\end{proof}
	
	Third, we also deduce  the extra decay estimates for Hessian of $\Gamma^I\partial^J n^1$ in the following lemma.
	\begin{lemma}\label{lem:Hessian_n1point}
		Let $|I|\leq 5, |I|+|J|\leq N-5.$ For all $t\in [0,T^*(\vec{E}_{0},\vec{n}_{0}))$, we have the following estimates on $\partial\partial\Gamma^I\partial^J n^1$.
		\begin{enumerate}
			\item {\rm {Let $r\leq \frac{t}{2}$}}. There holds
			\begin{equation}
				|\partial\partial\Gamma^I\partial^J n^1|\lesssim\langle t+r\rangle^{-\frac{3}{4}}\langle t-r\rangle^{-1}\left(\epsilon^2 K_0^{[\frac{|I|+|J|+3}{2}]}+C_1^{2+(|I|+|J|+26)\delta}\epsilon^2\right).
			\end{equation}
			
			\item {\rm {Let $r\geq 2t$}}. There holds
			\begin{equation}
				|\partial\partial\Gamma^I\partial^J n^1|\lesssim\langle t+r\rangle^{-\frac{3}{2}}\left(\epsilon^2 K_0^{[\frac{|I|+|J|+3}{2}]}+C_1^{2+(|I|+|J|+17)\delta}\epsilon^2\right).
			\end{equation}
			
			\item{\rm{Let $\frac{t}{2}\leq r\leq 2t$.}} There holds
			\begin{equation}
				|\partial\partial\Gamma^I\partial^J n^1|\lesssim\langle t+r\rangle^{-1}\langle t-r\rangle^{-1}\left(\epsilon^2 K_0^{[\frac{|I|+|J|+3}{2}]}+C_1^{2+(|I|+|J|+26)\delta}\epsilon^2\right).
			\end{equation}
			
			\item{\rm{}} There holds
			\begin{equation}\label{est:allHessian}
				|\partial\partial\Gamma^I\partial^J n^1|\lesssim\langle t+r\rangle^{-1}\langle t-r\rangle^{-\frac{1}{2}}\left(\epsilon^2 K_0^{[\frac{|I|+|J|+3}{2}]}+C_1^{2+(|I|+|J|+26)\delta}\epsilon^2\right).
			\end{equation}
			
		\end{enumerate}
	\end{lemma}
	\begin{proof}
		\textbf{Proof of (i).} Let $|I|\leq 5, |I|+|J|\leq N-5$ with $N\geq 10$. Since $r\leq \frac{t}{2}\leq 2t$, from Lemma~\ref{lem:Hessian} and inequalities~(\ref{est:commutators}),~(\ref{est:globalSobo}),~(\ref{est:BootE}),~(\ref{est:energy_n1_1}), we have
		\begin{equation*}
			\begin{aligned}
				&|\partial\partial\Gamma^I\partial^J n^1|\lesssim\frac{t}{\langle t-r\rangle}\left|\Gamma^I\partial^J |E|^2\right|+\frac{1}{\langle t-r\rangle}\sum_{|K|\leq 1}|\partial\Gamma^K\Gamma^I\partial^J n^1|\\
				&\lesssim\frac{t}{\langle t-r\rangle}\sum_{\substack{I_1+I_2=I\\J_1+J_2=J}}|\Gamma^{I_1}\partial^{J_1}E||\Gamma^{I_2}\partial^{J_2}E|+\langle t+r\rangle^{-\frac{3}{4}}\langle t-r\rangle^{-1}\sum_{|K|\leq 4}\|\partial\Gamma^K\Gamma^I\partial^J n^1\|\\
				%\lesssim&\langle t+r\rangle^{-2}\langle t-r\rangle^{-1}C_1^{2+(|I|+26)\delta} \epsilon^2+\langle t+r\rangle^{-\frac{3}{4}}\langle t-r\rangle^{-1}\left(\epsilon^2\langle K_0\rangle^{[\frac{|I|+3}{2}]}+C_1^{2+(|I|+q+4)\delta}\epsilon^2\right)\\
				&\lesssim\langle t+r\rangle^{-\frac{3}{4}}\langle t-r\rangle^{-1}\left(\epsilon^2 K_0^{[\frac{|I|+|J|+3}{2}]}+C_1^{2+(|I|+|J|+26)\delta}\epsilon^2\right).
			\end{aligned}
		\end{equation*}
		
		\textbf{Proof of (ii).} Let $r\geq 2t$ and $|I|\leq5, |I|+|J|\leq N-5$ with $N\geq10$. By~(\ref{est:commutators}),~(\ref{est:partial}), (\ref{est:standardSobo}), (\ref{est:energy_n1_1}), (\ref{est:conformal_n1}), we deduce
		\begin{equation*}
			\begin{aligned}
				&|\partial\partial\Gamma^I\partial^J n^1|\lesssim\frac{1}{\langle t-r\rangle}\left(\left|L_0\partial\Gamma^I\partial^J n^1\right|+\sum_{|K|=1}\left|\Gamma^K\partial\Gamma^I\partial^J n^1\right|\right) \\
				&\lesssim\langle t+r\rangle^{-2}\sum_{|K|\leq |I|+3}\left\|L_0\Gamma^K\partial\partial^J n^1\right\|\\
				&+\langle t+r\rangle^{-2}\left(\sum_{|K|\leq|I|+2}\|\partial\partial\Gamma^K\partial^J n^1\|+\sum_{|K|\leq |I|+4}\left\|\partial\Gamma^K\partial^J n^1\right\|\right)\\
				&\lesssim\langle t+r\rangle^{-\frac{3}{2}}\left(\epsilon^2 K_0^{[\frac{|I|+|J|+3}{2}]}+C_1^{2+(|I|+|J|+17)\delta}\epsilon^2\right).
			\end{aligned}
		\end{equation*}
		\textbf{Proof of (iii).} Let $\frac{t}{2}\leq r\leq 2t$ and $|I|\leq5, |I|+|J|\leq N-5$ with $N\geq 10$. Based on Lemma~\ref{lem:Hessian} and inequalities~(\ref{est:commutators}),~(\ref{est:standardSobo}),~(\ref{est:BootE}),~(\ref{est:energy_n1_1}), we have
		\begin{equation*}
			\begin{aligned}
				&|\partial\partial\Gamma^I\partial^J n^1|\lesssim\frac{t}{\langle t-r\rangle}\left|\Gamma^I\partial^J |E|^2\right|+\frac{1}{\langle t-r\rangle}\sum_{|K|\leq 1}|\partial\Gamma^K\Gamma^I\partial^J n^1|\\
				&\lesssim\frac{t}{\langle t-r\rangle}\sum_{\substack{I_1+I_2= I\\J_1+J_2=J}}|\Gamma^{I_1}\partial^{J_1}E||\Gamma^{I_2}\partial^{J_2}E|+\langle t+r\rangle^{-1}\langle t-r\rangle^{-1}\sum_{|K|\leq 4}\|\partial\Gamma^K\Gamma^I\partial^{J} n^1\|\\
				%\lesssim&\langle t+r\rangle^{-2}\langle t-r\rangle^{-1}C_1^{2+(|I|+26)\delta} \epsilon^2+\langle t+r\rangle^{-1}\langle t-r\rangle^{-1}\left(\epsilon^2\langle K_0\rangle^{[\frac{|I|+3}{2}]}+C_1^{2+(|I|+17)\delta}\epsilon^2\right)\\
				&\lesssim\langle t+r\rangle^{-1}\langle t-r\rangle^{-1}\left(\epsilon^2 K_0^{[\frac{|I|+|J|+3}{2}]}+C_1^{2+(|I|+|J|+26)\delta}\epsilon^2\right).
			\end{aligned}
		\end{equation*}
		\textbf{Proof of (iv).} When $r\leq\frac{t}{2}$ or $r\geq 2t$, $\langle t+r\rangle\sim\langle t-r\rangle$. Thus from the above results, we have~(\ref{est:allHessian}).
		
		The proof is done.
	\end{proof}
	
	%\begin{lemma}
	%For $|I|\leq N-5$, we have
	%\begin{equation}
	%\left\|\frac{\langle t+r\rangle}{\langle t-r\rangle}\Gamma^I\Delta n^1\right\|_{L^\infty}\lesssim(C_1\epsilon)^2K_0^{|I|+14}\langle t\rangle^{\frac{1}{4}}.
	%\end{equation}
	%\end{lemma}
	%\begin{proof}
	%\textbf{Case I.} Let $r\geq 2t$. Then we can deduce
	%\begin{equation}
	%\left\|\frac{\langle t+r\rangle}{\langle t-r\rangle}\Gamma^I\Delta n^1\right\|_{L^\infty}\lesssim\|\Gamma^I\Delta n^1\|_{L^\infty}\lesssim(C_1\epsilon)^2 K_0^{|I|+11}\langle t\rangle^{-\frac{3}{4}}.
	%\end{equation}
	%\textbf{Case II.} Let $r\leq 2t$, by Lemma, we have
	%\begin{equation}
	%|\partial\partial\Gamma^I n^1|\lesssim\frac{t}{\langle t-r\rangle}\left|\Gamma^I |E|^2\right|+\frac{1}{\langle t-r\rangle}\sum_{|J|\leq1}|\partial\Gamma^J\Gamma^I n^1|.
	%\end{equation}
	%Therefore,
	%\begin{equation}
	%\begin{aligned}
	%&\frac{\langle t+r\rangle}{\langle t-r\rangle}|\partial\partial\Gamma^I n^1|\\
	%\lesssim&\frac{t\langle t+r\rangle}{\langle t-r\rangle^2}\sum_{|I_1|+|I_2|\leq N-5}|\Gamma^{I_1}E||\Gamma^{I_2}E|+\frac{\langle t+r\rangle^\frac{1}{4}}{\langle t-r\rangle^2}\sum_{|J|\leq4}\|\partial\Gamma^J\Gamma^I n^1\|_{L^2}\\
	%\lesssim&(C_1\epsilon)^2K_0^{|I|+14}\langle t+r\rangle^{-1}\langle t-r\rangle^{-2}+(C_1\epsilon)^2K_0^{|I|+11}\langle t\rangle^{\frac{1}{4}}\langle t-r\rangle^{-2}\\
	%\lesssim&(C_1\epsilon)^2K_0^{|I|+14}\langle t\rangle^{\frac{1}{4}}\langle t-r\rangle^{-2}.
	%\end{aligned}
	%\end{equation}
	%\end{proof}
	
	Next, we give another pointwise estimate of $\Gamma^I\partial^J E$.
	%\begin{lemma}
	%We have
	%\begin{equation}
	%\sum_{|I|\leq1}\|\Gamma^I E\|_{L^\infty}\lesssim\langle t\rangle^{-\frac{3}{2}}\left(\epsilon\langle K_0\rangle^3+C_1^{\frac{7}{2}}\epsilon^3\langle K_0\rangle^{20}\right)+\langle t\rangle^{-\frac{3}{2}+\delta_0}\langle K_0\rangle\sum_{|I|\leq 5}\mathcal{E}^{\frac{1}{2}}_{gst,1}(t,\Gamma^I E).
	%\end{equation}
	%\end{lemma}
	%\begin{proof}
	%By Proposition~\ref{est:KGpointwise} and Corollary~\ref{est:KGpoint}, for $N\geq 10$, we can deduce
	%\begin{equation*}
	%\begin{aligned}
	%&\sum_{|I|\leq1}\|\Gamma^IE\|_{L^\infty}\\
	%\lesssim&\langle t\rangle^{-\frac{3}{2}}\left(\sum_{|J|\leq5,|I|\leq 1}\|\langle r\rangle^{\frac{3}{2}+\delta_1}\Gamma^J\Gamma^I E(0,x)\|_{L^2}+\sum_{|J|\leq4,|I|\leq 1}\max_{0\leq s\leq t}\langle s\rangle^{\delta_0}\|\langle s+r\rangle\Gamma^J\Gamma^I(n^0E+\Delta n^1E)\|_{L^2}\right)\\
	%\lesssim&\langle t\rangle^{-\frac{3}{2}}\epsilon\langle K_0\rangle^3+\langle t\rangle^{-\frac{3}{2}}\sum_{|I_1|\leq 5,|I_2|\leq 5}\max_{0\leq s\leq t}\langle s\rangle^{\delta_0}\|\langle s+r\rangle\Gamma^{I_1}n^0\|_{L^\infty}\|\Gamma^{I_2}E\|_{L^2}\\
	%&+\langle t\rangle^{-\frac{3}{2}}\sum_{|I_1|\leq 5,|I_2|\leq 5}\max_{0\leq s\leq t}\langle s\rangle^{\delta_0}\|\Gamma^{I_1}\Delta n^1\|_{L^2}\|\langle s+r\rangle\Gamma^{I_2}E\|_{L^{\infty}}\\
	%\lesssim&\langle t\rangle^{-\frac{3}{2}}\left(\epsilon\langle K_0\rangle^3+C_1^\frac{7}{2}\epsilon^3\langle K_0\rangle^{20}\right)+\langle t\rangle^{-\frac{3}{2}+\delta_0}\langle K_0\rangle\sum_{|I|\leq 5}\mathcal{E}^{\frac{1}{2}}_{gst,1}(t,\Gamma^IE).
	%\end{aligned}
	%\end{equation*}
	%\end{proof}
	
	\begin{lemma}\label{lem:PointE2}
		For all $t\in [0,T^*(\vec{E}_{0},\vec{n}_{0}))$ and $|I|\leq 5, |I|+|J|\leq N-5$, we have
		\begin{equation}\label{est:PointE2}
			|\Gamma^I\partial^J E|\lesssim \langle t+r\rangle^{-\frac{9}{8}}C_1^{1+(|I|+|J|+8)\delta}\epsilon.
		\end{equation}
	\end{lemma}
	\begin{proof}
		For $|I|\leq 5, |I|+|J|\leq N-5$, from~(\ref{est:globalSobo}) and~(\ref{est:BootE}), we get
		\begin{equation*}
			\begin{aligned}
				|\Gamma^I\partial^J E|\lesssim&\langle t+r\rangle^{-\frac{3}{4}}C_1^{1+(|I|+|J|+3)\delta}\epsilon,\\
				|\Gamma^I\partial^J E|\lesssim&\langle t+r\rangle^{-\frac{3}{2}}C_1^{1+(|I|+|J|+13)\delta}\epsilon.
			\end{aligned}
		\end{equation*}
		By interpolating the above inequalities, we obtain~(\ref{est:PointE2}).
	\end{proof}
	
	Then, we deduce the lower-order pointwise estimate of $\Gamma^I\partial^J E$.
	\begin{lemma}\label{lem:Elower}
		%For all $t\in [t_0,T^*(\vec{E}_{0},\vec{n}_{0}))$ and $|I|+|J|\leq 1$, we have
		%\begin{equation}\label{est:Elower}
		%\begin{aligned}
		%\langle t+r\rangle^{\frac{3}{2}-\delta_2}|\Gamma^I\partial^J E|
		%\lesssim\epsilon\langle K_0\rangle^{[\frac{|I|+|J|+5}{2}]}+\frac{C_E}{1-2^{-\delta_2}},
		%\end{aligned}
		%\end{equation}
		%with $C_E=C_1^{1+(|I|+|J|+4)\delta}\epsilon\langle K_0\rangle
		%+C_1^{1+(|I|+|J|+17)\delta}\epsilon^3\langle K_0\rangle^{[\frac{|I|+|J|+4}{2}]}+C_1^{3+(|I|+|J|+31)\delta}\epsilon^3$, and $\delta_2>0$ an arbitrarily small number.
		
		For all $t\in [0,T^*(\vec{E}_{0},\vec{n}_{0}))$, we have
		\begin{equation}\label{est:Elower}
			\begin{aligned}
				\sum_{|I|+|J|\leq 1}\langle t+r\rangle^{\frac{5}{4}}|\Gamma^I\partial^J E|
				\lesssim\epsilon K_0^{3}+C_1^{1+5\delta}\epsilon K_0,
			\end{aligned}
		\end{equation}
		where $C_1^{2+27\delta}\epsilon^2<1$.
	\end{lemma}
	\begin{proof}
		First, for $|I|+|J|\leq 1$, based on the H\"{o}lder inequality, estimates in (\ref{est:BootE}), (\ref{est:Gamma_n0point}), we have
		\begin{equation}
			\begin{aligned}
				&\sum_{|K|\leq 4}\|\langle t+r\rangle\Gamma^K\Gamma^I\partial^J(n^0 E)\|
				\lesssim\sum_{|K|\leq |I|+4}\|\langle t+r\rangle\Gamma^K\partial^J(n^0E)\|\\
				&\lesssim\sum_{\substack{|K|\leq |I|+4\\K_1+K_2=K\\J_1+J_2=J}}\|\langle t+r\rangle\Gamma^{K_1}\partial^{J_1}n^0\|_{L^\infty}\|\Gamma^{K_2}\partial^{J_2}E\|\\
				&\lesssim C_1^{1+(|I|+|J|+4)\delta}\epsilon K_0\lesssim C_1^{1+5\delta}\epsilon K_0,
			\end{aligned}
		\end{equation}
		where we use the fact $|I|+|J|\leq 1.$
		
		Based on the H\"{o}lder inequality, estimates in (\ref{est:commutators}), (\ref{est:BootE}), (\ref{est:Gamma_n0point}) and~(\ref{est:energy_n1_2}), we have
		\begin{equation}
			\begin{aligned}
				&\sum_{|K|\leq 4}\|\langle t+r\rangle\Gamma^K\Gamma^I\partial^J(\Delta n^1 E)\|
				\lesssim\sum_{|K|\leq |I|+4}\|\langle t+r\rangle\Gamma^K\partial^J(\Delta n^1 E)\|\\
				&\lesssim\sum_{\substack{|K|\leq |I|+4\\K_1+K_2=K\\J_1+J_2=J}}\|\Gamma^{K_1}\partial^{J_1}\Delta n^1\|\|\langle t+r\rangle\Gamma^{K_2}\partial^{J_2}E\|_{L^\infty}\\
				&\lesssim\sum_{\substack{|K|\leq |I|+4\\|K_1|+|K_2|\leq |K|\\J_1+J_2=J}}\|\partial\partial\Gamma^{K_1}\partial^{J_1}n^1\|\|\langle t+r\rangle\Gamma^{K_2}\partial^{J_2}E\|_{L^\infty}\\
				&\lesssim C_1^{1+(|I|+|J|+17)\delta}\epsilon^3 K_0^{[\frac{|I|+|J|+4}{2}]}+C_1^{3+(|I|+|J|+31)\delta}\epsilon^3\\
				&\lesssim C_1^{1+18\delta}\epsilon^3 K_0^2+C_1^{1+5\delta}\epsilon\lesssim\epsilon K_0^2+C_1^{1+5\delta}\epsilon,
			\end{aligned}
		\end{equation}
		where we use the fact $|I|+|J|\leq 1$ and choose $C_1$ and $\epsilon$ such that $C_1^{2+27\delta}\epsilon^2<1.$
		%\begin{equation}
		%\begin{aligned}
		%&\sum_{|K|\leq 4}\|\langle t+r\rangle\Gamma^K\Gamma^I\partial^J(n^0 E+\Delta n^1 E)\|_{L^2}\\
		%&\lesssim\sum_{|K|\leq |I|+4}\|\langle t+r\rangle\Gamma^K\partial^J(n^0E+\Delta n^1 E)\|_{L^2}\\
		%&\lesssim\sum_{\substack{|K|\leq |I|+4\\K_1+K_2=K\\J_1+J_2=J}}\left(\|\langle t+r\rangle\Gamma^{K_1}\partial^{J_1}n^0\|_{L^\infty}\|\Gamma^{K_2}\partial^{J_2}E\|_{L^2}\right)\\
		%&+\sum_{\substack{|K|\leq |I|+4\\K_1+K_2=K\\J_1+J_2=J}}\|\Gamma^{K_1}\partial^{J_1}\Delta n^1\|_{L^2}\|\langle t+r\rangle\Gamma^{K_2}\partial^{J_2}E\|_{L^\infty}\\
		%&\lesssim\sum_{\substack{|K|\leq |I|+4\\|K_1|+|K_2|\leq |K|\\J_1+J_2=J}}\left(\|\langle t+r\rangle\Gamma^{K_1}\partial^{J_1}n^0\|_{L^\infty}\|\Gamma^{K_2}\partial^{J_2}E\|_{L^2}+\|\partial\partial\Gamma^{K_1}\partial^{J_1}n^1\|_{L^2}\|\langle t+r\rangle\Gamma^{K_2}\partial^{J_2}E\|_{L^\infty}\right)\\
		%&\lesssim C_1^{1+(|I|+|J|+4)\delta}\epsilon\langle K_0\rangle+C_1^{1+(|I|+|J|+17)\delta}\epsilon^3\langle K_0\rangle^{[\frac{|I|+|J|+4}{2}]}+C_1^{3+(|I|+|J|+31)\delta}\epsilon^3=C_E.
		%\end{aligned}
		%\end{equation}
		
		Recall $N\geq10$ and from Proposition~\ref{pro:KGpointwise}, Corollary~\ref{cor:KGpoint} and~(\ref{initialdata}), for $|I|+|J|\leq 1$, we have
		\begin{equation*}
			\begin{aligned}
				\langle t+r\rangle^{\frac{3}{2}}|\Gamma^I\partial^J E|&\lesssim \sum_{k=0}^{\infty}\sum_{|K|\leq 5}\left\|\langle r\rangle^{\frac{3}{2}}\zeta_{k}(r)\Gamma^K\Gamma^I\partial^J E(0,x)\right\|\\
				&+\sum_{k=0}^{\infty}\sum_{|K|\leq 4}\max_{0\leq s\leq t}\zeta_{k}(s)\|\langle s+r\rangle\Gamma^K\Gamma^I\partial^J (n^0E+\Delta n^1E)\|\\
				&\lesssim\sum_{|K|\leq 5}\left\|\langle r\rangle^{2}\Gamma^K\Gamma^I\partial^J E(0,x)\right\|+C_E\sum_{k=0}^{\infty}\max_{0\leq s\leq t}\frac{\zeta_{k}(s)}{\langle s\rangle^{\frac{1}{4}}}\langle s\rangle^{\frac{1}{4}}\\
				&\lesssim\epsilon K_0^{3}+\frac{C_E}{1-2^{-\frac{1}{4}}}\langle t\rangle^{\frac{1}{4}},
			\end{aligned}
		\end{equation*}
		where $C_E=C_1^{1+5\delta}\epsilon K_0+\epsilon K_0^2.$
		
		The proof of~(\ref{est:Elower}) is done.
	\end{proof}
	
	Also, we give the following weighted $L^2$ estimate for $\Gamma^I\partial^J E$.
	\begin{lemma}\label{lem:weightE}
		For all $t\in [0,T^*(\vec{E}_{0},\vec{n}_{0}))$, $|I|\leq 9, |I|+|J|\leq N-1$, we have %the following weighted estimate
		\begin{equation*}
			\begin{aligned}
				\left\|\frac{\langle t+r\rangle}{\langle t-r\rangle}\Gamma^I\partial^J E\right\|
				%&\lesssim C_1^{1+(|I|+|J|+8)\delta}\epsilon\langle K_0\rangle+C_1^{1+(|I|+|J|+13)\delta}\epsilon^3\langle K_0\rangle^{[\frac{|I|+|J|+3}{2}]}+C_1^{3+(|I|+|J|+27)\delta}\epsilon^3\\
				\lesssim C_1^{1+(|I|+|J|+8)\delta}\epsilon K_0+C_1^{1+(|I|+|J|+13)\delta}\epsilon^3 K_0^{[\frac{|I|+|J|+3}{2}]},
			\end{aligned}
		\end{equation*}
		where $C_1^{2+19\delta}\epsilon^2<1.$
	\end{lemma}
	\begin{proof}
		Let $|I|\leq 9, |I|+|J|\leq N-1$, with $N\geq 10$.
		First, from Lemma~\ref{lem:n0}, inequalities~(\ref{est:BootE}),~(\ref{est:PointE2}) and the H\"{o}lder inequality, we have
		\begin{equation}
			\begin{aligned}
				\left\|\frac{\langle t+r\rangle}{\langle t-r\rangle}\Gamma^I\partial^J (n^0 E)\right\|&\lesssim\sum_{\substack{I_1+I_2=I,J_1+J_2=J\\|I_1|\leq 9,|I_1|+|J_1|\leq N-1\\|I_2|\leq5,|I_2|+|J_2|\leq N-5}}\|\Gamma^{I_1}\partial^{J_1} n^0\|\left\|\frac{\langle t+r\rangle}{\langle t-r\rangle}\Gamma^{I_2}\partial^{J_2}E\right\|_{L^\infty}\\
				&+\sum_{\substack{I_1+I_2=I,J_1+J_2=J\\|I_1|\leq8,|I_1|+|J_1|\leq N-2\\|I_2|\leq 9,|I_2|+|J_2|\leq N-1}}\left\|\frac{\langle t+r\rangle}{\langle t-r\rangle}\Gamma^{I_1}\partial^{J_1}n^0\right\|_{L^\infty}\|\Gamma^{I_2}\partial^{J_2}E\|\\
				&\lesssim C_1^{1+(|I|+|J|+8)\delta}\epsilon K_0.
			\end{aligned}
		\end{equation}
		Second, based on Lemma~\ref{lem:Hessian_n1point}, inequalities~(\ref{est:commutators}),~(\ref{est:BootE}),~(\ref{est:energy_n1_2}) and the H\"{o}lder inequality, we get
		\begin{equation}
			\begin{aligned}
				&\left\|\frac{\langle t+r\rangle}{\langle t-r\rangle}\Gamma^I\partial^J (\Delta n^1E)\right\|\\
				&\lesssim\sum_{\substack{I_1+I_2=I,J_1+J_2=J\\|I_1|\leq 9,|I_1|+|J_1|\leq N-1\\|I_2|\leq 5,|I_2|+|J_2|\leq N-5}}\|\Gamma^{I_1}\partial^{J_1}\Delta n^1\|\left\|\frac{\langle t+r\rangle}{\langle t-r\rangle}\Gamma^{I_2}\partial^{J_2}E\right\|_{L^\infty}\\
				&+\sum_{\substack{I_1+I_2=I,J_1+J_2=J\\|I_1|\leq 5,|I_1|+|J_1|\leq N-5\\|I_2|\leq 9,|I_2|+|J_2|\leq N-1}}\left\|\frac{\langle t+r\rangle}{\langle t-r\rangle}\Gamma^{I_1}\partial^{J_1}\Delta n^1\right\|_{L^\infty}\|\Gamma^{I_2}\partial^{J_2}E\|\\
				&\lesssim\sum_{\substack{|I_1|+|I_2|\leq|I|\\|I_1|+|J_1|\leq N-1\\|I_2|\leq 5,|I_2|+|J_2|\leq N-5}}\|\partial\partial\Gamma^{I_1}\partial^{J_1}n^1\|\left\|\frac{\langle t+r\rangle}{\langle t-r\rangle}\Gamma^{I_2}\partial^{J_2}E\right\|_{L^\infty}\\
				&+\sum_{\substack{|I_1|+|I_2|\leq|I|\\|I_1|\leq 5,|I_1|+|J_1|\leq N-5\\|I_2|+|J_2|\leq N-1}}\left\|\frac{\langle t+r\rangle}{\langle t-r\rangle}\partial\partial\Gamma^{I_1}\partial^{J_1}n^1\right\|_{L^\infty}\|\Gamma^{I_2}\partial^{J_2}E\|\\
				&\lesssim C_1^{1+(|I|+|J|+13)\delta}\epsilon^3 K_0^{[\frac{|I|+|J|+3}{2}]}+C_1^{3+(|I|+|J|+27)\delta}\epsilon^3\\
				&\lesssim C_1^{1+(|I|+|J|+13)\delta}\epsilon^3 K_0^{[\frac{|I|+|J|+3}{2}]}+C_1^{1+(|I|+|J|+8)\delta}\epsilon,
			\end{aligned}
		\end{equation}
		where we use $C_1^{2+19\delta}\epsilon^2<1.$
		
		Last, from Lemma~\ref{est:weightKG},~(\ref{est:BootE}) and the above inequalities, we obtain
		\begin{equation*}
			\begin{aligned}
				\left\|\frac{\langle t+r\rangle}{\langle t-r\rangle}\Gamma^I\partial^J E\right\|
				&\lesssim\sum_{|K|\leq 1}\|\partial\Gamma^K\Gamma^I\partial^J E\|+\left\|\frac{\langle t+r\rangle}{\langle t-r\rangle}\Gamma^I\partial^J (n^0 E+\Delta n^1E)\right\|\\
				&\lesssim C_1^{1+(|I|+|J|+8)\delta}\epsilon K_0+C_1^{1+(|I|+|J|+13)\delta}\epsilon^3 K_0^{[\frac{|I|+|J|+3}{2}]}.
			\end{aligned}
		\end{equation*}
		Thus the proof is completed.
	\end{proof}

	Finally, we deduce the following weighted $L^2$ estimate of the nonlinear term in~(\ref{eq:vectorfieldE}).
	\begin{lemma}\label{lem:nonlinear}
		For all $t\in [0,T^*(\vec{E}_{0},\vec{n}_{0}))$, $|I|\leq9, |I|+|J|\leq N-1$, we have
		\begin{equation}\label{est:nonlinear}
			\begin{aligned}
				&\|\langle t+r\rangle\Gamma^I\partial^J(n^0E+\Delta n^1E)\|\\
				%&\lesssim\langle t\rangle^{-\frac{1}{8}}\left(C_1^{1+(|I|+|J|+8)\delta}\epsilon\langle K_0\rangle^2+C_1^{3+(|I|+|J|+39)\delta}\epsilon^3\langle K_0\rangle^{[\frac{|I|+|J|+6}{2}]}+C_1^{5+(|I|+|J|+53)\delta}\epsilon^5\right)\\
				&\lesssim\langle t\rangle^{-\frac{1}{8}}\left(C_1^{1+(|I|+|J|+8)\delta}\epsilon K_0^2+C_1^{3+(|I|+|J|+39)\delta}\epsilon^3 K_0^{[\frac{|I|+|J|+6}{2}]}\right),
			\end{aligned}
		\end{equation}
		where $C_1^{2+19\delta}\epsilon^2<1.$
	\end{lemma}
	\begin{proof}
		Let $|I|\leq9, |I|+|J|\leq N-1$ with $N\geq10$.
		First, from Lemmas~\ref{lem:n0},~\ref{lem:PointE2},~\ref{lem:weightE} and the H\"{o}lder inequality, we have
		\begin{equation*}
			\begin{aligned}
				&\|\langle t+r\rangle\Gamma^I\partial^J(n^0E)\|\\
				\lesssim&\sum_{\substack{I_1+I_2=I,J_1+J_2=J\\|I_1|\leq 9,|I_1|+|J_1|\leq N-1\\|I_2|\leq 5,|I_2|+|J_2|\leq N-5}}\|\Gamma^{I_1}\partial^{J_1}n^0\|\|\langle t+r\rangle\Gamma^{I_2}\partial^{J_2}E\|_{L^\infty}\\
				+&\sum_{\substack{I_1+I_2=I,J_1+J_2=J\\|I_1|\leq 8,|I_1|+|J_1|\leq N-2\\|I_2|\leq 9,|I_2|+|J_2|\leq N-1}}\|\langle t-r\rangle\Gamma^{I_1}\partial^{J_1}n^0\|_{L^\infty}\left\|\frac{\langle t+r\rangle}{\langle t-r\rangle}\Gamma^{I_2}\partial^{J_2}E\right\|\\
				%\lesssim&C_1^{1+(|I|+8)\delta}\epsilon\langle K_0\rangle\langle t\rangle^{-\frac{1}{8}}\\
				%&+\langle t\rangle^{-\frac{1}{2}}\langle K_0\rangle\left(C_1^{1+(|I|+8)\delta}\epsilon\langle K_0\rangle+C_1^{1+(|I|+13)\delta}\epsilon^3\langle K_0\rangle^{[\frac{|I|+3}{2}]}+C_1^{3+(|I|+27)\delta}\epsilon^3\right)\\
				\lesssim&\langle t\rangle^{-\frac{1}{8}}\left(C_1^{1+(|I|+|J|+8)\delta}\epsilon K_0^2+C_1^{1+(|I|+|J|+13)\delta}\epsilon^3 K_0^{[\frac{|I|+|J|+5}{2}]}\right),
			\end{aligned}
		\end{equation*}
		where we use $C_1^{2+19\delta}\epsilon^2<1.$
		
		Second, from Lemmas~\ref{lem:Hessian_n1point}, \ref{lem:weightE}, inequalities~(\ref{est:commutators}),~(\ref{est:BootE}),~(\ref{est:energy_n1_2}) and the H\"{o}lder inequality, we obtain
		\begin{equation*}
			\begin{aligned}
				&\|\langle t+r\rangle\Gamma^I\partial^J(\Delta n^1E)\|\\
				\lesssim&\sum_{\substack{I_1+I_2=I,J_1+J_2=J\\|I_1|\leq 9,|I_1|+|J_1|\leq N-1\\|I_2|\leq 5,|I_2|+|J_2|\leq N-5}}\|\Gamma^{I_1}\partial^{J_1}\Delta n^1\|\|\langle t+r\rangle\Gamma^{I_2}\partial^{J_2}E\|_{L^\infty}\\
				+&\sum_{\substack{I_1+I_2=I,J_1+J_2=J\\|I_1|\leq 5,|I_1|+|J_1|\leq N-5\\|I_2|\leq 9,|I_2|+|J_2|\leq N-1}}\|\langle t-r\rangle\Gamma^{I_1}\partial^{J_1}\Delta n^1\|_{L^\infty}\left\|\frac{\langle t+r\rangle}{\langle t-r\rangle}\Gamma^{I_2}\partial^{J_2}E\right\|\\
				\lesssim&\sum_{\substack{|I_1|+|I_2|\leq|I|,J_1+J_2=J\\|I_2|\leq 5,|I_2|+|J_2|\leq N-5}}\|\partial\partial\Gamma^{I_1}\partial^{J_1}n^1\|\|\langle t+r\rangle\Gamma^{I_2}\partial^{J_2}E\|_{L^\infty}\\
				+&\sum_{\substack{|I_1|+|I_2|\leq|I|,J_1+J_2=J\\|I_1|\leq 5,|I_1|+|J_1|\leq N-5}}\|\langle t-r\rangle\partial\partial\Gamma^{I_1}\partial^{J_1}n^1\|_{L^\infty}\left\|\frac{\langle t+r\rangle}{\langle t-r\rangle}\Gamma^{I_2}\partial^{J_2}E\right\|\\
				%\lesssim&\left(\epsilon^2\langle K_0\rangle^{[\frac{|I_1|}{2}]}+C_1^{2+(|I_1|+14)\delta}\epsilon^2\right)C_1^{1+(|I_2|+13)\delta}\epsilon\langle t\rangle^{-\frac{1}{2}}\\
				%&+\langle t\rangle^{-\frac{1}{2}}\left(\epsilon^2\langle K_0\rangle^{[\frac{|I_1|+3}{2}]}+C_1^{2+(|I_1|+26)\delta}\epsilon^2\right)\\
				%&\left(C_1^{1+(|I_2|+8)\delta}\epsilon\langle K_0\rangle+C_1^{1+(|I_2|+13)\delta}\epsilon^3\langle K_0\rangle^{[\frac{|I_2|+3}{2}]}+C_1^{3+(|I_2|+27)\delta}\epsilon^3\right)\\
				\lesssim&\langle t\rangle^{-\frac{1}{2}}C_1^{3+(|I|+|J|+39)\delta}\epsilon^3 K_0^{[\frac{|I|+|J|+6}{2}]},
			\end{aligned}
		\end{equation*}
		where we use $C_1^{2+19\delta}\epsilon^2<1.$
		
		We see that~(\ref{est:nonlinear}) follows from the above two inequalities.
	\end{proof}

	\subsection{Proof of Proposition~\ref{pro:KGZ-L}}
	In this subsection, we will complete the proof of Proposition~\ref{pro:KGZ-L} by improving all the estimates of $E$ in~\eqref{est:BootE}.
	\begin{proof}[Proof of Proposition~\ref{pro:KGZ-L}]
		For any initial data $(\vec{E}_{0},\vec{n}_{0})$ satisfying the conditions~(\ref{initialdata}), we consider the corresponding solution $(E,n^0,n^1)$ of~(\ref{eq:KGZ-L2}). From the conditions~(\ref{initialdata}), we observe that 
		\begin{equation}\label{est:I.D.1}
			\begin{aligned}
				\mathcal{E}^{\frac{1}{2}}_{gst,1}(0,\Gamma^I\partial^J E)&\lesssim\epsilon  K_0^{[\frac{|I|+|J|+1}{2}]},\quad\mbox{for}\quad |I|\leq10, |I|+|J|\leq N,\\
				\|\langle r\rangle^{2}\Gamma^I\partial^J E(0,x)\|&\lesssim \epsilon K_0^{[\frac{|I|+|J|}{2}]},\quad \ \ \ \mbox{for}\quad |I|\leq 10, |I|+|J|\leq N.
			\end{aligned}
		\end{equation}
		\textbf{Step 1.} Closing the estimate in $\mathcal{E}_{gst,1}(t,\Gamma^{I}\partial^{J}{E})$.
		%First, for $|I|=|J|=0$, by the H\"{o}lder inequality, Lemma~\ref{lem:ghostKG} and~(\ref{initialdata}), we get
		%\begin{equation}
		%\begin{aligned}
		%\mathcal{E}_{gst,1}(t,E)
		%&\lesssim\mathcal{E}_{gst,1}(t_0, E)+\int_{t_0}^{t}\int_{\R^{3}}\left|(n^0E+\Delta n^1 E)\partial_t E\right|\d x\d s\\
		%&\lesssim\epsilon^2+
		%\end{aligned}
		%\end{equation}
		For $|I|\leq10, |I|+|J|\leq N$, $N\geq 10$, using the H\"{o}lder inequality, Lemma~\ref{lem:ghostKG} and~(\ref{est:I.D.1}), we have
		\begin{equation}\label{est:close-energy}
			\begin{aligned}
				\mathcal{E}_{gst,1}(t,\Gamma^I\partial^J E)
				&\lesssim\mathcal{E}_{gst,1}(0,\Gamma^I\partial^{J} E)+\int_{0}^{t}\int_{\R^{3}}\left|\Gamma^I\partial^J (n^0E+\Delta n^1 E)\partial_t\Gamma^I\partial^J E\right|\d x\d s\\
				&\lesssim\epsilon^2  K_0^{|I|+|J|+1}+S_1+S_2,
			\end{aligned}
		\end{equation}
		where
		\begin{equation*}
			\begin{aligned}
				S_1=&\int_{0}^{t}\|\Gamma^I\partial^J(n^0E)\|\|\partial_t\Gamma^I\partial^J E\| \d s,\\
				S_2=&\int_{0}^{t}\|\Gamma^I\partial^J(\Delta n^1E)\|\|\partial_t\Gamma^I\partial^J E\| \d s.
			\end{aligned}
		\end{equation*}
		To estimate $S_1$, we consider four scenarios. 
		First, if $|I|=|J|=0$, based on~(\ref{est:BootE}),~(\ref{est:Gamma_n0point}), H\"{o}lder inequality and Cauchy inequality, we deduce
		\begin{equation}
			\begin{aligned}
				S_1&\lesssim\int_{0}^{t}\|\langle s-r\rangle^{\frac{1}{2}+\delta}n^0\|_{L^\infty}\left\|\frac{E}{\langle s-r\rangle^{\frac{1}{2}+\delta}}\right\|\|\partial_t E\| \d s\\
				&\lesssim C_2^{-1}\mathcal{E}_{gst,1}(t, E)+C_2\int_{0}^{t}\langle s\rangle^{-2+2\delta} K_0^2\mathcal{E}_{gst,1}(s,E)\d s,
			\end{aligned}
		\end{equation}
		in which $C_2$ is a constant to be chosen later.
		
		Second, if $|I|\leq 8,1\leq |I|+|J|\leq N-2$, 
		\begin{equation}
			\begin{aligned}
				S_1
				&\lesssim\sum_{\substack{I_1+I_2=I\\J_1+J_2=J}}\int_{0}^{t}\|\langle s-r\rangle^{\frac{1}{2}+\delta}\Gamma^{I_1}\partial^{J_1}n^0\|_{L^\infty}\left\|\frac{\Gamma^{I_2}\partial^{J_2}E}{\langle s-r\rangle^{\frac{1}{2}+\delta}}\right\|\|\partial_t\Gamma^I\partial^J E\|_{L^2} \d s\\
				&\lesssim
				\sum_{\substack{|I_2|+|J_2|\leq |I|+|J|}}C_2^{-1}\int_{0}^{t}\left\|\frac{\Gamma^{I_2}\partial^{J_2} E}{\langle s-r\rangle^{\frac{1}{2}+\delta}}\right\|^2\d s\\
				&+\sum_{\substack{|I_1|+|J_1|\leq |I|+|J|}}C_2\int_{0}^{t}\|\langle s-r\rangle^{\frac{1}{2}+\delta}\Gamma^{I_1}\partial^{J_1}n^0\|^2_{L^\infty}\|\partial_t\Gamma^I\partial^J E\|^2\d s\\
				&\lesssim C_2^{-1}\mathcal{E}_{gst,1}(t,\Gamma^I\partial^J E)+C_2^{-1}C_1^{2+2(|I|+|J|-1)\delta}\epsilon^2\\
				&+C_2\int_{0}^{t}\langle s\rangle^{-2+2\delta} K_0^2\mathcal{E}_{gst,1}(s,\Gamma^I\partial^J E)\d s,
			\end{aligned}
		\end{equation}
		in which $C_2$ is a constant to be chosen later.
		
		Next, if $|I|=9$ or $|I|=10$, and $|I|\leq |I|+|J|\leq \max\{|I|,N-2\}$, we obtain
		\begin{equation}
			\begin{aligned}
				&S_1
				\lesssim\sum_{\substack{|I_1|\leq|I|\\|I_1|+|J_1|\leq \max\{|I|,N-2\}\\|I_2|+|J_2|\leq 1}}\int_{0}^{t}\|\Gamma^{I_1}\partial^{J_1}n^0\|\|\Gamma^{I_2}\partial^{J_2}E\|_{L^\infty}\|\partial_t\Gamma^I\partial^J E\|\d s\\
				&+\sum_{\substack{I_1+I_2=I,J_1+J_2=J\\|I_1|\leq 8,|I_1|+|J_1|\leq N-2\\|I_2|\leq |I|\\|I_2|+|J_2|\leq \max\{|I|,N-2\}}}\int_{0}^{t}\|\langle s-r\rangle^{\frac{1}{2}+\delta}\Gamma^{I_1}\partial^{J_1}n^0\|_{L^\infty}\left\|\frac{\Gamma^{I_2}\partial^{J_2}E}{\langle s-r\rangle^{\frac{1}{2}+\delta}}\right\|\|\partial_t\Gamma^I\partial^J E\| \d s\\
				&\lesssim S_1^1+S_1^2.
			\end{aligned}
		\end{equation}
		By Lemmas~\ref{lem:n0},~\ref{lem:Elower} and~(\ref{est:BootE}), we have
		\begin{equation}
			\begin{aligned}
				S_1^1&\lesssim\sum_{|I_2|+|J_2|\leq 1}\int_{0}^{t}C_1^{1+(|I|+|J|)\delta}\epsilon K_0\|\Gamma^{I_2}\partial^{J_2}E\|_{L^\infty}\d s\\
				&\lesssim C_1^{1+(|I|+|J|)\delta}\epsilon^2 K_0^{4}+C_1^{2+(|I|+|J|+5)\delta}\epsilon^2 K_0^{2}\\
				&\lesssim C_1^{2+(|I|+|J|+5)\delta}\epsilon^2 K_0^{4}\\
				&\lesssim C_1^{2+(2|I|+2|J|-4)\delta}\epsilon^2 K_0^{4},
			\end{aligned}
		\end{equation}
		where we use the fact $|I|+|J|\geq 9$ in the last step.
		
		Using again Cauchy inequality, Lemma~\ref{lem:n0} and~(\ref{est:BootE}), we get
		\begin{equation}
			\begin{aligned}
				S_1^2&\lesssim\sum_{\substack{|I_2|+|J_2|\leq |I|+|J|}}C_2^{-1}\int_{0}^{t}\left\|\frac{\Gamma^{I_2}\partial^{J_2}E}{\langle s-r\rangle^{\frac{1}{2}+\delta}}\right\|^2\d s\\
				&+\sum_{\substack{|I_1|\leq 8\\|I_1|+|J_1|\leq N-2}}C_2\int_{0}^{t}\|\langle s-r\rangle^{\frac{1}{2}+\delta}\Gamma^{I_1}\partial^{J_1}n^0\|^2_{L^\infty}\|\partial_t\Gamma^I\partial^J E\|^2\d s\\
				&\lesssim C_2^{-1}\mathcal{E}_{gst,1}(t,\Gamma^I\partial^J E)+C_2^{-1}C_1^{2+2(|I|+|J|-1)\delta}\epsilon^2\\
				&+C_2\int_{0}^{t}\langle s\rangle^{-2+2\delta} K_0^2\mathcal{E}_{gst,1}(s,\Gamma^I\partial^J E)\d s,
			\end{aligned}
		\end{equation}
		in which $C_2$ is a constant to be chosen later.
		
		Finally, let $|I|\leq 10$ and $\max\{N-1,|I|\}\leq |I|+|J|\leq N.$
		\begin{equation}
			\begin{aligned}
				&S_1\lesssim\sum_{\substack{I_1+I_2=I,J_1+J_2=J\\|I_1|\leq 10,|I_1|+|J_1|\leq N\\|I_2|+|J_2|\leq 1}}\int_{0}^{t}\|\Gamma^{I_1}\partial^{J_1}n^0\|\|\Gamma^{I_2}\partial^{J_2}E\|_{L^\infty}\|\partial_t\Gamma^I\partial^J E\|\d s\\
				&+\sum_{\substack{I_1+I_2=I,J_1+J_2=J\\|I_1|\leq 8,|I_1|+|J_1|\leq N-2\\|I_2|\leq 10,|I_2|+|J_2|\leq N}}\int_{0}^{t}\|\langle s-r\rangle^{\frac{1}{2}+\delta}\Gamma^{I_1}\partial^{J_1}n^0\|_{L^\infty}\left\|\frac{\Gamma^{I_2}\partial^{J_2}E}{\langle s-r\rangle^{\frac{1}{2}+\delta}}\right\|\|\partial_t\Gamma^I\partial^J E\| \d s\\
				&\lesssim S_1^3+S_1^4.
			\end{aligned}
		\end{equation}
		%By Lemmas~\ref{lem:n0},~\ref{lem:Elower} and~(\ref{est:BootE}), we have
		Similarly to $S_1^1$, we have
		\begin{equation}
			\begin{aligned}
				S_1^3&\lesssim\sum_{|I_2|+|J_2|\leq 1}\int_{0}^{t}C_1^{1+(|I|+|J|)\delta}\epsilon K_0\|\Gamma^{I_2}\partial^{J_2}E\|_{L^\infty}\d s\\
				&\lesssim C_1^{2+(2|I|+2|J|-4)\delta}\epsilon^2 K_0^{4}.
			\end{aligned}
		\end{equation}
		
		Using again Cauchy inequality, Lemma~\ref{lem:n0} and~(\ref{est:BootE}), we get
		\begin{equation}
			\begin{aligned}
				S_1^4&\lesssim\sum_{\substack{|I_2|+|J_2|\leq |I|+|J|}}C_2^{-1}\int_{0}^{t}\left\|\frac{\Gamma^{I_2}\partial^{J_2}E}{\langle s-r\rangle^{\frac{1}{2}+\delta}}\right\|^2\d s\\
				&+\sum_{\substack{|I_1|\leq 8\\|I_1|+|J_1|\leq N-2}}C_2\int_{0}^{t}\|\langle s-r\rangle^{\frac{1}{2}+\delta}\Gamma^{I_1}\partial^{J_1}n^0\|^2_{L^\infty}\|\partial_t\Gamma^I\partial^J E\|^2\d s\\
				&\lesssim C_2^{-1}\mathcal{E}_{gst,1}(t,\Gamma^I\partial^J E)+C_2^{-1}C_1^{2+2(|I|+|J|-1)\delta}\epsilon^2\\
				&+C_2\int_{0}^{t}\langle s\rangle^{-2+2\delta} K_0^2\mathcal{E}_{gst,1}(s,\Gamma^I\partial^J E)\d s,
			\end{aligned}
		\end{equation}
		in which $C_2$ is a constant to be chosen later.
		
		Thus, the above inequalities imply
		\begin{equation}\label{est:S_1}
			\begin{aligned}
				S_1&\lesssim 
				C_1^{2+(2|I|+2|J|-4)\delta}\epsilon^2 K_0^{4}
				+C_2^{-1}\mathcal{E}_{gst,1}(t,\Gamma^I\partial^J E)\\
				&+C_2^{-1}C_1^{2+2(|I|+|J|-1)\delta}\epsilon^2
				+C_2\int_{0}^{t}\langle s\rangle^{-2+2\delta} K_0^2\mathcal{E}_{gst,1}(s,\Gamma^I\partial^J E)\d s.
			\end{aligned}
		\end{equation}

		Now we estimate $S_2$. For $|I|\leq 10, |I|+|J|\leq N$, based on H\"{o}lder inequality, we infer
		\begin{equation*}
			\begin{aligned}
				S_2
				&\lesssim\sum_{\substack{I_1+I_2=I,J_1+J_2=J\\|I_1|\leq 10,|I_1|+|J_1|\leq N\\|I_2|\leq 5,|I_2|+|J_2|\leq N-5}}\int_{0}^{t}\|\Gamma^{I_1}\partial^{J_1}\Delta n^1\|\|\Gamma^{I_2}\partial^{J_2}E\|_{L^\infty}\|\partial_t\Gamma^I\partial^J E\|\d s\\
				+&\sum_{\substack{I_1+I_2=I,J_1+J_2=J\\|I_1|+|J_1|\leq 4\\|I_2|\leq 10,|I_2|+|J_2|\leq N}}\int_{0}^{t}\|\langle s-r\rangle^{\frac{1}{2}+\delta}\Gamma^{I_1}\partial^{J_1}\Delta n^1\|_{L^\infty}\left\|\frac{\Gamma^{I_2}\partial^{J_2}E}{\langle s-r\rangle^{\frac{1}{2}+\delta}}\right\|\|\partial_t\Gamma^I\partial^J E\|\d s\\
				\lesssim&\ S_2^1+S_2^2.\\
			\end{aligned}
		\end{equation*}
		By inequalities~(\ref{est:commutators}),~(\ref{est:BootE}),~(\ref{est:energy_n1_2}), we deduce
		\begin{equation}
			\begin{aligned}
				S_2^1&\lesssim\sum_{\substack{|I_1|+|I_2|\leq|I|\\J_1+J_2=J\\|I_2|\leq 5,|I_2|+|J_2|\leq N-5}}\int_{0}^{t}\|\partial\partial\Gamma^{I_1}\partial^{J_1}n^1\|\|\Gamma^{I_2}\partial^{J_2}E\|_{L^\infty}\|\partial_t\Gamma^I\partial^J E\|\d s\\
				&\lesssim C_1^{2+(2|I|+2|J|+13)\delta}\epsilon^4 K_0^{[\frac{|I|+|J|}{2}]}+C_1^{4+(2|I|+2|J|+27)\delta}\epsilon^4\\
				&\lesssim C_1^{4+(2|I|+2|J|+27)\delta}\epsilon^4 K_0^{[\frac{|I|+|J|}{2}]}.
			\end{aligned}
		\end{equation}
		By Cauchy inequality, Lemma~\ref{lem:Hessian_n1point} and inequalities~(\ref{est:commutators}),~(\ref{est:BootE}), we deduce
		\begin{equation}
			\begin{aligned}
				S_2^2&\lesssim\sum_{\substack{|I_1|+|I_2|\leq|I|\\J_1+J_2=J\\|I_1|+|J_1|\leq 4}}\int_{0}^{t}\|\langle s-r\rangle^{\frac{1}{2}+\delta}\partial\partial\Gamma^{I_1}\partial^{J_1}n^1\|_{L^\infty}\left\|\frac{\Gamma^{I_2}\partial^{J_2}E}{\langle s-r\rangle^{\frac{1}{2}+\delta}}\right\|\|\partial_t\Gamma^I\partial^J E\|\d s\\
				&\lesssim C_2^{-1}\sum_{|I_2|+|J_2|\leq|I|+|J|}\int_{0}^{t}\left\|\frac{\Gamma^{I_2}\partial^{J_2}E}{\langle s-r\rangle^{\frac{1}{2}+\delta}}\right\|^2\d s\\
				&+C_2\sum_{|I_1|+|J_1|\leq 4}\int_{0}^{t}\|\langle s-r\rangle^{\frac{1}{2}+\delta}\partial\partial\Gamma^{I_1}\partial^{J_1}n^1\|^2_{L^\infty}\|\partial_t\Gamma^I\partial^J E\|^2\d s\\
				%&\lesssim C_2^{-1}C_1^{2+2(|I|+|J|-1)\delta}\epsilon^2
				%+C_2^{-1}\mathcal{E}_{gst,1}(t,\Gamma^I\partial^J E)\\
				%&+C_2\left(C_1^{2+2(|I|+|J|)\delta}\epsilon^6 K_0^7+C_1^{6+(2|I|+2|J|+60)\delta}\epsilon^6+C_1^{4+(2|I|+2|J|+30)\delta}\epsilon^6 K_0^3\right)\\
				&\lesssim C_2^{-1}C_1^{2+2(|I|+|J|-1)\delta}\epsilon^2
				+C_2^{-1}\mathcal{E}_{gst,1}(t,\Gamma^I\partial^J E)
				+C_2C_1^{6+(2|I|+2|J|+60)\delta}\epsilon^6 K_0^7.
			\end{aligned}
		\end{equation}
		From the above two inequalities, we obtain
		\begin{equation}\label{est:S_2}
			\begin{aligned}
				S_2&\lesssim C_1^{4+(2|I|+2|J|+27)\delta}\epsilon^4 K_0^{[\frac{|I|+|J|}{2}]}
				+C_2^{-1}C_1^{2+2(|I|+|J|-1)\delta}\epsilon^2\\
				&+C_2^{-1}\mathcal{E}_{gst,1}(t,\Gamma^I\partial^J E)
				+C_2C_1^{6+(2|I|+2|J|+60)\delta}\epsilon^6 K_0^7.
			\end{aligned}
		\end{equation}
		
		Therefore, for any multi-index $I\in\mathbb{N}^{10},J\in\mathbb{N}^{4}, |I|\leq10, |I|+|J|\leq N, N\geq 10$, from inequalities~(\ref{est:close-energy}),~(\ref{est:S_1}),~(\ref{est:S_2}), we get
		\begin{equation}
			\begin{aligned}
				\mathcal{E}_{gst,1}(t,\Gamma^I\partial^J E)&\lesssim C_2^{-1}\mathcal{E}_{gst,1}(t,\Gamma^I\partial^J E)+C_3\\
				&+C_2\int_{0}^{t}\langle s\rangle^{-2+2\delta} K_0^2\mathcal{E}_{gst,1}(s,\Gamma^I\partial^J E)\d s,
			\end{aligned}
		\end{equation}
		where 
		%\begin{equation*}
		%\begin{aligned}
		%C_3&=\epsilon^2\langle K_0\rangle^{|I|+|J|+1}+C_1^{2+(2|I|+2|J|-4)\delta}\epsilon^2\langle K_0\rangle^4+C_1^{4+(|I|+|J|+32)\delta}\epsilon^4\langle K_0\rangle^3\\
		%&+C_1^{2+(2|I|+2|J|+13)\delta}\epsilon^4\langle K_0\rangle^{[\frac{|I|+|J|}{2}]}+C_1^{4+(2|I|+2|J|+27)\delta}\epsilon^4+C_2^{-1}C_1^{2+2(|I|+|J|-1)\delta}\epsilon^2\\
		%&+C_2\left(C_1^{2+(2|I|+2|J|)\delta}\epsilon^6\langle K_0\rangle^7+C_1^{6+(2|I|+2|J|+60)\delta}\epsilon^6+C_1^{4+(2|I|+2|J|+30)\delta}\epsilon^6\langle K_0\rangle^3\right).
		%\end{aligned}
		%\end{equation*}
		\begin{equation*}
			\begin{aligned}
				C_3&=
				\epsilon^2 K_0^{|I|+|J|+1}+C_2^{-1}C_1^{2+2(|I|+|J|-1)\delta}\epsilon^2
				+C_1^{2+(2|I|+2|J|-4)\delta}\epsilon^2 K_0^{4}\\
				&+C_1^{4+(2|I|+2|J|+27)\delta}\epsilon^4 K_0^{[\frac{|I|+|J|}{2}]}
				+C_2C_1^{6+(2|I|+2|J|+60)\delta}\epsilon^6 K_0^7.
			\end{aligned}
		\end{equation*}
		Additionally, we take $C_2^{-1}$ small enough such that the implicit constant $C_0$ in $\lesssim$ times $C_2^{-1}$ is $\frac{1}{2}$, that is $C_2=2C_0$. We can get further
		\begin{equation*}
			\begin{aligned}
				\mathcal{E}_{gst,1}(t,\Gamma^I\partial^J E)&\leq \frac{1}{2}\mathcal{E}_{gst,1}(t,\Gamma^I\partial^J E)+C_0C_3\\
				&+2C_0^2\int_{0}^{t}\langle s\rangle^{-2+2\delta} K_0^2\mathcal{E}_{gst,1}(s,\Gamma^I\partial^JE)\d s.
			\end{aligned}
		\end{equation*}
		
		Finally, using Gronwall inequality in Lemma~\ref{lem:gronwall}, we obtain
		\begin{equation}
			\mathcal{E}_{gst,1}(t,\Gamma^I\partial^J E)\leq 2C_0C_3 \exp\left(4C_0^2 K_0^2\right)\leq\frac{1}{4}C_1^{2+2(|I|+|J|)\delta}\epsilon^2,
		\end{equation}
		where we take $C_1$ and $\epsilon$ satisfying
		\begin{equation*}
			\begin{aligned}
				&C_1^2>40C_0K_0^{N+1}\exp\left(4C_0^2 K_0^2\right),\\
				&C_1^\delta>20\exp\left(4C_0^2 K_0^2\right),\\
				&C_1^{3+33\delta}\epsilon^2<1.
			\end{aligned}
		\end{equation*}
		This strictly improves the estimate of $\mathcal{E}_{gst,1}(t,\Gamma^I\partial^J E)$ in~(\ref{est:BootE}).
		%$C_1^2\gg\langle K_0\rangle^{N+1}\exp\left(2C_0C_2\langle K_0\rangle^2\right)$, $C_1^\delta\gg\exp\left(2C_0C_2\langle K_0\rangle^2\right)$ and $C_1^{3+33\delta}\epsilon^2\ll 1$ 
		%where 
		%\begin{equation*}
		%\begin{aligned}
		%C_3&=\epsilon^2\langle K_0\rangle^{|I|+1}+C_1^{2+(2|I|-3)\delta}\epsilon^2\langle K_0\rangle^4+C_1^{4+(|I|+32)\delta}\epsilon^4\langle K_0\rangle^3\\
		%&+C_1^{2+(2|I|+13)\delta}\epsilon^4\langle K_0\rangle^{[\frac{|I|}{2}]}+C_1^{4+(2|I|+27)\delta}\epsilon^4+C_2^{-1}C_1^{2+2(|I|-1)\delta}\epsilon^2\\
		%&+C_2\left(C_1^{2+2|I|\delta}\epsilon^6\langle K_0\rangle^7+C_1^{6+(2|I|+60)\delta}\epsilon^6\right).
		%\end{aligned}
		%\end{equation*}

		\smallskip
		
		\textbf{Step 2.} Closing the estimate of $\Gamma^I\partial^J E$. 
		For $|I|\leq5, |I|+|J|\leq N-5$, by Lemma~\ref{lem:nonlinear}, we deduce
		\begin{equation}
			\begin{aligned}
				&\sum_{|K|\leq 4}\|\langle t+r\rangle\Gamma^K\Gamma^I\partial^J(n^0E+\Delta n^1 E)\|\\
				&\lesssim\sum_{|I|\leq|K|\leq |I|+4}\|\langle t+r\rangle\Gamma^K\partial^J(n^0E+\Delta n^1 E)\|\\
				&\lesssim\langle t\rangle^{-\frac{1}{8}}\left(C_1^{1+(|I|+|J|+12)\delta}\epsilon K_0^2+C_1^{3+(|I|+|J|+43)\delta}\epsilon^3 K_0^{[\frac{|I|+|J|+10}{2}]}\right).
			\end{aligned}
		\end{equation}
		
		Then based on the above inequality, Corollary~\ref{cor:KGpoint} and~(\ref{est:I.D.1}), we obtain
		\begin{equation}
			\begin{aligned}
				\langle t+r\rangle^{\frac{3}{2}}|\Gamma^I\partial^J E|
				&\lesssim\epsilon K_0^{[\frac{|I|+|J|+5}{2}]}+
				C_1^{1+(|I|+|J|+12)\delta}\epsilon K_0^2\\
				&+C_1^{3+(|I|+|J|+43)\delta}\epsilon^3 K_0^{[\frac{|I|+|J|+10}{2}]}.
			\end{aligned}
		\end{equation}
		Finally, we can choose $C_1$ and $\epsilon$ satisfying
		\begin{equation*}
			\begin{aligned}
				&C_1^2>40C_0K_0^{N+1}\exp\left(4C_0^2 K_0^2\right),\\
				&C_1^\delta>20\exp\left(4C_0^2 K_0^2\right),\\
				&C_1^{3+33\delta}\epsilon^2<1,
				%&C_1^2\gg K_0^{N+1}\exp\left(4C_0^2 K_0^2\right),\\
				%&C_1^\delta\gg\exp\left(4C_0^2 K_0^2\right),\\
				%&C_1^{3+33\delta}\epsilon^2\ll 1,
			\end{aligned}
		\end{equation*}
		such that
		\begin{equation*}
			\langle t+r\rangle^{\frac{3}{2}}|\Gamma^I\partial^J E|\leq\frac{1}{2}C_1^{1+(|I|+|J|+13)\delta}\epsilon,
		\end{equation*}
		which strictly improves the estimate of $\Gamma^I\partial^J E$ in~(\ref{est:BootE}). Note that $C_0$ is the implicit constant in $\lesssim$.
		
		In conclusion, the proof of Proposition~\ref{pro:KGZ-L} is completed. Besides, estimates in~(\ref{est:pointwise1}) follow from~(\ref{est:BootE}), (\ref{est:Gamma_n0point}), (\ref{est:allHessian}).
		%and we establish global solution $(E,n)$ to the system~(\ref{eq:KGZ-L})--(\ref{eq:KGZ-ID}).
	\end{proof}

	\subsection{End of the proof of Theorem~\ref{thm:KGZ-L}}\label{SS:scatter}
	To complete the proof of Theorem~\ref{thm:KGZ-L}, we need to show that the solution $(E,n)$ scatters linearly. By Lemmas~\ref{lem:scatterWave},~\ref{lem:scatterKG}, we only need to bound
	\begin{equation*}
		\int_{0}^{+\infty}\|nE\|_{H^N(\mathbb{R}^3)}\d s \quad\mbox{and} \quad \int_{0}^{+\infty}\left\|\Delta |E|^2\right\|_{\dot{H}^{N-1}(\mathbb{R}^3)}\d s.
	\end{equation*}
	
	\textbf{Step 1.} Boundedness of $\int_{0}^{+\infty}\|nE\|_{H^N(\mathbb{R}^3)}\d s.$
	
	From the definition $n=n^0+\Delta n^1$ and the H\"{o}lder inequality, we deduce
	\begin{equation}
		\begin{aligned}
			\|nE\|_{H^N(\mathbb{R}^3)}
			&\lesssim\sum_{|I|\leq N}\left(\|\nabla^I(n^0E)\|+\|\nabla^I(\Delta n^1E)\|\right)\\
			&\lesssim S_{3}+S_{4}+S_{5}+S_{6},
		\end{aligned}
	\end{equation}
	where
	\begin{equation}
		\begin{aligned}
			S_{3}&=\sum_{\substack{I_1+I_2= I\\|I_1|\leq N\\|I_2|\leq N-5}}\|\nabla^{I_1}n^0\|\|\nabla^{I_2}E\|_{L^\infty},\\
			S_{4}&=\sum_{\substack{I_1+I_2=I\\|I_1|\leq N-2\\|I_2|\leq N}}\|\langle s-r\rangle^{\frac{1}{2}+\delta}\nabla^{I_1}n^0\|_{L^\infty}\left\|\frac{\nabla^{I_2}E}{\langle s-r\rangle^{\frac{1}{2}+\delta}}\right\|,\\
			S_{5}&=\sum_{\substack{I_1+I_2= I\\|I_1|\leq N\\|I_2|\leq N-5}}\|\nabla^{I_1}\Delta n^1\|\|\nabla^{I_2}E\|_{L^\infty},\\
			S_{6}&=\sum_{\substack{I_1+I_2=I\\|I_1|\leq N-5\\|I_2|\leq N}}\|\langle s-r\rangle^{\frac{1}{2}+\delta}\nabla^{I_1}\Delta n^1\|_{L^\infty}\left\|\frac{\nabla^{I_2}E}{\langle s-r\rangle^{\frac{1}{2}+\delta}}\right\|.
		\end{aligned}
	\end{equation}
	Based on estimates in (\ref{est:BootE}) and (\ref{est:Gamma_n0L^2}), we have
	\begin{equation}
		\begin{aligned}
			S_{3}\lesssim C_1^{1+(N+8)\delta}\epsilon K_0\langle s\rangle^{-\frac{3}{2}}.
		\end{aligned}
	\end{equation}
	Thanks to the Cauchy inequality, we get
	\begin{equation}
		\begin{aligned}
			S_{4}&\lesssim\sum_{|I_1|\leq N-2}\|\langle s-r\rangle^{\frac{1}{2}+\delta}\nabla^{I_1}n^0\|_{L^\infty}^2+\sum_{|I_2|\leq N}\left\|\frac{\nabla^{I_2}E}{\langle s-r\rangle^{\frac{1}{2}+\delta}}\right\|^2\\
			&\lesssim K_0^2\langle s\rangle^{-2+2\delta}+\sum_{|I_2|\leq N}\left\|\frac{\nabla^{I_2}E}{\langle s-r\rangle^{\frac{1}{2}+\delta}}\right\|^2,
		\end{aligned}
	\end{equation}
	where we use (\ref{est:Gamma_n0point}) in the last step.
	
	Based on the estimates in (\ref{est:BootE}), (\ref{est:energy_n1_2}), we obtain
	\begin{equation}
		\begin{aligned}
			S_{5}&\lesssim\sum_{\substack{|I_1|\leq N,|I_2|\leq N-5}}\|\partial\partial\nabla^{I_1} n^1\|\|\nabla^{I_2}E\|_{L^\infty}\\
			&\lesssim(\epsilon^2 K_0^{[\frac{N}{2}]}+C_1^{2+(N+14)\delta}\epsilon^2)C_1^{1+(N+8)\delta}\epsilon\langle s\rangle^{-\frac{3}{2}}.
		\end{aligned}
	\end{equation}
	Using again the Cauchy inequality, from Lemma~\ref{lem:Hessian_n1point} we get
	\begin{equation}
		\begin{aligned}
			S_{6}&\lesssim\sum_{|I_1|\leq N-5}\|\langle s-r\rangle^{\frac{1}{2}+\delta}\partial\partial\nabla^{I_1}n^1\|_{L^\infty}^2+\sum_{|I_2|\leq N}\left\|\frac{\nabla^{I_2}E}{\langle s-r\rangle^{\frac{1}{2}+\delta}}\right\|^2\\
			&\lesssim\left(\epsilon^2 K_0^{[\frac{N-2}{2}]}+C_1^{2+(N+21)\delta}\epsilon^2\right)^2\langle s\rangle^{-\frac{3}{2}}+\sum_{|I_2|\leq N}\left\|\frac{\nabla^{I_2}E}{\langle s-r\rangle^{\frac{1}{2}+\delta}}\right\|^2.
		\end{aligned}
	\end{equation}
	Therefore, we deduce
	\begin{equation*}
		\begin{aligned}
			\int_{0}^{+\infty}\|nE\|_{H^N(\mathbb{R}^3)}\d s
			%\lesssim&\left(\langle K_0\rangle^2+C_1^{1+(N+8)\delta}\epsilon\langle K_0\rangle+C_1^{1+(N+8)\delta}\epsilon^3\langle K_0\rangle^{[\frac{N}{2}]}+C_1^{3+(2N+22)\delta}\epsilon^3+\epsilon^4\langle K_0\rangle^{N-2}+C_1^{4+(2N+42)\delta}\epsilon^4\right)\langle t\rangle^{-\frac{1}{2}}\\
			%+&\sum_{|I_2|\leq N}\int_{t}^{\infty}\left\|\frac{\nabla^{I_2}E}{\langle s-r\rangle^{\frac{1}{2}+\delta}}\right\|_{L^2}^2\d s\\
			\lesssim K_0^2+C_1^{1+(N+8)\delta}\epsilon K_0+C_1^{3+(2N+22)\delta}\epsilon^2.
		\end{aligned}
	\end{equation*}
	
	\textbf{Step 2.} Boundedness of $\int_{0}^{+\infty}\|\Delta |E|^2\|_{\dot{H}^{N-1}(\mathbb{R}^3)}\d s.$
	
	By the H\"{o}lder inequality and estimates in (\ref{est:BootE}), we observe that
	\begin{equation}
		\begin{aligned}
			\left\|\Delta|E|^2\right\|_{\dot{H}^{N-1}(\mathbb{R}^3)}&\lesssim\sum_{|I|=N-1}\left(\|\nabla^{I}(\Delta E\cdot E)\|+\|\nabla^{I}(\nabla E\cdot \nabla E)\|\right)\\
			&\lesssim S_{7}+S_{8}+S_{9},
		\end{aligned}
	\end{equation}
	where
	\begin{equation*}
		\begin{aligned}
			S_{7}&=\sum_{\substack{I_1+I_2= I\\|I_1|\leq N-1,|I_2|\leq N-5}}\|\nabla^{I_1}\Delta E\|\|\nabla^{I_2}E\|_{L^\infty}\lesssim C_1^{2+(N+13)\delta}\epsilon^2\langle s\rangle^{-\frac{3}{2}},\\
			S_{8}&=\sum_{\substack{I_1+I_2= I\\|I_1|\leq N-7,|I_2|\leq N-1}}\|\nabla^{I_1}\Delta E\|_{L^\infty}\|\nabla^{I_2}E\|\lesssim C_1^{2+(N+13)\delta}\epsilon^2\langle s\rangle^{-\frac{3}{2}},\\
			S_{9}&=\sum_{\substack{I_1+I_2= I\\|I_1|\leq N-1,|I_2|\leq N-6}}\|\nabla^{I_1}\nabla E\|\|\nabla^{I_2}\nabla E\|_{L^\infty}\lesssim C_1^{2+(N+13)\delta}\epsilon^2\langle s\rangle^{-\frac{3}{2}}.\\
		\end{aligned}
	\end{equation*}
	Therefore, we obtain
	\begin{equation*}
		\int_{0}^{+\infty}\|\Delta |E|^2\|_{\dot{H}^{N-1}(\mathbb{R}^3)}\d s\lesssim C_1^{2+(N+13)\delta}\epsilon^2.
	\end{equation*}
	The proof of Theorem~\ref{thm:KGZ-L} is completed.

	\section{Proof of Theorem~\ref{thm:KGZ-L2}}\label{S:proofTh1.4}
	
	\subsection{Bootstrap assumption}
	Let $N\in \mathbb{N} \ \mbox{with} \ N\geq 10$. Fix $0<\delta\ll 1$. To prove Theorem~\ref{thm:KGZ-L2}, we introduce the following bootstrap assumption of $E$: for $C_1\gg1$ and $0<\epsilon\ll C_1^{-1}$ to be chosen later, 
	\begin{equation}\label{est:BootE2}
		\left\{\begin{aligned}
			\mathcal{E}^{\frac{1}{2}}_{gst,1}(t,\Gamma^I\partial^{J} E)&\leq C_1^{1+(|I|+|J|)\delta}\epsilon,\quad\quad \ \ \ \mbox{for}\quad |I|\leq 10, \ |I|+|J|\leq N,\\
			\sup_{x\in \R^3}\langle t+r\rangle^{\frac{3}{2}}|\Gamma^I \partial^{J}E(t,x)|&\leq C_1^{1+(|I|+|J|+13)\delta}\epsilon,\  \ \quad\mbox{for}\quad |I|\leq 5,\ |I|+|J|\leq N-5.
		\end{aligned}\right.
	\end{equation}

	For all initial data $(\vec{E}_{0},\vec{n}_{0})$ satisfying~(\ref{initialdata2}) and \eqref{initialdata3}, we set
	\begin{equation*}\label{def:T2}
		T^{*}(\vec{E}_{0},\vec{n}_{0})=\sup\left\{ t\in [0,+\infty): E\ \mbox{satisfies}~\eqref{est:BootE2}\ \mbox{on}\ [0,t]\right\}.
	\end{equation*} 
	Note that we denote $(E_{0},E_{1},n_{0},n_{1})$ by $(\vec{E}_{0},\vec{n}_{0})$.
	
	In this section, we will first prove the following proposition, which is part of Theorem~\ref{thm:KGZ-L2}.
	\begin{proposition}\label{pro:KGZ-L2}
		For all initial data $(\vec{E}_{0},\vec{n}_{0})$ satisfying the conditions~\eqref{initialdata2} and \eqref{initialdata3} in Theorem~\ref{thm:KGZ-L2}, we have $T^{*}(\vec{E}_{0},\vec{n}_{0})=+\infty$.
	\end{proposition}

	\subsection{Key estimates}
	In this subsection, we establish some estimates on solution $(E,n^0,n^1)$ and the nonlinear term in~(\ref{eq:vectorfieldE}) that will be used in the proof of Proposition~\ref{pro:KGZ-L2}. From now on, the implied constants in $\lesssim$ do not depend on the constants $C_1$ and $\epsilon$ appearing in the bootstrap assumption~(\ref{est:BootE2}).		
	
	First, we give estimates of the solution $n^0$ to the homogeneous wave equation in~(\ref{eq:KGZ-L2}).
	\begin{lemma}[Estimates of $n^0$]\label{lem:n0-2}
		We have the following estimates.
		
		\begin{enumerate}
			\item {\rm {$L^2$ estimate of $\Gamma^I\partial^J n^0$}}. For $|I|\leq 10,|I|+|J|\leq N$, we have
			\begin{equation}\label{est:Gamma_n0L^2-2}
				\|\Gamma^I\partial^J n^0\|\lesssim K_0.
			\end{equation}
			
			\item {\rm {Pointwise decay estimate of $\Gamma^I\partial^J n^0$}}. For $|I|\leq 8, |I|+|J|\leq N-2$, we have
			\begin{equation}\label{est:Gamma_n0point2}
				|\Gamma^I\partial^J n^0|\lesssim \langle t+r\rangle^{-1}\langle t-r\rangle^{-\frac{1}{2}}K_0.
			\end{equation}
			
		\end{enumerate}
	\end{lemma}
	\begin{proof}
		The proof is similar to Lemma~\ref{lem:n0}, and we omit it.
	\end{proof}
	
	Second, we introduce some energy estimates of the solution $n^1$ to the nonhomogeneous wave equation in~(\ref{eq:KGZ-L2}), including standard energy estimates and conformal energy estimates.
	\begin{lemma}\label{lem:energyn0-2}
		Let $|I|\leq 10, |I|+|J|\leq N.$ For all $t\in [0,T^*(\vec{E}_{0},\vec{n}_{0})),$ we have the following estimates.
		\begin{enumerate}
			\item {\rm {Standard energy estimate of $\Gamma^I\partial^J n^1$}}. We have
			\begin{equation}\label{est:energy_n1_1-2}
				\mathcal{E}^{\frac{1}{2}}(t,\Gamma^I\partial^J n^1)\lesssim \epsilon^2 K_0^{\max\{[\frac{|I|+|J|-1}{2}],0\}}+C_1^{2+(|I|+|J|+13)\delta}\epsilon^2.
			\end{equation}
			
			\item {\rm {Standard energy estimate of $\partial\Gamma^I\partial^J n^1$}}. We have
			\begin{equation}\label{est:energy_n1_2-2}
				\mathcal{E}^{\frac{1}{2}}(t,\partial\Gamma^I\partial^J n^1)\lesssim \epsilon^2 K_0^{[\frac{|I|+|J|}{2}]}+C_1^{2+(|I|+|J|+14)\delta}\epsilon^2.
			\end{equation}
			
			\item{\rm{Conformal energy estimate of $\Gamma^I\partial^J n^1$.}} We have
			\begin{equation}\label{est:conformal_n1-2}
				\mathcal{E}^{\frac{1}{2}}_{con}(t,\Gamma^I\partial^J n^1)\lesssim \epsilon^2 K_0^{\max\{[\frac{|I|+|J|-1}{2}],0\}}+C_1^{2+(|I|+|J|+13)\delta}\epsilon^2\langle t\rangle^{\frac{1}{2}}.
			\end{equation}
		\end{enumerate}
	\end{lemma}				
	\begin{proof}
		The proof is similar to Lemma~\ref{lem:energyn1}, and we omit it.
	\end{proof}
	Third, we also deduce  the extra decay estimates for Hessian of $\Gamma^I\partial^J n^1$ in the following lemma.
	\begin{lemma}\label{lem:Hessian_n1point2}
		Let $|I|\leq 5, |I|+|J|\leq N-5.$ For all $t\in [0,T^*(\vec{E}_{0},\vec{n}_{0}))$, we have the following estimates on $\partial\partial\Gamma^I\partial^J n^1$.
		\begin{enumerate}
			\item {\rm {Let $r\leq \frac{t}{2}$}}. We have
			\begin{equation}
				|\partial\partial\Gamma^I\partial^J n^1|\lesssim\langle t+r\rangle^{-\frac{3}{4}}\langle t-r\rangle^{-1}\left(\epsilon^2 K_0^{[\frac{|I|+|J|+3}{2}]}+C_1^{2+(|I|+|J|+26)\delta}\epsilon^2\right).
			\end{equation}
			
			\item {\rm {Let $r\geq 2t$}}. We have
			\begin{equation}
				|\partial\partial\Gamma^I\partial^J n^1|\lesssim\langle t+r\rangle^{-\frac{3}{2}}\left(\epsilon^2 K_0^{[\frac{|I|+|J|+3}{2}]}+C_1^{2+(|I|+|J|+17)\delta}\epsilon^2\right).
			\end{equation}
			
			\item{\rm{Let $\frac{t}{2}\leq r\leq 2t$.}} We have
			\begin{equation}
				|\partial\partial\Gamma^I\partial^J n^1|\lesssim\langle t+r\rangle^{-1}\langle t-r\rangle^{-1}\left(\epsilon^2 K_0^{[\frac{|I|+|J|+3}{2}]}+C_1^{2+(|I|+|J|+26)\delta}\epsilon^2\right).
			\end{equation}
			
			\item{\rm{}} We have
			\begin{equation}\label{est:allHessian2}
				|\partial\partial\Gamma^I\partial^J n^1|\lesssim\langle t+r\rangle^{-1}\langle t-r\rangle^{-\frac{1}{2}}\left(\epsilon^2 K_0^{[\frac{|I|+|J|+3}{2}]}+C_1^{2+(|I|+|J|+26)\delta}\epsilon^2\right).
			\end{equation}
		\end{enumerate}
	\end{lemma}
	\begin{proof}
		The proof is similar to Lemma~\ref{lem:Hessian_n1point}, and we omit it.
	\end{proof}
	Next, we give another pointwise estimate of $\Gamma^I\partial^J E$.
	\begin{lemma}\label{lem:PointE2-2}
		For all $t\in [0,T^*(\vec{E}_{0},\vec{n}_{0})), |I|\leq 5, |I|+|J|\leq N-5$, we have
		\begin{equation}\label{est:PointE3}
			|\Gamma^I\partial^J E|\lesssim \langle t+r\rangle^{-\frac{9}{8}}C_1^{1+(|I|+|J|+8)\delta}\epsilon.
		\end{equation}
	\end{lemma}
	\begin{proof}
		The proof is similar to Lemma~\ref{lem:PointE2}, and we omit it.
		%From~(\ref{est:globalSobo}) and~(\ref{est:BootE2}), for $|I|\leq N-5$, we get
		%\begin{equation}
		%\begin{aligned}
		%|\Gamma^IE|\lesssim&\langle t+r\rangle^{-\frac{3}{4}}C_1^{1+(|I|+3)\delta}\epsilon,\\
		%|\Gamma^IE|\lesssim&\langle t+r\rangle^{-\frac{3}{2}}C_1^{1+(|I|+13)\delta}\epsilon,
		%\end{aligned}
		%\end{equation}
		%which yield~(\ref{est:PointE3}).
	\end{proof}
	The following lemma is a lower-order pointwise decay estimate for $\Gamma^I\partial^J E$.
	\begin{lemma}\label{lem:Elower2}
		For all $t\in [0,T^*(\vec{E}_{0},\vec{n}_{0}))$, we have
		\begin{equation}\label{est:Elower2}
			\begin{aligned}
				\sum_{|I|+|J|\leq 1}\langle t+r\rangle^{\frac{5}{4}}|\Gamma^I\partial^J E|
				\lesssim\epsilon K_0^{3}+C_1^{1+5\delta}\epsilon K_0,
			\end{aligned}
		\end{equation}
		where $C_1^{2+27\delta}\epsilon^2<1$.
	\end{lemma}
	\begin{proof}
		The proof is similar to Lemma~\ref{lem:Elower}, and we omit it.
	\end{proof}
	Also, we give the following weighted $L^2$ estimate for $\Gamma^I\partial^J E$.
	\begin{lemma}\label{lem:weightE2}
		Let $|I|\leq 9, |I|+|J|\leq N-1$. For all $t\in [0,T^*(\vec{E}_{0},\vec{n}_{0})),$ we have the following weighted $L^2$ estimate on $\Gamma^I\partial^J E$.
		\begin{equation}
			\left\|\frac{\langle t+r\rangle}{\langle t-r\rangle}\Gamma^I\partial^J E\right\|\lesssim C_1^{1+(|I|+|J|+8)\delta}\epsilon K_0+C_1^{1+(|I|+|J|+13)\delta}\epsilon^3  K_0^{[\frac{|I|+|J|+3}{2}]},
		\end{equation}
		where $C_1^{2+19\delta}\epsilon^2<1.$	
	\end{lemma}
	\begin{proof}
		The proof is similar to Lemma~\ref{lem:weightE}, and we omit it.
	\end{proof}							
	
	Finally, we deduce the following weighted $L^2$ estimate of the nonlinear term in~(\ref{eq:vectorfieldE}).
	\begin{lemma}\label{lem:nonlinear2}
		For $|I|\leq 9, |I|+|J|\leq N-1$, and all $t\in [0,T^*(\vec{E}_{0},\vec{n}_{0}))$, we have
		\begin{equation}
			\begin{aligned}
				&\|\langle t+r\rangle\Gamma^I\partial^J (n^0E+\Delta n^1E)\|\\
				&\lesssim\langle t\rangle^{-\frac{1}{8}}\left(C_1^{1+(|I|+|J|+8)\delta}\epsilon K_0^2+C_1^{3+(|I|+|J|+39)\delta}\epsilon^3 K_0^{[\frac{|I|+|J|+6}{2}]}\right),
			\end{aligned}
		\end{equation}
		where $C_1^{2+19\delta}\epsilon^2<1.$
	\end{lemma}
	
	\begin{proof}
		The proof is similar to Lemma~\ref{lem:nonlinear}, and we omit it.
	\end{proof}			
	
	\subsection{Proof of Proposition~\ref{pro:KGZ-L2}}
	In this subsection, we will complete the proof of Proposition~\ref{pro:KGZ-L2} by improving all the estimates of $E$ in~\eqref{est:BootE2}.						
	\begin{proof}[Proof of Proposition~\ref{pro:KGZ-L2}]
		For any initial data $(\vec{E}_{0},\vec{n}_{0})$ satisfying the conditions~(\ref{initialdata2}) and \eqref{initialdata3}, we consider the corresponding solution $(E,n^0,n^1)$ of~(\ref{eq:KGZ-L2}). From the initial conditions~(\ref{initialdata2}), we observe that 
		\begin{equation}\label{est:I.D.2}
			\begin{aligned}
				\mathcal{E}^{\frac{1}{2}}_{gst,1}(0,\Gamma^I\partial^J E)&\lesssim\epsilon K_0^{[\frac{|I|+|J|+1}{2}]},\quad\mbox{for}\quad |I|\leq 10, |I|+|J|\leq N,\\
				\|\langle r\rangle^{2}\Gamma^I\partial^J E(0,x)\|&\lesssim \epsilon K_0^{[\frac{|I|+|J|}{2}]},\quad \ \ \ \mbox{for}\quad |I|\leq 10, |I|+|J|\leq N.
			\end{aligned}
		\end{equation}
		
		\textbf{Step 1.} Closing the estimate in $\mathcal{E}_{gst,1}(t,\Gamma^{I}\partial^{J}{E})$. Let $|I|\leq 10, |I|+|J|\leq N$ with $N\geq 10$.
		
		First, if $|I|=|J|=0,$ multiplying the equation~(\ref{eq:vectorfieldE}) by $e^q\partial_t E$ and integrating with respect to $x$ and $t$, we can obtain
		\begin{equation}\label{est:energy_E}
			\begin{aligned}
				&\frac{1}{2}\int_{\R^{3}}e^q\left(|\partial E|^2+| E|^2+n^0|E|^2\right)\d x(t)\\
				&+\frac{1}{2}\int_{0}^{t}\int_{\R^{3}}e^q\left(\frac{|E|^2}{\langle r-s\rangle^{1+2\delta}}+\frac{n^0|E|^2}{\langle r-s\rangle^{1+2\delta}}+\sum_{a=1}^{3}\frac{|G_a E|^2}{\langle r-s\rangle^{1+2\delta}}-\partial_t n^0|E|^2\right)\d x\d s\\
				&=\frac{1}{2}\int_{\R^{3}}e^q\left(|\partial E|^2+|E|^2+n^0|E|^2\right)\d x(0)
				-\int_{0}^{t}\int_{\R^{3}}e^q(\Delta n^1E)\partial_t E\d x\d s.
			\end{aligned}
		\end{equation}
		By~(\ref{est:Sobo}) and (\ref{initialdata2}), we deduce
		\begin{equation}\label{est:n0_1}
			\begin{aligned}
				|\partial_t n^0|&\leq C_{KS}\langle t+r\rangle^{-1}\langle t-r\rangle^{-\frac{1}{2}}\sum_{|K|\leq2}\|Z^K\partial_t n^0\|\\
				&\leq71C_{KS}K_0\langle t+r\rangle^{-1}\langle t-r\rangle^{-\frac{1}{2}}.
			\end{aligned}
		\end{equation}
		Similarly, we have
		\begin{equation*}
			\begin{aligned}
				\sum_{1\leq a \leq 3}|\partial_a\partial_t n^0|&\leq C_{KS}\langle t+r\rangle^{-1}\langle t-r\rangle^{-\frac{1}{2}}\sum_{1\leq a \leq 3}\sum_{|K|\leq2}\|Z^K\partial_a\partial_t n^0\|\\
				&\leq\frac{174}{71^2\times174C_{KS}^2K_0^2}\langle t+r\rangle^{-1}\langle t-r\rangle^{-\frac{1}{2}}.
			\end{aligned}
		\end{equation*}
		Then, we can get another estimate on $\partial_t n^0$ in the following
		\begin{equation}\label{est:n0_2}
			\begin{aligned}
				|\partial_t n^0|&\leq\int_{0}^{+\infty}|\partial_r\partial_t n^0|\d \rho\leq\sum_{1\leq a \leq 3}\int_{0}^{+\infty}|\partial_a\partial_t n^0|\d \rho\\
				&\leq\frac{1}{71^2C_{KS}^2K_0^2}\langle t+r\rangle^{-\frac{3}{8}}.
			\end{aligned}
		\end{equation}
		By interpolating inequalities~(\ref{est:n0_1}) and~(\ref{est:n0_2}), we obtain
		\begin{equation}
			|\partial_t n^0|\leq\langle t+r\rangle^{-\frac{1}{8}+2\delta}\langle t-r\rangle^{-1-2\delta}<\langle t-r\rangle^{-1-2\delta}.
		\end{equation}
		
		Moreover, by Lemma~\ref{lem:homo-positive} and (\ref{initialdata3}), we see $n^0\geq 0$.
		Thus, from the definition of $\mathcal{E}_{gst,1}(t,E)$ and the above inequalities, we have
		\begin{equation*}
			\begin{aligned}
				\mathcal{E}_{gst,1}(t,E)&\lesssim\mathcal{E}_{gst,1}(0,E)+\|n^0(0,x)\|_{L^\infty}\|E(0,x)\|^2\\
				&+\int_{0}^{t}\|\Delta n^1E\partial_tE\|_{L^1}\d s.
			\end{aligned}
		\end{equation*}
		Based on H\"{o}lder inequality and estimates in Lemma~\ref{lem:Hessian_n1point2} and (\ref{est:BootE2}), we obtain
		\begin{equation*}
			\begin{aligned}
				\int_{0}^{t}\|\Delta n^1E\partial_tE\|_{L^1}\d s&\lesssim\int_{0}^{t}\|\langle s-r\rangle^{\frac{1}{2}+\delta}\Delta n^1\|_{L^\infty}\left\|\frac{E}{\langle s-r\rangle^{\frac{1}{2}+\delta}}\right\|\|\partial_tE\|\d s\\
				&\lesssim C_1\epsilon\left(\int_{0}^{t}\|\langle s-r\rangle^{\frac{1}{2}+\delta}\Delta n^1\|^2_{L^\infty}\d s\right)^{\frac{1}{2}}\left(\int_{0}^{t}\left\|\frac{E}{\langle s-r\rangle^{\frac{1}{2}+\delta}}\right\|^2\d s\right)^{\frac{1}{2}}\\
				&\lesssim C_1^2\epsilon^4K_0+C_1^{4+26\delta}\epsilon^4.
			\end{aligned}
		\end{equation*}
		Finally, by (\ref{initialdata2}),~(\ref{est:I.D.2}) and the above inequalities, we conclude
		\begin{equation}
			\mathcal{E}_{gst,1}(t,E)\lesssim\epsilon^2+\epsilon^2K_0+C_1^2\epsilon^4K_0+C_1^{4+26\delta}\epsilon^4\leq\frac{1}{4}C_1^2\epsilon^2,
		\end{equation}
		where we take $K_0\ll C_1^2$ and $C_1^{2+26\delta}\epsilon^2\ll1.$
		
		If $|I|\leq 10, 1\leq|I|+|J|\leq N,$ multiplying the equation~(\ref{eq:vectorfieldE}) by $e^q\partial_t\Gamma^I\partial^J E$ and integrating with respect to $x$ and $t$, we can obtain
		\begin{equation}\label{eq:ghost}
			\begin{aligned}
				&\frac{1}{2}\int_{\R^{3}}e^q\left(|\partial\Gamma^I\partial^J E|^2+|\Gamma^I\partial^J E|^2+n^0|\Gamma^I\partial^J E|^2\right)\d x(t)\\
				&+\frac{1}{2}\int_{0}^{t}\int_{\R^{3}}e^q\left(\frac{|\Gamma^I\partial^J E|^2}{\langle r-s\rangle^{1+2\delta}}+\frac{n^0|\Gamma^I\partial^J E|^2}{\langle r-s\rangle^{1+2\delta}}+\sum_{a=1}^{3}\frac{|G_a\Gamma^I\partial^J E|^2}{\langle r-s\rangle^{1+2\delta}}\right)\d x\d s\\
				&=\frac{1}{2}\int_{\R^{3}}e^q\left(|\partial\Gamma^I\partial^J E|^2+|\Gamma^I\partial^J E|^2+n^0|\Gamma^I\partial^J E|^2\right)\d x(0)\\
				&+\frac{1}{2}\int_{0}^{t}\int_{\R^{3}}e^q\partial_t n^0|\Gamma^I\partial^J E|^2\d x\d s-\int_{0}^{t}\int_{\R^{3}}e^q\Gamma^I\partial^J(\Delta n^1E)\partial_t\Gamma^I\partial^J E\d x\d s\\
				&-\sum_{\substack{I_1+I_2= I,J_1+J_2=J\\|I_2|+|J_2|< |I|+|J|}}\int_{0}^{t}\int_{\R^{3}}e^q\Gamma^{I_1}\partial^{J_1}n^0\Gamma^{I_2}\partial^{J_2}E\partial_t\Gamma^I\partial^J E\d x\d s.
			\end{aligned}
		\end{equation}
		Thanks to Lemma~\ref{lem:homo-positive} and (\ref{initialdata2}), we see $n^0\geq 0$.
		Thus from the definition of $\mathcal{E}_{gst,1}(t,\Gamma^I\partial^J E)$ and the above identity, we have
		\begin{equation}\label{est:energyE2}
			\begin{aligned}
				\mathcal{E}_{gst,1}(t,\Gamma^I\partial^J E)
				&\lesssim\mathcal{E}_{gst,1}(0,\Gamma^I\partial^J E)+\|n^0(0,x)\|_{L^\infty}\|\Gamma^I\partial^J E(0,x)\|^2\\
				&+\int_{0}^{t}\int_{\R^{3}}|\partial_t n^0||\Gamma^I\partial^J E|^2\d x \d s\\
				&+\sum_{\substack{I_1+I_2=I,J_1+J_2=J\\|I_2|+|J_2|< |I|+|J|}}\int_{0}^{t}\int_{\R^{3}}|\Gamma^{I_1}\partial^{J_1}n^0||\Gamma^{I_2}\partial^{J_2}E||\partial_t\Gamma^I\partial^J E|\d x\d s\\
				&+\sum_{\substack{I_1+I_2=I,J_1+J_2=J}}\int_{0}^{t}\int_{\R^{3}}|\Gamma^{I_1}\partial^{J_1}\Delta n^1||\Gamma^{I_2}\partial^{J_2}E||\partial_t\Gamma^I\partial^J E|\d x\d s\\
				&\lesssim R_1+R_2+R_3+R_4+R_5.
			\end{aligned}
		\end{equation}
		Estimates~(\ref{est:I.D.2}) and~(\ref{est:Gamma_n0point2}) yield
		\begin{equation}\label{est:R1R2}
			R_1+R_2\lesssim\epsilon^2 K_0^{|I|+|J|+1}.
		\end{equation}
		
		Now we estimate $R_3$ in~(\ref{est:energyE2}). If $1\leq |I|+|J|\leq N$, we have 
		$$|\Gamma^I\partial^JE|\lesssim\langle t+r\rangle\sum_{|I'|+|J'|<|I|+|J|}|\partial\Gamma^{I'}\partial^{J'}E|.$$
		
		Based on the above inequality, we get
		\begin{equation}\label{est:R3}
			\begin{aligned}
				&R_3\lesssim\int_{0}^{t}\int_{\R^{3}}|\langle s-r\rangle^{\frac{1}{2}+\delta}\partial_t n^0||\Gamma^{I}\partial^J E|^{\frac{1}{3}}|\Gamma^I\partial^J E|^{\frac{2}{3}}\left|\frac{\Gamma^I\partial^J E}{\langle s-r\rangle^{\frac{1}{2}+\delta}}\right|\d x\d s\\
				&\lesssim\sum_{\substack{|I'|+|J'|\\<|I|+|J|}}\int_{0}^{t}\int_{\R^{3}}|\langle s-r\rangle^{\frac{1}{2}+\delta}\langle s+r\rangle^{\frac{1}{3}}\partial_t n^0||\partial\Gamma^{I'}\partial^{J'} E|^{\frac{1}{3}} |\Gamma^I\partial^J E|^{\frac{2}{3}} \left|\frac{\Gamma^I\partial^J E}{\langle s-r\rangle^{\frac{1}{2}+\delta}}\right| \d x\d s\\
				&\lesssim\sum_{\substack{|I'|+|J'|\\<|I|+|J|}}\int_{0}^{t}\left\|\langle s-r\rangle^{\frac{1}{2}+\delta}\langle s+r\rangle^{\frac{1}{3}}\partial_t n^0\right\|_{L^\infty}\|\partial\Gamma^{I'}\partial^{J'} E\|^{\frac{1}{3}}\|\Gamma^I\partial^J E\|^{\frac{2}{3}}\left\|\frac{\Gamma^I\partial^J E}{\langle s-r\rangle^{\frac{1}{2}+\delta}}\right\|\ d s\\
				&\lesssim C_1^{1+(|I|+|J|-\frac{1}{3})\delta}\epsilon K_0\left(\int_{0}^{t}\langle s\rangle^{-\frac{4}{3}+2\delta}\d s\right)^{\frac{1}{2}}\left(\int_{0}^{t}\left\|\frac{\Gamma^I\partial^J E}{\langle s-r\rangle^{\frac{1}{2}+\delta}}\right\|^2\d s\right)^{\frac{1}{2}}\\
				&\lesssim C_1^{2+(2|I|+2|J|-\frac{1}{3})\delta}\epsilon^2 K_0,
			\end{aligned}
		\end{equation}
		in which we use the H\"{o}lder inequality and estimates in~(\ref{est:BootE2}),~(\ref{est:Gamma_n0point2}).
		
		For $R_4$, we first consider the case where $|I|\leq 8, 1\leq|I|+|J|\leq N-2$. Employing again the H\"{o}lder inequality and estimates in~(\ref{est:BootE2}),~(\ref{est:Gamma_n0point2}), we can deduce
		\begin{equation*}\label{est:R4-1}
			\begin{aligned}
				R_4&\lesssim\sum_{\substack{I_1+I_2=I,J_1+J_2=J\\|I_2|+|J_2|<|I|+|J|}}\int_{0}^{t}\|\langle s-r\rangle^{\frac{1}{2}+\delta}\Gamma^{I_1}\partial^{J_1}n^0\|_{L^\infty}\left\|\frac{\Gamma^{I_2}\partial^{J_2} E}{\langle s-r\rangle^{\frac{1}{2}+\delta}}\right\|\|\partial_t\Gamma^I\partial^J E\|\d s\\
				&\lesssim C_1^{1+(|I|+|J|)\delta}\epsilon K_0\left(\int_{0}^{t}\langle s\rangle^{-2+2\delta}\d s\right)^{\frac{1}{2}}\sum_{|I_2|+|J_2|<|I|+|J|}\left(\int_{0}^{t}\left\|\frac{\Gamma^{I_2}\partial^{J_2}E}{\langle s-r\rangle^{\frac{1}{2}+\delta}}\right\|^2\d s\right)^{\frac{1}{2}}\\
				&\lesssim C_1^{2+(2|I|+2|J|-1)\delta}\epsilon^2 K_0.
			\end{aligned}
		\end{equation*}
		Now let us deal with the case that $|I|=9,10,$ and $|I| \leq|I|+|J|\leq\max\{|I|,N-2\}$. From Lemmas~\ref{lem:n0-2},~\ref{lem:Elower2} and estimates in~(\ref{est:BootE2}), we have
		\begin{equation*}\label{est:R4-2}
			\begin{aligned}
				&R_4\lesssim\sum_{\substack{|I_1|\leq|I|\\|I_1|+|J_1|\leq \max\{|I|,N-2\}\\|I_2|+|J_2|\leq1}}\int_{0}^{t}\|\Gamma^{I_1}\partial^{J_1}n^0\|\|\Gamma^{I_2}\partial^{J_2}E\|_{L^\infty}\|\partial_t\Gamma^I\partial^J E\|\d s\\
				&+\sum_{\substack{|I_1|\leq8\\|I_1|+|J_1|\leq N-2\\|I_2|\leq |I|\\|I_2|+|J_2|<|I|+|J|}}\int_{0}^{t}\|\langle s-r\rangle^{\frac{1}{2}+\delta}\Gamma^{I_1}\partial^{J_1}n^0\|_{L^\infty}\left\|\frac{\Gamma^{I_2}\partial^{J_2}E}{\langle s-r\rangle^{\frac{1}{2}+\delta}}\right\|\|\partial_t\Gamma^I\partial^J E\|\d s\\
				&\lesssim\sum_{|I_2|+|J_2|\leq1}\int_{0}^{t} K_0 C_1^{1+(|I|+|J|)\delta}\epsilon\|\Gamma^{I_2}\partial^{J_2}E\|_{L^\infty}\d s+C_1^{2+(2|I|+2|J|-1)\delta}\epsilon^2 K_0\\
				&\lesssim C_1^{2+(2|I|+2|J|-4)\delta}\epsilon^2 K_0^4+C_1^{2+(2|I|+2|J|-1)\delta}\epsilon^2 K_0,\\
				%\lesssim&C_1^{2+(2|I|-1)\delta}\epsilon^2\langle M_1\rangle^4+C_1^{4+(|I|+32)\delta}\epsilon^4\langle M_1\rangle^3.
			\end{aligned}
		\end{equation*}
		where the H\"{o}lder inequality and $|I|+|J|\geq 9$ are used.
		
		Let $|I|\leq 10,\max\{N-1,|I|\}\leq |I|+|J|\leq N.$ From Lemmas~\ref{lem:n0-2},~\ref{lem:Elower2} and estimates in~(\ref{est:BootE2}), we have
		\begin{equation*}
			\begin{aligned}
				R_4&\lesssim\sum_{\substack{I_1+I_2=I,J_1+J_2=J\\|I_1|\leq10,|I_1|+|J_1|\leq N\\|I_2|+|J_2|\leq1}}\int_{0}^{t}\|\Gamma^{I_1}\partial^{J_1}n^0\|\|\Gamma^{I_2}\partial^{J_2}E\|_{L^\infty}\|\partial_t\Gamma^I\partial^J E\|\d s\\
				&+\sum_{\substack{I_1+I_2=I,J_1+J_2=J\\|I_1|\leq8\\|I_1|+|J_1|\leq N-2\\|I_2|\leq10\\|I_2|+|J_2|<|I|+|J|}}\int_{0}^{t}\|\langle s-r\rangle^{\frac{1}{2}+\delta}\Gamma^{I_1}\partial^{J_1}n^0\|_{L^\infty}\left\|\frac{\Gamma^{I_2}\partial^{J_2}E}{\langle s-r\rangle^{\frac{1}{2}+\delta}}\right\|\|\partial_t\Gamma^I\partial^J E\|\d s\\
				&\lesssim\sum_{|I_2|+|J_2|\leq1}\int_{0}^{t} K_0 C_1^{1+(|I|+|J|)\delta}\epsilon\|\Gamma^{I_2}\partial^{J_2}E\|_{L^\infty}\d s+C_1^{2+(2|I|+2|J|-1)\delta}\epsilon^2 K_0\\
				&\lesssim C_1^{2+(2|I|+2|J|-4)\delta}\epsilon^2K_0^4+C_1^{2+(2|I|+2|J|-1)\delta}\epsilon^2 K_0,
			\end{aligned}
		\end{equation*}
		where we use the fact $|I|+|J|\geq 9$ in the last step.
		
		Hence, from the above three estimates, we conclude that
		\begin{equation}\label{est:R4}
			R_4\lesssim C_1^{2+(2|I|+2|J|-4)\delta}\epsilon^2 K_0^4+C_1^{2+(2|I|+2|J|-1)\delta}\epsilon^2 K_0.
		\end{equation}
		Let $|I|\leq10, 1\leq|I|+|J|\leq N$. By the H\"{o}lder inequality and (\ref{est:commutators}), we obtain
		\begin{equation*}
			\begin{aligned}
				&R_5\lesssim\sum_{\substack{I_1+I_2=I,J_1+J_2=J\\|I_1|\leq10,|I_1|+|J_1|\leq N\\|I_2|\leq5,|I_2|+|J_2|\leq N-5}}\int_{0}^{t}\|\Gamma^{I_1}\partial^{J_1}\Delta n^1\|\|\Gamma^{I_2}\partial^{J_2}E\|_{L^\infty}\|\partial_t\Gamma^I\partial^J E\|\d s\\
				&+\sum_{\substack{I_1+I_2=I,J_1+J_2=J\\|I_1|+|J_1|\leq 4\\|I_2|\leq10,|I_2|+|J_2|\leq N}}\int_{0}^{t}\|\langle s-r\rangle^{\frac{1}{2}+\delta}\Gamma^{I_1}\partial^{J_1}\Delta n^1\|_{L^\infty}
				\left\|\frac{\Gamma^{I_2}\partial^{J_2}E}{\langle s-r\rangle^{\frac{1}{2}+\delta}}\right\|\|\partial_t\Gamma^I\partial^J E\|\d s\\
				&\lesssim R_{5}^1+R_{5}^2,\\
			\end{aligned}
		\end{equation*}
		where
		\begin{equation*}
			\begin{aligned}
				R_{5}^1&\lesssim\sum_{\substack{|I_1|+|I_2|\leq|I|\\|J_1|+|J_2|\leq|J|\\|I_1|\leq10,|I_1|+|J_1|\leq N\\|I_2|+|J_2|\leq N-5}}\int_{0}^{t}\|\partial\partial\Gamma^{I_1}\partial^{J_1}n^1\|\|\Gamma^{I_2}\partial^{J_2}E\|_{L^\infty}\|\partial_t\Gamma^I\partial^J E\|\d s,\\
				%P_{32}=&\sum_{\substack{|J|+|I_2|<|I|\\|J|<N,|I_2|\leq N-5}}\int_{t_0}^{t}\|\partial\partial\Gamma^{J}n^1\|_{L^2}\|\Gamma^{I_2}E\|_{L^\infty}\|\partial_t\Gamma^I E\|_{L^2}\d s\\
				R_{5}^2&\lesssim\sum_{\substack{|I_1|+|I_2|\leq|I|\\|J_1|+|J_2|\leq|J|\\|I_1|+|J_1|\leq 4\\|I_2|\leq10,|I_2|+|J_2|\leq N}}\int_{0}^{t}\|\langle s-r\rangle^{\frac{1}{2}+\delta}\partial\partial\Gamma^{I_1}\partial^{J_1}n^1\|_{L^\infty}
				\left\|\frac{\Gamma^{I_2}\partial^{J_2}E}{\langle s-r\rangle^{\frac{1}{2}+\delta}}\right\|\|\partial_t\Gamma^I\partial^J E\|\d s.\\
				%P_{34}=&\sum_{\substack{|J|+|I_2|<|I|\\|J|<N-5,|I_2|\leq N}}\int_{t_0}^{t}\|\langle s-r\rangle^{\frac{1}{2}+\delta}\partial\partial\Gamma^{J}n^1\|_{L^\infty}
				%\left\|\frac{\Gamma^{I_2}E}{\langle s-r\rangle^{\frac{1}{2}+\delta}}\right\|_{L^2}\|\partial_t\Gamma^I E\|_{L^2}\d s,\\
			\end{aligned}
		\end{equation*}
		By estimates in (\ref{est:BootE2}),~(\ref{est:energy_n1_2-2}), we see that
		\begin{equation*}
			\begin{aligned}
				R_{5}^1%&\lesssim\sum_{|I_1|+|I_2|\leq|I|}\left(\epsilon^2 K_0^{[\frac{|I_1|+|J_1|}{2}]}+C_1^{2+(|I_1|+|J_1|+14)\delta}\epsilon^2\right)C_1^{2+(|I|+|J|+|I_2|+|J_2|+13)\delta}\epsilon^2\\
				&\lesssim C_1^{2+(2|I|+2|J|+13)\delta}\epsilon^4 K_0^{[\frac{|I|+|J|}{2}]}+C_1^{4+(2|I|+2|J|+27)\delta}\epsilon^4.
			\end{aligned}
		\end{equation*}
		According to Lemma~\ref{lem:Hessian_n1point2} and (\ref{est:BootE2}), we observe that
		\begin{equation*}
			\begin{aligned}
				R_{5}^2
				&\lesssim\sum_{\substack{|I_1|+|I_2|\leq|I|\\|J_1|+|J_2|\leq|J|\\|I_1|+|J_1|\leq4}}C_1^{1+(|I|+|J|)\delta}\epsilon\left(\epsilon^2 K_0^{[\frac{|I_1|+|J_1|+3}{2}]}+C_1^{2+(|I_1|+|J_1|+26)\delta}\epsilon^2\right)\\
				&\cdot\left(\int_{0}^{t}\langle s\rangle^{-\frac{3}{2}}\d s\right)^{\frac{1}{2}}\left(\int_{0}^{t}\left\|\frac{\Gamma^{I_2}\partial^{J_2}E}{\langle s-r\rangle^{\frac{1}{2}+\delta}}\right\|^2\d s\right)^{\frac{1}{2}}\\
				&\lesssim C_1^{2+2(|I|+|J|)\delta}\epsilon^4 K_0^{3}+C_1^{4+(2|I|+2|J|+26)\delta}\epsilon^4.
			\end{aligned}
		\end{equation*}
		From the above estimates, we deduce
		\begin{equation}\label{est:R5}
			R_5\lesssim C_1^{2+(2|I|+2|J|+13)\delta}\epsilon^4 K_0^{[\frac{|I|+|J|}{2}]}+C_1^{2+2(|I|+|J|)\delta}\epsilon^4 K_0^{3}+C_1^{4+(2|I|+2|J|+27)\delta}\epsilon^4.
		\end{equation}
		Then by estimates in (\ref{est:energyE2}),~(\ref{est:R1R2}),~(\ref{est:R3}),~(\ref{est:R4}),~(\ref{est:R5}), we obtain
		\begin{equation*}
			\begin{aligned}
				\mathcal{E}_{gst,1}(t,\Gamma^I\partial^J E)
				&\lesssim\epsilon^2 K_0^{|I|+|J|+1}+C_1^{2+(2|I|+2|J|-\frac{1}{3})\delta}\epsilon^2 K_0\\
				&+C_1^{2+(2|I|+2|J|-4)\delta}\epsilon^2 K_0^4+C_1^{2+(2|I|+2|J|+13)\delta}\epsilon^4 K_0^{[\frac{|I|+|J|}{2}]}\\
				&+C_1^{2+2(|I|+|J|)\delta}\epsilon^4K_0^3+C_1^{4+(2|I|+2|J|+27)\delta}\epsilon^4.
			\end{aligned}
		\end{equation*}
		Finally, we take $C_1$ and $\epsilon$ satisfying
		\begin{equation*}
			\begin{aligned}
				&C_1^2\gg K_0^{N+1},\\
				&C_1^{\frac{1}{3}\delta}\gg K_0,\\
				&C_1^{3+31\delta}\epsilon^2\ll 1,
			\end{aligned}
		\end{equation*}
		which imply
		\begin{equation*}
			\mathcal{E}_{gst,1}(t,\Gamma^I\partial^J E)\leq\frac{1}{4}C_1^{2+2(|I|+|J|)\delta}\epsilon^2.
		\end{equation*}
		This strictly improves the estimate of $\mathcal{E}_{gst,1}(t,\Gamma^I\partial^JE)$ in~(\ref{est:BootE2}).
		
		\smallskip
		
		\textbf{Step 2.} Closing the estimate of $\Gamma^I\partial^J E$. 
		Let $|I|\leq 5, |I|+|J|\leq N-5$. By Lemma~\ref{lem:nonlinear2}, we conclude
		\begin{equation}
			\begin{aligned}
				&\sum_{|K|\leq 4}\|\langle t+r\rangle\Gamma^K\Gamma^I\partial^J(n^0E+\Delta n^1 E)\|\\
				&\lesssim\sum_{|I|\leq|K|\leq |I|+4}\|\langle t+r\rangle\Gamma^K\partial^J(n^0E+\Delta n^1 E)\|\\
				&\lesssim\langle t\rangle^{-\frac{1}{8}}\left(C_1^{1+(|I|+|J|+12)\delta}\epsilon K_0^2+C_1^{3+(|I|+|J|+43)\delta}\epsilon^3 K_0^{[\frac{|I|+|J|+10}{2}]}\right).
			\end{aligned}
		\end{equation}
		
		By Corollary~\ref{cor:KGpoint} and~(\ref{est:I.D.2}), we have
		\begin{equation}
			\begin{aligned}
				\langle t+r\rangle^{\frac{3}{2}}|\Gamma^I\partial^J E|
				&\lesssim\epsilon K_0^{[\frac{|I|+|J|+5}{2}]}+
				C_1^{1+(|I|+|J|+12)\delta}\epsilon K_0^2\\
				&+C_1^{3+(|I|+|J|+43)\delta}\epsilon^3 K_0^{[\frac{|I|+|J|+10}{2}]}.
			\end{aligned}
		\end{equation}
		Finally, we take $C_1$ and $\epsilon$ satisfying
		\begin{equation*}
			\begin{aligned}
				&C_1^2\gg K_0^{N+1},\\
				&C_1^{\frac{1}{3}\delta}\gg K_0,\\
				&C_1^{3+31\delta}\epsilon^2\ll 1,
			\end{aligned}
		\end{equation*}
		which imply
		\begin{equation*}
			\langle t+r\rangle^{\frac{3}{2}}|\Gamma^I\partial^J E|\leq\frac{1}{2}C_1^{1+(|I|+|J|+13)\delta}\epsilon.
		\end{equation*}
		This strictly improves the estimate of $\Gamma^I\partial^J E$ in~(\ref{est:BootE2}).
		
		In conclusion, the proof of Proposition~\ref{pro:KGZ-L2} is completed. Besides, estimates in~(\ref{est:pointwise2}) follow from~(\ref{est:BootE2}), (\ref{est:Gamma_n0point2}), (\ref{est:allHessian2}).
		
		Moreover, the solution $(E,n)$ in Theorem~\ref{thm:KGZ-L2} scatters linearly; see Section~\ref{SS:scatter} for details. The proof of Theorem~\ref{thm:KGZ-L2} is done.
	\end{proof}

	\section{Proof of Theorem~\ref{thm:DKG-L}} \label{S:proofdkg}
	
	In this section, we prove global stability result of Dirac-Klein-Gordon system via a continuity argument. Recall that in \eqref{eq:DKG-L2}, we rewrite the Klein-Gordon solution $v$ in \eqref{eq:DKG-L2} as $v= V^0 + V^1$, where  $V^0$ solves the linear Klein-Gordon equation 
	\begin{align}
		-\Box V^0 + V^0 = 0, \ \ \ V^0(0,x) =v_0, \ \ \partial_t V^0(0,x) =v_1,
	\end{align}
	and $V^1$ solves the nonlinear Klein-Gordon equation
	\begin{align}
		-\Box V^1 + V^1 =\psi^*\gamma^0 \psi , \ \ \ V^1(0,x) =0, \ \ \ \partial_t V^1(0,x) = 0.
	\end{align}
	In this way, the Dirac-Klein-Gordon system \eqref{eq:DKG-L} can be reformulated as 
	\begin{equation} \label{DKGequiva}
		\left\{
		\begin{aligned}
			-i \gamma^{\mu} \partial_{\mu} \psi & = (V^0 + V^1) \psi, \\
			-\Box V^0 + V^0 & =0, \\
			-\Box V^1 + V^1 &= \psi^*\gamma^0 \psi.
		\end{aligned}
		\right.
	\end{equation}
	The prescribed initial data at $t=0$ is 
	\begin{align}\label{DKGequinitialdata}
		(\psi, V^0, \partial_t V^0, V^1, \partial_t V^1)|_{t=0}= (\psi_0, v_0, v_1, 0, 0),
	\end{align}
	where $\psi_0, v_0, v_1$ is defined in \eqref{eq:DKG-ID}.

	\subsection{Bootstrap setting and auxiliary estimates}
	Let $N\in \mathbb{N}$ be a large integer and $0<\delta<1$ be a sufficiently small number. In the high regularity setting, the Cauchy problem \eqref{DKGequiva}-\eqref{DKGequinitialdata} is locally well-posed and  there exists some $T>0$ and  constant $C_1>0$, such that for  $0<t<T$, we have the following  estimates:
	\begin{equation} \label{priorienergy}
		\left \{
		\begin{aligned}
			\mathcal{E}_D^{\frac{1}{2}}(t,  \widehat{\Gamma}^I \psi) & \leq C_1 \epsilon^{1-\delta |I|}, \ \ \ |I| \leq N,   \\
			\mathcal{E}_{gst,1}^{\frac{1}{2}}(t,  \Gamma^I V^1) & \leq C_1 \epsilon^{2(1-N\delta)}, \ \ \  |I| \leq N,\\
			%E_{gst,1}^{\frac 12}(t,  \Gamma^I  V^0) & \leq C_1 K_0,   \ \ \quad   |I| \leq N, \\
			\|\langle r-t \rangle \chi(r-2t)  \widehat{\Gamma}^I \psi\| & \leq C_1 \epsilon^{1-\delta |I|}, \qquad \ \ |I| \leq N, \\
			\|\langle r-t \rangle \chi(r-2t)  \Gamma^I v \|  & \leq C_1 K_0, \qquad \ \ |I| \leq N, 
		\end{aligned}
		\right.
	\end{equation}
	and 
	\begin{equation} \label{priorimax}
		\left \{
		\begin{aligned}
			|\widehat{\Gamma}^I \psi(t,x)| & \leq  C_1 \epsilon^{1-\delta(|I|+5)} \langle t+r \rangle^{-\frac{3}{4}+\delta} \langle t-r \rangle^{-1}, \ \ |I| \leq N-5,  \\
			|[\widehat{\Gamma}^I \psi]_{-}| & \leq C_1 \epsilon^{1-\delta (|I|+5)} \langle t+r \rangle^{-\frac{7}{4}+\delta},  \ \  |I| \leq N-5,  \\
			|\Gamma^I V^1(t,x)| & \leq C_1 \epsilon^{(1-\delta N)2} \langle t+r \rangle^{-\frac 32}, \ \  |I|  \leq N- 7. \\
			%|\Gamma^I V^0(t,x)| & \leq C_1 K_0  \langle t+r \rangle^{-\frac 32}, \qquad \  \   |I|  \leq N- 5.
		\end{aligned}
		\right.
	\end{equation}
	We remind one that $v= V^0 + V^1$ and the definition of $[\psi]_{-}$ is given in \eqref{def:phi-}. Let 
	\begin{align}
		T^* = \sup\{T>0| \textrm{The estimates in \eqref{priorienergy} and \eqref{priorimax} hold for all } \ 0<t <T \}.
	\end{align}
	To prove Theorem \ref{thm:DKG-L}, it suffices to show $T^* = \infty$. In view of the method of continuity, the maximal existence time $T^*$ can not be finite  if we refine all the estimates in \eqref{priorienergy} and \eqref{priorimax}. To proceed, we present several estimates as a direct consequence of the bootstrap assumption.
	\begin{proposition} \label{prop:linearkgbou}
		Let \eqref{initikg} hold, then there exists some constant $C$, such that  
		\begin{align}
			\mathcal{E}_{gst,1}^{\frac 12}(t,  \Gamma^I  V^0) & \leq C K_0,   \ \ \quad   |I| \leq N, \label{est:linekge}\\
			|\Gamma^I V^0(t,x)| & \leq C K_0  \langle t+r \rangle^{-\frac 32}, \qquad \  \   |I|  \leq N- 5. \label{est:maxilinekg}
		\end{align} 
	\end{proposition}
	\begin{proof}
		First fix  $|I|\leq N$. Noting that 
		\begin{align} \label{eq:linearkge1}
			-\Box \Gamma^I V^0 + \Gamma^I V^0 = 0.
		\end{align}
		Then we apply the energy estimate \eqref{est:kgenergy} to derive 
		\begin{align}
			\mathcal{E}^{\frac 12}_{gst,1}(t, \Gamma^I V^0) \lesssim \mathcal{E}^{\frac 12}_{gst,1}(0, \Gamma^I V^0)\lesssim K_0.
		\end{align}
		This yields \eqref{est:linekge}.
		Next, we take $|I| \leq N- 5$. One can  apply Corollary \ref{cor:KGpoint}  to deduce 
		\begin{align}
			\langle t + r \rangle^{\frac 32} |\Gamma^I V^0(t,x)| \lesssim \sum_{|J|\leq 5} \|\langle x\rangle^{\frac 32} \log(2+|x|) \Gamma^J \Gamma^I V^0(0,x)\| \lesssim K_0.
		\end{align}
		This implies the desired result. The proof is completed.
	\end{proof}
	In \eqref{priorienergy} and \eqref{priorimax}, we first select $C_1>C$, where $C$ is given in Proposition \ref{prop:linearkgbou},  then by bootstrap assumption \eqref{priorienergy}, \eqref{priorimax} and Proposition  \ref{prop:linearkgbou}, we can get the following conclusions. 
	\begin{corollary}\label{l2energy} 
		For all $0<t<T^*$, we have the following estimates.
		\begin{equation*}
			\| \widehat{\Gamma}^I \psi\| + \Big( \int_0^t \bigg\|\frac{[\widehat{\Gamma}^I \psi]_{-}}{\langle s-r\rangle^{\frac 12 +\delta}} \bigg\|^2 \d s \Big)^{\frac 12} \lesssim C_1 \epsilon^{1-\delta |I|}, \quad  \forall \   |I| \leq N, 
		\end{equation*}
		\begin{align*}
			\|\partial  \Gamma^I v(t)\|  + \|\Gamma^I v(t)\| + \Big( \int_0^t \bigg\| \frac{\Gamma^I v}{\langle s-r \rangle^{\frac 12 +\delta}} \bigg\|^2 \d s \Big)^{\frac 12}   \lesssim C_1 K_0,  \quad  \forall \  |I| \leq N. 
		\end{align*} 
	\end{corollary}
	\begin{corollary} \label{maxiestimate} 
		For all $0<t<T^*$, we have the following estimates.
		\begin{align}
			| \widehat{\Gamma}^I \psi (t,x)| & \lesssim C_1 \epsilon^{1-\delta (|I|+3)} \langle t+r \rangle ^{-\frac34} , \ \ \forall  \ |I| \leq N-3, \label{est:prioripopsi} \\
			| \Gamma^I v(t,x)| & \lesssim  C_1 K_0 \langle t+ r \rangle^{-\frac 34}, \ \  \ \ |I|  \leq N-3, \label{priorikgpoint} \\
			|\Gamma^I v(t,x)| & \lesssim  C_1 K_0 \langle t + r \rangle^{-\frac 32}, \ \ \  |I| \leq N-7.  \label{est:optidecaykg} 
		\end{align} 
		
	\end{corollary}
	\begin{proof}
		To show \eqref{est:prioripopsi}, it suffices to apply \eqref{priorienergy},  \eqref{est:globalSobo} and the relation between $\widehat{\Gamma}$ and $\Gamma$.  For \eqref{priorikgpoint}, noting that $v = V^0 + V^1$, one can also  apply the decay estimate \eqref{est:globalSobo} in Lemma \ref{lem:Sobolev} to get the desired result. Finally, \eqref{est:optidecaykg} can be derived by combining the estimates of $V^0$ and $V^1$ in \eqref{priorimax}. The proof is completed.   
	\end{proof}
	Next, we deduce the decay results away from the light cone, which follow from the weighted $L^2$ estimates. For convenience, we denote 
	\begin{align}
		D_{int} & = \{ (t, x) \in [0,T^*)\times \mathbb{R}^3:  r \leq 3t +3  \}, \\
		D_{ext} & = \{ (t, x) \in [0,T^*)\times \mathbb{R}^3:  r\geq 2 t +3 \}.
	\end{align}
	\begin{corollary}\label{cor:ext}
		Let the bootstrap assumption \eqref{priorienergy}-\eqref{priorimax} hold. Then  for all $0 < t <T^*$, we have 
		\begin{align} 
			|\widehat{\Gamma}^I \psi 1_{D_{ext}}| & \lesssim C_1 \epsilon^{1-\delta(|I|+3)} \langle t + r \rangle^{-2}, \  \ \forall \ |I|\leq N-3, \label{est:pointexterdir}\\
			|\Gamma^{I} v  1_{D_{ext}}| & \lesssim C_1 K_0 \langle t + r \rangle^{-2}, \ \  \ \forall \ |I|\leq N-3. \label{est:pointexterkg}
		\end{align}
	\end{corollary}
	\begin{proof}
		In view of \eqref{est:standardSobo}, 
		\begin{align}
			|\langle r-t \rangle \chi(r-2t-1) \widehat{\Gamma}^I \psi | & \lesssim r^{-1} \sum_{k \leq 1, |J|\leq 2} \big\|\partial_r^k\Omega^{J} \big(\langle r-t \rangle \chi(r-2t-1) \widehat{\Gamma}^I \psi   \big)\big\| \nonumber \\
			&  \lesssim r^{-1} \sum_{ |J|\leq 2} \|\langle r - t \rangle (\chi + \chi' )(r-2t-1)\Omega^J  \widehat{\Gamma}^I \psi \| 		    \nonumber \\
			& + r^{-1} \sum_{ k\leq 1, |J|\leq 2} \|\langle r - t \rangle \chi(r-2t-1)\partial_r^k \Omega^J \widehat{\Gamma}^I \psi \|.
		\end{align}
		Noticing that $\chi(r-2t) = 1$ in the region $r \geq 2t +2$ and estimates in \eqref{priorienergy}, we can have 
		\begin{align}
			|\langle r-t \rangle \chi(r-2t-1) \widehat{\Gamma}^I \psi | & \lesssim  r^{-1} \sum_{ k\leq 1, |J|\leq 2} \|\langle r - t \rangle \chi(r-2t)\partial_r^k \Omega^J \widehat{\Gamma}^I \psi \| \nonumber \\
			& \lesssim r^{-1} C_1 \epsilon^{1-\delta(|I|+3)}.
		\end{align}
		Given that $r \sim \langle t + r \rangle \sim \langle t - r \rangle$ in the domain $D_{ext}$, we have 
		\begin{align}
			|\widehat{\Gamma}^I \psi 1_{D_{ext}} | \lesssim C_1 \epsilon^{1-\delta(|I|+3)} \langle t +r \rangle^{-2}.
		\end{align}
		This yields \eqref{est:pointexterdir}.  The proof of \eqref{est:pointexterkg} can be handled similarly and we omit the details.  
	\end{proof}
	
	\subsection{Conformal energy bounds for $\widetilde{\Psi}$} 
	
	Following from Lemma \ref{transwavedi} and \eqref{eq:waverepre}, we know $ \widetilde{\Psi} = \Psi  + v \psi $ and $ \widetilde{\Psi} $ solves 
	\begin{equation} \label{eq:modifiedwave}
		\left \{
		\begin{aligned}
			-\Box \widetilde{\Psi} & = (\psi^* \gamma^0 \psi) \psi + i \gamma^{\mu} v  \partial_{\mu}(v\psi) + 2 Q_0(v,\psi).  \\
			\widetilde{\Psi}|_{t=0} & = v_0 \psi_0, \ \partial_t\widetilde{\Psi}|_{t=0} = -i \gamma^0 \psi_0 + v_1 \psi_0 + v_0(-\gamma^0\gamma^a \partial_a\psi_0 + i \gamma^0 v_0 \psi_0),
		\end{aligned}
		\right.
	\end{equation}
	where we recall $\psi|_{t=0} = \psi_0$ and $(v,\partial_t v)|_{t=0} = (v_0, v_1)$.
	In this section, we aim at obtaining some proper bounds on conformal energy of $\widetilde{\Psi} $ and its lower order derivatives. For this purpose, we first treat the quadratic null term $Q_0(v, \psi)$. 
	\begin{proposition}\label{prop:nonlinear1}
		Suppose that the bootstrap assumption  \eqref{priorienergy}-\eqref{priorimax} hold and $N\geq 13$. Then for $  |I|  \leq N-2$ and $0\leq t<T^*$,  we have 
		\begin{align}
			\int_0^t \|\langle s+r \rangle \Gamma^I Q_0(v, \psi)\|\, \d s \lesssim C_1^3 K_0^2 \epsilon^{1-\delta(|I|+1)},
		\end{align}
		provided $C_1 \epsilon <1$ and $0<\delta <\frac{1}{N+2}$.
	\end{proposition}
	\begin{proof}
		Let $|I|\leq N-2$.
		We can compute 
		\begin{align}
			\|\langle s+r \rangle  \Gamma^I Q_0(v, \psi)\| & \leq \|\langle s+r \rangle  \Gamma^I Q_0(v, \psi)1_{D_{int}}\| \nonumber  \\
			& \qquad  + \|\langle s+r \rangle \Gamma^I Q_0(v, \psi)1_{D_{ext}}\|  \leq A + B.  \label{est:inoutsep}
		\end{align}
		Now we apply \eqref{est:GamQ0} and \eqref{est:Q0inside} to get 
		\begin{align}
			A & \lesssim \sum_{\substack{|I_1|+|I_2| \leq |I|}} \|\langle s + r \rangle Q_0(\Gamma^{I_1}v, \Gamma^{I_2}\psi)1_{D_{int}}\| \nonumber \\
			& \lesssim \sum_{\substack{|I_1|+|I_2| \leq |I|}} \|\langle s - r\rangle (\partial \Gamma^{I_1}v) \partial \Gamma^{I_2}\psi 1_{D_{int}} \| \nonumber \\
			& \qquad +   \sum_{\substack{|I_1|+|I_2| \leq |I|}} \|(\Gamma \Gamma^{I_1}v) \Gamma \Gamma^{I_2}\psi 1_{D_{int}}\| \lesssim A_1 + A_2.  \label{ineq:inter0}
		\end{align}
		{\bf Case 1.}  $6 \leq |I| \leq  N-2.$
		We have 
		\begin{align}
			A_1 & \lesssim \sum_{\substack{|I_1| \leq 5   \\ |I_2|  \leq  |I|}} \|\partial \Gamma^{I_1}v\|_{L^{\infty}} \|\langle s - r\rangle \partial \Gamma^{I_2}\psi 1_{D_{int}} \|  \\
			& + \sum_{\substack{|I_2| \leq |I|-6  \\ |I_1|  \leq  |I|}} \big \| (\partial \Gamma^{I_1}v) \langle s - r\rangle \partial \Gamma^{I_2}\psi 1_{D_{int}} \big \|  \lesssim A_1^1 + A_1^2.
		\end{align}
		Thanks to \eqref{est:hatparGa}, \eqref{est:optidecaykg} and Lemma \ref{le:Diracdecay},  one can deduce that 
		\begin{align}
			A_1^1 & \lesssim C_1 K_0 \langle s \rangle^{-\frac 32} \sum_{|I_2| \leq |I|} \|\langle s - r\rangle \partial \widehat{\Gamma}^{I_2}\psi 1_{D_{int}} \| \nonumber \\
			& \lesssim C_1 K_0 \langle s \rangle^{-\frac 32} \sum_{|I_2| \leq |I| } \big( \sum_{|J|\leq 1}\|\widehat{\Gamma}^J \widehat{\Gamma}^{I_2} \psi\| +  s\|\widehat{\Gamma}^{I_2}(v\psi)\|  \big).  \label{ineq:inter1}
		\end{align}
		By Lemma \ref{est:hatGfPhi} and Corollary \ref{maxiestimate}, we have 
		\begin{align}
			\sum_{|I_2| \leq |I|}\|\widehat{\Gamma}^{I_2}(v\psi)\| & \lesssim \sum_{\substack{|K_1| \leq |I|,|K_2| \leq |I|-2 }} \|\Gamma^{K_1} v\| \|\widehat{\Gamma}^{K_2} \psi\|_{L^{\infty}}  \\
			& + \sum_{|K_2| \leq |I|, |K_1| \leq |I|-2} \|\Gamma^{K_1} v\|_{L^{\infty}}\|\widehat{\Gamma}^{K_2} \psi\| \nonumber  \\
			& \lesssim C_1^2 K_0 \epsilon^{1- \delta(|I|+1)} \langle s \rangle^{-\frac 34}.
		\end{align}
		One can then insert the above bounds into \eqref{ineq:inter1} and use Corollary \ref{l2energy} to get 
		\begin{align}
			A_1^1  & \lesssim C_1^2 K_0 \epsilon^{1-\delta(|I|+1)} \langle s \rangle^{-\frac 32} + C_1^3K_0^2 \epsilon^{1-\delta(|I|+1)} \langle s \rangle^{-\frac 54}  \nonumber  \\
			& \lesssim C_1^3K_0^2 \epsilon^{1-\delta(|I|+1)} \langle s \rangle^{-\frac 54}.
		\end{align}
		Regarding $A_1^2$, we first note that
		\begin{align}
			|\partial \Gamma^{I_1} v 1_{D_{int}}| & \lesssim \frac{\langle s -r \rangle}{\langle s +r \rangle} \sum_{|J|\leq 1} |\partial \Gamma^J \partial \Gamma^{I_1} v| + |\partial \Gamma^{I_1} (\psi^* \gamma^0 \psi)|, \label{ineq:inter2} \\
			|\langle s -r \rangle \partial \Gamma^{I_2} \psi 1_{D_{int}}| & \lesssim \sum_{|J|\leq 1}|\widehat{\Gamma}^J \widehat{\Gamma}^{I_2} \psi| + s |\widehat{\Gamma}^{I_2} (v\psi)|. \label{ineq:inter3}
		\end{align}
		Following from \eqref{ineq:inter2} and Corollary \ref{l2energy}, we see 
		\begin{align}
			A_1^2 & \lesssim  \sum_{ |I_2| \leq |I|-6, |I_1| \leq |I|}  C_1 K_0\langle s \rangle^{-1}\|\langle s - r\rangle^2 \partial \Gamma^{I_2}\psi 1_{D_{int}}  \|_{L^{\infty}}  \nonumber  \\
			& +  \sum_{|I_1| \leq |I|, |I_2| \leq |I|-6} \|\partial \Gamma^{I_1} (\psi^* \gamma^0 \psi)\langle s - r\rangle \partial \Gamma^{I_2}\psi 1_{D_{int}}\|.
		\end{align}
		By \eqref{ineq:inter3}, \eqref{priorimax}, (\ref{priorikgpoint}) and $|I_2|\leq |I|-6$, we find 
		\begin{align}
			\|\langle s - r\rangle^2 \partial \Gamma^{I_2}\psi 1_{D_{int}}  \|_{L^{\infty}} 
			& \lesssim \sum_{|J|\leq 1}\|\langle s - r\rangle\widehat{\Gamma}^J \widehat{\Gamma}^{I_2}\psi \|_{L^{\infty}} + \|\langle s -r \rangle s \widehat{\Gamma}^{I_2}(v\psi)\|_{L^{\infty}}  \nonumber  \\
			&  \lesssim C_1 \epsilon^{1-\delta(|I_2|+1+5)} \langle s \rangle^{-\frac 34 + \delta} + C_1^2 K_0\epsilon^{1-\delta(|I_2|+5)}\langle s \rangle^{-\frac 12 + \delta}  \nonumber \\
			& \lesssim C_1^2 K_0 \epsilon^{1-\delta(|I_2|+6)} \langle s \rangle^{-\frac 12 + \delta}.  \label{ineq:linfty1}
		\end{align}
		On the one hand, 
		\begin{align} \label{ineq:l2bounds}
			\sum_{|I_1| \leq |I|} \|\partial \Gamma^{I_1} (\psi^* \gamma^0 \psi)\| \lesssim \sum_{|J_1|+|J_2| \leq |I|+1} \|\widehat{\Gamma}^{J_1} \psi \widehat{\Gamma}^{J_2} \psi\| \lesssim C_1^2 \epsilon^{2-\delta(|I|+4)} \langle s \rangle^{-\frac 34},
		\end{align}
		in which we use Corollaries \ref{l2energy} and \ref{maxiestimate}.  On the other hand,  by \eqref{ineq:inter3} and $|I_2|\leq |I|-3 \leq  N-4$, we have 
		\begin{align}
			\|\langle s - r\rangle \partial \Gamma^{I_2}\psi 1_{D_{int}}\|_{L^{\infty}} & \lesssim \sum_{|J|\leq 1}\|\widehat{\Gamma}^J \widehat{\Gamma}^{I_2} \psi \|_{L^{\infty}} +  \|s\widehat{\Gamma}^{I_2} (v\psi)\|_{L^{\infty}}  \nonumber \\
			& \lesssim C_1 \epsilon^{1-\delta(|I_2|+4)} \langle s \rangle^{-\frac 34}  + C_1^2 K_0\epsilon^{1-\delta(|I_2| +3)} \langle s \rangle^{-\frac 12}  \nonumber \\
			& \lesssim C_1^2 K_0\epsilon^{1-\delta(|I_2| +4)} \langle s \rangle^{-\frac 12}.  \label{ineq:inter4} 
		\end{align}
		Gathering the estimates in \eqref{ineq:l2bounds} and \eqref{ineq:inter4}, one can get
		\begin{align}
			& \sum_{|I_1| \leq |I|, |I_2| \leq |I|-6} \|\partial \Gamma^{I_1} (\psi^* \gamma^0 \psi)\langle s - r\rangle \partial \Gamma^{I_2}\psi 1_{D_{int}}\| \nonumber \\
			& \lesssim \sum_{|I_1| \leq |I|, |I_2| \leq |I|-3} \|\partial \Gamma^{I_1} (\psi^* \gamma^0 \psi)\| \| \langle s - r\rangle \partial \Gamma^{I_2}\psi 1_{D_{int}}\|_{L^{\infty}}  \nonumber \\
			& \lesssim C_1^4 K_0\epsilon^{2-\delta(|I|+4)} \epsilon^{1-\delta(|I| +1)}    \langle s \rangle^{-\frac 54}. \label{ineq:inter5}
		\end{align}
		Owing to bounds in \eqref{ineq:linfty1} and \eqref{ineq:inter5},  we obtain 
		\begin{align}
			A_1^2 & \lesssim C_1^3 K_0^2 \epsilon^{1-\delta(|I| +1)} \langle s \rangle^{-\frac 32 + \delta} + C_1^4 K_0\epsilon^{2-\delta(|I|+4)} \epsilon^{1-\delta(|I| +1)}    \langle s \rangle^{-\frac 54}  \\
			&  \lesssim C_1^3 K_0^2 \epsilon^{1-\delta(|I| +1)} \langle s \rangle^{-\frac 54},
		\end{align}
		where we choose $\epsilon$ small enough such that $C_1 \epsilon <1$ and $0<\delta < 1/(N+2)$. Collecting the estimates of $A_1^1$ and $A_1^2$, one can find 
		\begin{align}
			A_1 \lesssim  C_1^3 K_0^2 \epsilon^{1-\delta(|I| +1)} \langle s \rangle^{-\frac 54}.
		\end{align}
		Next, we turn to the estimate of $A_2$. Indeed, 
		\begin{align}
			A_2 & \lesssim \sum_{|I_1| \leq |I|-6, |I_2| \leq |I|} \|\Gamma \Gamma^{I_1} v\|_{L^{\infty}} \|\widehat{\Gamma}\widehat{\Gamma}^{I_2} \psi\|  \nonumber \\
			& \qquad \qquad + \sum_{|I_2| \leq |I|-6, |I_1|\leq |I|} \|(\Gamma \Gamma^{I_1} v)\widehat{\Gamma}\widehat{\Gamma}^{I_2} \psi 1_{D_{int}}\|.
		\end{align}
		According to \eqref{est:KG3}, 
		\begin{align}
			|\Gamma \Gamma^{I_1} v 1_{D_{int}}| \lesssim \frac{\langle s -r \rangle}{ \langle s + r  \rangle} \sum_{|J|\leq 1} |\partial \Gamma^J \Gamma \Gamma^{I_1} v| + \big|\Gamma\Gamma^{I_1} (\psi^*\gamma^0\psi ) \big|.
		\end{align}
		Then one can compute 
		\begin{align}
			\|(\Gamma \Gamma^{I_1} v)\widehat{\Gamma}\widehat{\Gamma}^{I_2} \psi 1_{D_{int}}\| & \lesssim \langle s \rangle^{-1} \sum_{|J|\leq 1} \|\partial \Gamma^J \Gamma \Gamma^{I_1} v  \|   \|\langle s -r \rangle \widehat{\Gamma}\widehat{\Gamma}^{I_2} \psi 1_{D_{int}}\|_{L^{\infty}}  \nonumber \\
			& \qquad \qquad \qquad +  \|\Gamma\Gamma^{I_1} (\psi^*\gamma^0\psi ) \widehat{\Gamma}\widehat{\Gamma}^{I_2} \psi \|  \nonumber \\
			&  \lesssim C_1 K_0 \langle s \rangle^{-1} \|\langle s -r \rangle \widehat{\Gamma}\widehat{\Gamma}^{I_2} \psi 1_{D_{int}}\|_{L^{\infty}} \nonumber \\
			& \qquad \qquad + \sum_{|J|\leq |I_1|+1} \|\widehat{\Gamma}^J(\psi^* \gamma^0 \psi)\| \|\widehat{\Gamma}\widehat{\Gamma}^{I_2} \psi \|_{L^{\infty}}.
		\end{align}
		By the estimates in \eqref{priorimax}, Corollaries \ref{l2energy} and \ref{maxiestimate}, we get 
		\begin{align}
			& \sum_{|I_2| \leq |I|-6, |I_1|\leq |I|} \|(\Gamma \Gamma^{I_1} v)\widehat{\Gamma}\widehat{\Gamma}^{I_2} \psi 1_{D_{int}}\|  \nonumber \\
			& \lesssim C_1^2 K_0 \epsilon^{1-\delta|I|} \langle s \rangle^{-\frac 74 + \delta} +  \sum_{|I_2| \leq |I|-6, |I_1|\leq |I|}C_1^3 \epsilon^{2-\delta(|I_1|+4)} \epsilon^{1-\delta(|I_2|+4)} \langle s \rangle^{-\frac 32}  \nonumber \\
			& \lesssim C_1^2 K_0 \epsilon^{1-\delta |I|} \langle s \rangle^{-\frac 32},  \label{ineq:inter6}
		\end{align}
		in which we use the fact $C_1 \epsilon <1$ and $0<\delta < 1/(N+2)$.  Applying the  estimates \eqref{est:optidecaykg}, Corollary \ref{l2energy} and \eqref{ineq:inter6},  we get 
		\begin{align}
			A_2 \lesssim C_1^2 K_0  \epsilon^{1-\delta(|I|+1)}  \langle s \rangle^{-\frac 32}.
		\end{align}
		Gathering the bounds of $A_1$ and $A_2$, and using \eqref{ineq:inter0}, one can see 
		\begin{align}
			A  \lesssim  C_1^3 K_0^2 \epsilon^{1-\delta(|I| +1)} \langle s \rangle^{-\frac 54}.
		\end{align}
		Concerning the estimate of $B$, we have 
		\begin{align}
			B & \lesssim \sum_{|I_1|+|I_2| \leq |I|} \big\|\langle s + r \rangle \big(\partial \Gamma^{I_1} v \big) \partial \widehat{\Gamma}^{I_2} \psi 1_{D_{ext}}  \big \| \nonumber \\
			& \lesssim \sum_{|I_1| \leq |I|-3, |I_2| \leq |I|} \|\partial \Gamma^{I_1} v 1_{D_{ext}}\|_{L^{\infty}} \|\langle s - r \rangle \partial \widehat{\Gamma}^{I_2} \psi 1_{D_{ext}} \|  \nonumber \\
			& +  \sum_{|I_2| \leq |I|-3, |I_1| \leq |I|} \|\langle s - r \rangle \partial \Gamma^{I_1} v 1_{D_{ext}}  \| \|\partial \widehat{\Gamma}^{I_2} \psi 1_{D_{ext}}\|_{L^{\infty}},
		\end{align}
		in which we use the fact $|I|-2 \leq N-3$ and $\langle s + r \rangle \sim \langle s - r \rangle$ in the region $D_{ext}$. Recall the definition of the cut-off function $\chi$,  one can deduce 
		\begin{align}
			\|\langle s - r \rangle \partial \widehat{\Gamma}^{I_2} \psi 1_{D_{ext}} \| & \lesssim \|\langle s - r \rangle \chi(r-2s)\partial \widehat{\Gamma}^{I_2} \psi  \|, \\
			\|\langle s - r \rangle \partial \Gamma^{I_1} v 1_{D_{ext}}  \| & \lesssim \|\langle s - r \rangle \chi(r-2s)\partial \Gamma^{I_1} v   \|. 
		\end{align}
		As a consequence, by estimates in \eqref{priorienergy} and \eqref{priorimax}, Corollary~\ref{cor:ext}, we can show 
		\begin{align}
			B \lesssim C_1^2 K_0 \epsilon^{1-\delta(|I|+1)} \langle s \rangle^{-2}.
		\end{align}
		By adding the bounds of $A$ and $B$, we arrive 
		\begin{align}
			\|\langle s+r \rangle  \Gamma^I Q_0(v, \psi)\| & \lesssim C_1^3 K_0^2 \epsilon^{1-\delta(|I| +1)} \langle s \rangle^{-\frac 54} + C_1^2 K_0 \epsilon^{1-\delta(|I|+1)} \langle s \rangle^{-2} \nonumber\\
			& \lesssim C_1^3 K_0^2 \epsilon^{1-\delta(|I| +1)} \langle s \rangle^{-\frac 54}.  \label{est:largederiv}
		\end{align}
		
		{\bf Case 2.} Let  $|I| \leq 5$. The main point is that one can always take $L^{\infty}$ of Klein-Gordon solution $v$. Recall \eqref{est:inoutsep} and \eqref{ineq:inter0}, one can see 
		\begin{align}
			A_1 \lesssim \sum_{|I_1|+|I_2| \leq |I|} \|\partial \Gamma^{I_1} v\|_{L^{\infty}} \|\langle s -r \rangle \partial \widehat{\Gamma}^{I_2} \psi  1_{D_{int}}\|. 
		\end{align}
		Due to Lemma \ref{le:Diracdecay} and Corollaries \ref{l2energy} and \ref{maxiestimate}, 
		\begin{align}
			\big\|\langle s - r \rangle \partial \widehat{\Gamma}^{I_2} \psi  1_{D_{int}}\big\| & \lesssim \|\widehat{\Gamma} \widehat{\Gamma}^{I_2} \psi \| + \|s\widehat{\Gamma}^{I_2}(v \psi)\| \nonumber\\
			& \lesssim C_1 \epsilon^{1-\delta(|I_2|+1)} + C_1^2 K_0 \epsilon^{1-\delta|I_2|} \langle s \rangle^{-\frac 12}.
		\end{align} 
		Consequently, 
		\begin{align}
			A_1 & \lesssim C_1 K_0 \langle s \rangle^{-\frac 32} (C_1 \epsilon^{1-\delta(|I|+1)} + C_1^2 K_0 \epsilon^{1-\delta|I|} \langle s \rangle^{-\frac 12}) \nonumber \\
			& \lesssim C_1^3 K_0^2 \epsilon^{1-\delta(|I|+1)} \langle s \rangle^{-\frac 32}.
		\end{align}
		On the other hand, 
		\begin{align}
			A_2 &\lesssim \sum_{|I_1| + |I_2| \leq |I|} \|\Gamma \Gamma^{I_1} v\|_{L^{\infty}} \|\widehat{\Gamma}\widehat{\Gamma}^{I_2} \psi\| \nonumber \\
			& \lesssim C_1^2K_0 \epsilon^{1-\delta(|I|+1)} \langle s \rangle^{-\frac 32}.
		\end{align}
		Combining the bounds of $A_1$ and $A_2$, we obtain
		\begin{align}
			A \lesssim C_1^3 K_0^2 \epsilon^{1-\delta(|I|+1)} \langle s \rangle^{-\frac 32}.
		\end{align}
		With regard to the estimate of $B$, one can find
		\begin{align}
			B & \lesssim \sum_{|I_1|+|I_2| \leq |I|} \big\|\langle s + r \rangle \big(\partial \Gamma^{I_1} v \big) \partial \widehat{\Gamma}^{I_2} \psi 1_{D_{ext}}  \big \| \nonumber \\
			& \lesssim \sum_{|I_1|+|I_2| \leq |I|} \|\partial \Gamma^{I_1} v 1_{D_{ext}}\|_{L^{\infty}} \|\langle r- s \rangle  \chi(r-2s)\partial \widehat{\Gamma}^{I_2} \psi  \|  \nonumber\\
			& \lesssim C_1^2 K_0 \epsilon^{1-\delta(|I|+1)} \langle s \rangle^{-2}.
		\end{align}
		Adding the estimates of $A$ and $B$, we get 
		\begin{align}
			\|\langle s+r \rangle  \Gamma^I Q_0(v, \psi)\| & \lesssim  C_1^3 K_0^2 \epsilon^{1-\delta(|I|+1)} \langle s \rangle^{-\frac 32}.  \label{est:smallderi}
		\end{align}
		In summary, for $|I| \leq N-2$, collecting the estimates \eqref{est:largederiv} and \eqref{est:smallderi}, we finally  obtain 
		\begin{align}
			\|\langle s+r \rangle  \Gamma^I Q_0(v, \psi)\| & \lesssim  C_1^3 K_0^2 \epsilon^{1-\delta(|I| +1)} \langle s \rangle^{-\frac 54}. 
		\end{align}
		Hence, 
		\begin{align}
			\int_0^t \|\langle s+r \rangle \Gamma^I Q_0(v, \psi)\|ds \lesssim C_1^3 K_0^2 \epsilon^{1-\delta(|I|+1)}.
		\end{align}
		This immediately yields the desired result. The proof is done.
	\end{proof}
	
	\begin{proposition} \label{prop:nonlinear2}
		Suppose that the bootstrap assumption \eqref{priorienergy}-\eqref{priorimax} hold and $N\geq 13$. Then for $|I| \leq N-2$ and $0\leq t<T^*$,  we have 
		\begin{align}
			\int_0^t \big\|\langle s + r \rangle \Gamma^I \big[(\psi^* \gamma^0 \psi) \psi  \big]   \big\|\, \d s \lesssim C_1^3  \epsilon^{1-\delta|I|} \langle t \rangle^{\delta},
		\end{align}
		provided $0<\epsilon<1, 0<\delta <\frac{1}{N+2}$.
	\end{proposition}
	\begin{proof} Let $|I|\leq N-2$.
		First, by \eqref{est:hatGa}, one can see 
		\begin{align}
			\big\|\langle s + r \rangle \Gamma^I \big[(\psi^* \gamma^0 \psi) \psi  \big]   \big\| &\lesssim \sum_{|J|\leq |I|} \|\langle s + r \rangle \widehat{\Gamma}^J \big[(\psi^* \gamma^0 \psi) \psi  \big] \| \nonumber \\
			& \lesssim \sum_{|J_1|+|J_2|\leq |I|} \big\|\langle s + r \rangle \widehat{\Gamma}^{J_1}(\psi^*\gamma^0\psi) \widehat{\Gamma}^{J_2} \psi \big\| \nonumber \\
			& \lesssim \sum_{|K_1|+|K_2|+|K_3|\leq |I|} \|\langle s + r \rangle [\widehat{\Gamma}^{K_1} \psi]_{-} \widehat{\Gamma}^{K_2} \psi \widehat{\Gamma}^{K_3}\psi \| := R.  \nonumber
		\end{align}
		Furthermore, 
		\begin{align}
			R &\lesssim \sum_{ \substack{|K_1|,|K_2| \leq N-7, |K_3|\leq |I|\\{|K_1|+|K_2|+|K_3|\leq |I|}}} \|\langle s + r\rangle [\widehat{\Gamma}^{K_1}\psi]_{-} \|_{L^{\infty}}\|\widehat{\Gamma}^{K_2} \psi\|_{L^{\infty}} \|\widehat{\Gamma}^{K_3} \psi\| \nonumber \\
			& + \sum_{\substack{|K_2|, |K_3| \leq N-7,|K_1| \leq |I| \\|K_1|+|K_2|+|K_3|\leq |I| }}  \bigg\| \frac{[\widehat{\Gamma}^{K_1}\psi]_{-}}{\langle s - r\rangle^{\frac 12 + \delta}} \bigg\| \|\langle s - r\rangle^{\frac 12 + \delta} \langle s + r\rangle^{\frac{1}{2}} \widehat{\Gamma}^{K_2} \psi\|_{L^{\infty}} \|\langle s+r\rangle^{\frac{1}{2}}\widehat{\Gamma}^{K_3} \psi\|_{L^{\infty}}  \nonumber\\
			& \lesssim C_1^3 \epsilon^{2-\delta(N+6)} \epsilon^{1-\delta|I|} \langle s \rangle^{-\frac 32 +\delta}+ C_1^2 \epsilon^{2-\delta(N+6)} \langle s \rangle^{-\frac 12 + \delta} \sum_{|K_1|\leq |I|}\bigg\| \frac{[\widehat{\Gamma}^{K_1}\psi]_{-}}{\langle s - r\rangle^{\frac 12 + \delta}} \bigg\|, \nonumber
		\end{align}
		in which \eqref{priorienergy}, \eqref{priorimax}, (\ref{est:prioripopsi}) are used to get the last inequality.  Hence, one can show
		\begin{align}
			\int_0^t \big\|\langle s + r \rangle \Gamma^I \big[(\psi^* \gamma^0 \psi) \psi  \big]   \big\| \d s & \lesssim C_1^3 \epsilon^{2-\delta(N+6)} \epsilon^{1-\delta|I|} \nonumber  \\
			& \qquad + C_1^3 \epsilon^{2-\delta(N+6)} \epsilon^{1-\delta|I|} \Big(\int_0^t \langle s \rangle^{-1+2\delta} \d s\Big)^{\frac{1}{2}} \nonumber\\
			& \lesssim C_1^3  \epsilon^{1-\delta|I|} \langle t \rangle^{\delta},
		\end{align}
		where we used the assumption \eqref{priorienergy}, $0<\epsilon<1$ and $0<\delta < 1/(N+2)$. The proof is completed.
	\end{proof}
	
	\begin{proposition}\label{prop:nonlinear3}
		Suppose that the bootstrap assumption \eqref{priorienergy}-\eqref{priorimax} hold and $N\geq 13$. Then for $  |I|  \leq N-2$ and $0\leq t<T^*$,  we have 
		\begin{align}
			\int_0^t \big\|\langle s + r \rangle \Gamma^I \big(i \gamma^{\mu} v \partial_{\mu}(v\psi)  \big) \big\|\, \d s \lesssim C_1^3K_0^2 \epsilon^{1-\delta(|I|+1)}.
		\end{align}
	\end{proposition}  
	\begin{proof}
		Fix $|I|\leq N-2$. Following from Leibniz rule and Lemma \ref{est:hatGfPhi}, one can obtain
		\begin{align}
			\big\|\langle s + r \rangle \Gamma^I \big(i \gamma^{\mu} v \partial_{\mu}(v\psi)  \big) \big\| & \lesssim \sum_{|I_1|+|I_2|+|I_3|\leq |I|} \big\|\langle s + r \rangle \Gamma^{I_1} v \Gamma^{I_2}  \partial v \widehat{\Gamma}^{I_3} \psi  \big\| \nonumber \\
			& + \sum_{|I_1|+|I_2|+|I_3|\leq |I|} \big\|\langle s + r \rangle \Gamma^{I_1} v \Gamma^{I_2}   v \widehat{\Gamma}^{I_3} \partial\psi  \big\|: =R_1 + R_2. \nonumber
		\end{align}
		Next, we bound $R_1$ and $R_2$ separately. 
		
		{\bf Case 1.} $ 6 \leq |I| \leq N-2 $.  One can find 
		\begin{align}
			R_1 & \lesssim \sum_{|I_2|,|I_1|\leq 5,|I_3|\leq |I| } \|\langle s + r\rangle \Gamma^{I_1}v  \|_{L^{\infty}} \|\Gamma^{I_2} \partial v\|_{L^{\infty}} \|\widehat{\Gamma}^{I_3} \psi\| \nonumber\\
			& \qquad + \sum_{|I_1|,|I_3| \leq |I|-5,|I_2|\leq |I|+1 } \|\Gamma^{I_2}v\| \|\langle s + r\rangle \Gamma^{I_1}v\|_{L^{\infty}} \|\widehat{\Gamma}^{I_3} \psi\|_{L^{\infty}}.  \nonumber 
		\end{align}
		Thanks to Corollaries \ref{l2energy}, \ref{maxiestimate} and $|I|-6 \leq N-8$,  we can further bound $R_1$ as 
		\begin{align}
			R_1 &\lesssim C_1^3 K_0^2 \epsilon^{1-\delta|I|} \langle s \rangle^{-2} + C_1^3 K_0^2 \langle s \rangle^{-\frac{5}{4}} \sum_{|I_3|\leq |I|-5} \epsilon^{1-\delta(|I_3|+3)} \nonumber \\
			& \lesssim  C_1^3 K_0^2 \epsilon^{1-\delta|I|} \langle s \rangle^{-\frac 54}. \label{est:2vcubic1}
		\end{align}
		Concerning the estimate of $R_2$, one can have 
		\begin{align}
			R_2 &\lesssim \sum_{\substack{|I_3|\leq |I|\\ |I_1|,|I_2| \leq 5}} \|\widehat{\Gamma}^{I_3}\partial \psi\|\|\langle s + r \rangle \Gamma^{I_1}  v\|_{L^{\infty}} \|\Gamma^{I_2} v\|_{L^{\infty}}  \nonumber\\
			& + \sum_{\substack{|I_1|\leq |I|\\ |I_2|,|I_3| \leq |I|-5}} \|\widehat{\Gamma}^{I_3} \partial\psi\|_{L^{\infty}} \|\langle s + r \rangle \Gamma^{I_2}  v\|_{L^{\infty}} \|\Gamma^{I_1} v\|  \nonumber\\
			& \lesssim C_1^3 K_0^2 \epsilon^{1-\delta(|I|+1)} \langle s \rangle^{-2} + C_1^3K_0^2 \epsilon^{1-\delta(|I|-1)} \langle s \rangle^{-\frac 54} \nonumber\\
			& \lesssim C_1^3K_0^2 \epsilon^{1-\delta(|I|+1)} \langle s \rangle^{-\frac 54}.  \label{est:2vcubic2}
		\end{align}
		Gathering the estimates of $R_1$ and $R_2$, we get 
		\begin{align}
			\big\|\langle s + r \rangle \Gamma^I \big(i \gamma^{\mu} v \partial_{\mu}(v\psi)  \big) \big\| & \lesssim C_1^3K_0^2 \epsilon^{1-\delta(|I|+1)} \langle s \rangle^{-\frac 54}. \label{est:2vcubic3}
		\end{align}
		
		{\bf Case 2.} $|I| \leq 5$. In this case, we can always take $L^{\infty}$ norm of Klein-Gordon solution $v$. Indeed, one can see 
		\begin{align}
			R_1 & \lesssim \sum_{|I_3|\leq |I|, |I_2|+|I_1|\leq |I|} \|\langle s + r\rangle \Gamma^{I_1}v  \|_{L^{\infty}} \|\Gamma^{I_2} \partial v\|_{L^{\infty}} \|\widehat{\Gamma}^{I_3} \psi\| \nonumber\\
			& \lesssim C_1^3 K_0^2\epsilon^{1-\delta|I|} \langle s \rangle^{-2}. \label{est:2vcubic4}
		\end{align}
		While for $R_2$, we have 
		\begin{align}
			R_2 & \lesssim \sum_{\substack{|I_3|\leq |I|\\ |I_1|+|I_2| \leq |I|}} \|\widehat{\Gamma}^{I_3} \partial\psi\|\|\langle s + r \rangle \Gamma^{I_1}  v\|_{L^{\infty}} \|\Gamma^{I_2} v\|_{L^{\infty}} \nonumber\\
			& \lesssim C_1^3 K_0^2\epsilon^{1-\delta(|I|+1)} \langle s \rangle^{-2}. \label{est:2vcubic5}
		\end{align}
		The estimates in \eqref{est:2vcubic4} and \eqref{est:2vcubic5} then lead to 
		\begin{align}
			\big\|\langle s + r \rangle \Gamma^I \big(i \gamma^{\mu} v \partial_{\mu}(v\psi)  \big) \big\| & \lesssim C_1^3K_0^2 \epsilon^{1-\delta(|I|+1)} \langle s \rangle^{-2}.  \label{est:2vcubic6}
		\end{align}
		In summary, for $|I| \leq N-2$, it follows from \eqref{est:2vcubic3} and \eqref{est:2vcubic6} that 
		\begin{align}
			\big\|\langle s + r \rangle \Gamma^I \big(i \gamma^{\mu} v \partial_{\mu}(v\psi)  \big) \big\| & \lesssim C_1^3K_0^2 \epsilon^{1-\delta(|I|+1)} \langle s \rangle^{-\frac 54},
		\end{align}
		which immediately yields 
		\begin{align}
			\int_0^t \big\|\langle s + r \rangle \Gamma^I \big(i \gamma^{\mu} v \partial_{\mu}(v\psi)  \big) \big\| \d s \lesssim  C_1^3K_0^2 \epsilon^{1-\delta(|I|+1)}.
		\end{align}
		The proof is done.
	\end{proof}
	With Propositions \ref{prop:nonlinear1}-\ref{prop:nonlinear3} at hand, we can now present the conformal energy estimates for $\widetilde{\Psi}$. 
	\begin{proposition}
		Suppose that the bootstrap assumption \eqref{priorienergy}-\eqref{priorimax} hold and $N\geq 13$. Then for $  |I|  \leq N-2 $ and $0\leq t<T^*$,  we have 
		\begin{align}
			\mathcal{E}^{1/2}(t,\Gamma^I \widetilde{\Psi}) +  \mathcal{E}_{con}^{1/2}(t, \Gamma^I \widetilde{\Psi}) & \lesssim \epsilon K_0^{N} + C_1^3 K_0^2 \epsilon^{1-\delta(|I|+1)}\langle t \rangle^{\delta} , \label{est:confortilpsi}\\
			\mathcal{E}^{1/2}(t,\Gamma^I \Psi)+ \mathcal{E}_{con}^{1/2}(t, \Gamma^I \Psi) & \lesssim \epsilon K_0^{N} + C_1^3 K_0^2 \epsilon^{1-\delta(|I|+1)}\langle t \rangle^{\delta}, \label{est:conforpsi}  
		\end{align}
		provided $C_1 \epsilon <1, 0<\delta <\frac{1}{N+2}$.
	\end{proposition}
	\begin{proof}
		Fix $|I| \leq N-2$.  We first treat \eqref{est:confortilpsi}. 
		Recall that $\widetilde{\Psi}$ is the solution to the Cauchy problem \eqref{eq:modifiedwave},  	we rely on Lemma \ref{lem:staconene} to obtain 
		\begin{align}
			& \mathcal{E}^{1/2}(t,\Gamma^I \widetilde{\Psi}) + \mathcal{E}_{con}^{1/2}(t, \Gamma^{I} \widetilde{\Psi}) \nonumber\\
			& \lesssim \mathcal{E}_{con}^{1/2}(0, \Gamma^{I} \widetilde{\Psi}) + \mathcal{E}^{1/2}(0, \Gamma^{I} \widetilde{\Psi})  \nonumber \\
			&  + \int_0^t \big\|\langle s+r \rangle \Gamma^I \big((\psi^* \gamma^0 \psi) \psi   + i \gamma^{\mu} v \partial_{\mu}(v\psi)+ Q_0(v, \psi)  \big)\big\|\,ds  \nonumber \\
			& \lesssim \epsilon K_0^N  + C_1^3 K_0^2 \epsilon^{1-\delta(|I|+1)}\langle t \rangle^{\delta},
		\end{align} 
		in which we have used the assumption \eqref{initidirac},\eqref{initikg} and Propositions \ref{prop:nonlinear1}-\ref{prop:nonlinear3}. 
		
		Next, we show \eqref{est:conforpsi}. Recall that 
		\begin{align}
			\widetilde{\Psi} = \Psi + v\psi.
		\end{align}
		By a direct computation, one can get 
		\begin{align}
			& \|\Gamma^{I} (v\psi)\| + \|\partial \Gamma^I(v\psi)\|+  \|L_0 \Gamma^I(v\psi)\| + \sum_{1\leq a \leq 3} \|L_a \Gamma^{I}(v\psi)\| + \sum_{1\leq a <b\leq 3} \|\Omega_{ab}\Gamma^I (v\psi)\| \nonumber \\
			& \lesssim \|\Gamma^{I} (v \psi)\| + \sum_{|J|\leq |I|} \|\langle t + r \rangle \partial\widehat{\Gamma}^J(v\psi) \|.  \label{est:differ0}
		\end{align}
		{\bf Case 1.} $6 \leq |I| \leq N-2$. By Corollaries  \ref{l2energy} and \ref{maxiestimate}, we have 
		\begin{align}
			\|\Gamma^I (v\psi)\| & \lesssim \sum_{\substack{|I_1| \leq |I|\\ |I_2| \leq |I|-3}} \|\Gamma^{I_1} v\| \|\widehat{\Gamma}^{I_2} \psi\|_{L^{\infty}} + \sum_{\substack{|I_2| \leq |I|\\ |I_1| \leq |I|-3}} \|\Gamma^{I_1} v\|_{L^{\infty}}  \|\widehat{\Gamma}^{I_2}\psi\| \nonumber \\
			& \lesssim C_1^2 K_0 \epsilon^{1-\delta|I|}.  \label{est:differ1}
		\end{align}
		In order to bound $\|\langle t + r \rangle \partial \widehat{\Gamma}^J(v\psi) \|$, we use Lemma \ref{est:hatGfPhi} to obtain
		\begin{align}
			\sum_{|J|\leq |I|}\|\langle t + r \rangle \partial \widehat{\Gamma}^J(v\psi) \| &\lesssim \sum_{|J_1|+|J_2| \leq |I|+1} \|\langle t + r \rangle  \Gamma^{J_1}v \widehat{\Gamma}^{J_2}\psi\|  \nonumber  \\
			& \lesssim \sum_{\substack{|J_1| \leq 6 \\ |J_2| \leq |I|+1}} \|\langle t + r \rangle \Gamma^{J_1} v  \|_{L^{\infty}} \|\widehat{\Gamma}^{J_2}\psi \| \nonumber \\
			& + \sum_{\substack{|J_2| \leq |I|-5 \\ |J_1| \leq |I|+1}} \|\langle t + r \rangle \Gamma^{J_1} v \widehat{\Gamma}^{J_2}\psi \|   := A_1 + A_2. 
		\end{align}
		Due to estimates in Corollaries \ref{l2energy} and \ref{maxiestimate}, one can see 
		\begin{align}
			A_1 & \lesssim C_1^2 K_0 \epsilon^{1-\delta(|I|+1)}.
		\end{align}
		On the other hand, by  \eqref{est:KG3}, 
		\begin{align}
			| \Gamma^{J_1} v| \lesssim \frac{\langle t - r \rangle}{\langle t + r \rangle}\sum_{|J|\leq 1}|\partial \Gamma^J \Gamma^{J_1} v| + |\Gamma^{J_1}(\psi^*\gamma^0\psi)|.  \label{est:hessianv}
		\end{align}
		Inserting \eqref{est:hessianv} into $A_2$, one can show 
		\begin{align}
			A_2 & \lesssim \sum_{\substack{|J_2| \leq |I|-5\\|J_1| \leq |I|+2}} \| \partial \Gamma^{J_1} v\|\|\langle t - r \rangle \widehat{\Gamma}^{J_2} \psi\|_{L^{\infty}} \nonumber \\
			& + \sum_{\substack{|J_2| \leq |I|-5\\|J_1| \leq |I|+1}} \|  \langle t+r\rangle^{\frac{1}{2}}\Gamma^{J_1} (\psi^*\gamma^0\psi)\|\|\langle t + r \rangle^{\frac{1}{2}} \widehat{\Gamma}^{J_2} \psi\|_{L^{\infty}} \nonumber \\
			& \lesssim C_1^2 K_0 \epsilon^{1-\delta|I|} + 	\sum_{\substack{|J_1^2|\leq |I|-3,|J_2| \leq |I|-5\\|J_1^1| \leq |I|+1}} \|  \widehat{\Gamma}^{J_1^1} \psi\| \|\langle t + r \rangle^{\frac 12}\widehat{\Gamma}^{J_1^2} \psi\|_{L^{\infty}}  \|\langle t + r \rangle^{\frac 12} \widehat{\Gamma}^{J_2} \psi\|_{L^{\infty }}  \nonumber\\
			& \lesssim  C_1^2 K_0 \epsilon^{1-\delta|I|} + C_1^3  \epsilon^{1-\delta(|I|+1)} \epsilon^{2-2\delta |I|}   \lesssim C_1^3K_0 \epsilon^{1-\delta(|I|+1)}.
		\end{align}
		Adding the estimates of $A_1$ and $A_2$, we find
		\begin{align}
			\|\langle t + r \rangle \partial \widehat{\Gamma}^I(v\psi) \| &\lesssim C_1^3K_0\epsilon^{1-\delta(|I|+1)}. \label{est:differ2}
		\end{align}
		{\bf Case 2.} $|I| \leq 5.$ As before, we can always take $L^{\infty}$ norm on Klein-Gordon component $v$. Let us treat $\|\Gamma^I (v\psi)\|$, by Corollaries \ref{l2energy} and \ref{maxiestimate}, one can see 
		\begin{align}
			\|\Gamma^I(v\psi)\| &\lesssim \sum_{|I_1| + |I_2|\leq |I|} \|\Gamma^{I_1} v\|_{L^{\infty}} \|\widehat{\Gamma}^{I_2} \psi\|  \nonumber \\
			& \lesssim C_1^2 K_0 \epsilon^{1-\delta |I|}.  \label{est:differ3}
		\end{align}
		As to $\|\langle t + r \rangle \partial \widehat{\Gamma}^I(v\psi) \|$, similarly we have 
		\begin{align}
			\|\langle t + r \rangle \partial \widehat{\Gamma}^I(v\psi) \|& \lesssim \sum_{|J_1|+|J_2| \leq |I|+1} \|\langle t +r \rangle \Gamma^{J_1} v\|_{L^{\infty}} \|\widehat{\Gamma}^{J_2} \psi\| \nonumber\\
			& \lesssim C_1^2 K_0 \epsilon^{1-\delta(|I|+1)}.  \label{est:differ4}
		\end{align}
		In summary, for any $|I| \leq N-2$, it follows from \eqref{est:differ0}, \eqref{est:differ1}, \eqref{est:differ2} and \eqref{est:differ3}, \eqref{est:differ4} that 
		\begin{align}
			& \|\Gamma^{I} (v\psi)\| + \|\partial \Gamma^I(v\psi)\|+ \|L_0 \Gamma^I(v\psi)\| + \sum_{1\leq a\leq 3} \|L_a \Gamma^{I}(v\psi)\| + \sum_{1\leq a <b\leq 3} \|\Omega_{ab}\Gamma^I (v\psi)\| \nonumber \\
			& \lesssim  C_1^3K_0\epsilon^{1-\delta(|I|+1)}.
		\end{align}
		As a consequence, 
		\begin{align}
			& \mathcal{E}^{1/2}(t,\Gamma^I \Psi) +\mathcal{E}_{con}^{1/2}(t, \Gamma^I \Psi)\nonumber \\
			& \lesssim \mathcal{E}_{con}^{1/2}(t, \Gamma^I \widetilde{\Psi}) + \|\Gamma^{I} (v\psi)\| + \|L_0 \Gamma^I(v\psi)\| 
			+ \|\partial \Gamma^I(v\psi)\|  \nonumber\\ 
			& +\mathcal{E}^{1/2}(t, \Gamma^I \widetilde{\Psi}) + \sum_{a=1}^3 \|L_a \Gamma^{I}(v\psi)\| + \sum_{1\leq a <b\leq 3} \|\Omega_{ab}\Gamma^I (v\psi)\| \nonumber \\
			& \lesssim \epsilon K_0^{N}+ C_1^3K_0^2\epsilon^{1-\delta(|I|+1)}\langle t \rangle^{\delta}.
		\end{align}
		The proof is complete.
	\end{proof}
	Based on the energy estimates, we can derive the pointwise estimates. 
	\begin{corollary}\label{cor:pointscal}
		Suppose that the bootstrap assumption \eqref{priorienergy}-\eqref{priorimax} hold and $N\geq 13$. Let $|I| \leq N-5$ and $0\leq t<T^*$,  we have 
		\begin{align}
			|L_0 \Gamma^I \Psi| + |\Gamma \Gamma^I \Psi| \lesssim (\epsilon K_0^{N} + C_1^3 K_0^2 \epsilon^{1-\delta(|I|+4)}) \langle t + r \rangle^{-\frac 34+\delta},
		\end{align}
		provided $C_1 \epsilon <1, 0<\delta <\frac{1}{N+2}$.
		\begin{proof}
			By \eqref{est:globalSobo} and commutator estimates in Lemma \ref{lem:commutators}, one can get 
			\begin{align}
				|L_0\Gamma^I \Psi| & \lesssim  \langle t + r \rangle^{-\frac 34}\sum_{|J|\leq 3} \|\Gamma^J L_0 \Gamma^I\Psi\| \nonumber\\
				& \lesssim \langle t + r \rangle^{-\frac 34} \sum_{|J|\leq  3} \big( \|L_0\Gamma^J  \Gamma^I\Psi\|+\|\Gamma^J \Gamma^I\Psi\| \big) \nonumber\\
				& \lesssim  (\epsilon K_0^{N} + C_1^3 K_0^2 \epsilon^{1-\delta(|I|+4)}) \langle t + r \rangle^{-\frac 34+\delta}, 
			\end{align}
			where we use \eqref{est:conforpsi} in the last inequality. One can treat $\Gamma \Gamma^I \Psi$ similarly and the details are omitted. The proof is done.
		\end{proof}
	\end{corollary}

	\subsection{Refined  estimates  for Dirac solution $\psi$} In this subsection, we first apply the pointwise decay for $\Psi$ to establish corresponding  decay for $\psi$. Then we use the $L^2$ bounds and pointwise estimates to close the energy estimates for Dirac solution. Following \eqref{eq:diffwave}, we have 
	\begin{align} \label{est:Gaparf}
		i \gamma^{\mu} \partial_{\mu} \widehat{\Gamma}^I \Psi = \widehat{\Gamma}^I \psi.
	\end{align}
	Fix  $|I| \leq N-5$,  \eqref{est:Gaparf} and Corollary \ref{cor:pointscal} can  imply 
	\begin{align}
		|\widehat{\Gamma}^I \psi(t,x)|  \lesssim |\partial \widehat{\Gamma}^I \Psi| &\lesssim  \sum_{|J|\leq |I|} |\partial {\Gamma}^J \Psi| \nonumber\\
		& \lesssim \sum_{|J|\leq |I|} \langle t - r \rangle^{-1}\big( |L_0 {\Gamma}^J \Psi| + |\Gamma {\Gamma}^J \Psi|\big) \nonumber \\
		& \lesssim  (\epsilon K_0^{N} + C_1^3 K_0^2 \epsilon^{1-\delta(|I|+4)}) \langle t + r \rangle^{-\frac 34+\delta} \langle t - r \rangle^{-1}. 
	\end{align}
	On the other hand, applying Lemma \ref{lem:ghostrepre} and (\ref{est:partial}), we have 
	\begin{align}
		|[\widehat{\Gamma}^I \psi]_{-} | \lesssim \sum_{a=1}^3 |G_a \widehat{\Gamma}^I \Psi| & \lesssim \langle t +r \rangle^{-1} \sum_{|J|\leq |I|}\big( |L_0\Gamma^J \Psi| + |\Gamma \Gamma^J \Psi| \big) \nonumber\\
		& \lesssim (\epsilon K_0^{N} + C_1^3 K_0^2 \epsilon^{1-\delta(|I|+4)}) \langle t +r \rangle^{-\frac 74 +\delta}.
	\end{align}

	Now we summarize the above results as a proposition.
	\begin{proposition} \label{refdiracpo}
		Suppose that the bootstrap assumption \eqref{priorienergy}-\eqref{priorimax} hold, $N\geq 13$ and  $|I| \leq N-5$,  then for all $0\leq t<T^*$, we have
		\begin{align}
			|\widehat{\Gamma}^I \psi(t,x)| & \leq C (\epsilon K_0^{N} + C_1^3 K_0^2 \epsilon^{1-\delta(|I|+4)}) \langle t + r \rangle^{-\frac 34+\delta} \langle t - r \rangle^{-1}, \\
			|[\widehat{\Gamma}^I \psi]_{-}(t,x)| &\leq C(\epsilon K_0^{N} + C_1^3 K_0^2 \epsilon^{1-\delta(|I|+4)}) \langle t +r \rangle^{-\frac 74 +\delta},
		\end{align}
		where $C_1 \epsilon <1, 0<\delta <\frac{1}{N+2}$ and $C$ is some  constant independent of $\epsilon$ and $K_0$.
	\end{proposition}
	Next, we refine the energy estimates for $\psi$; the following conclusion holds.
	\begin{proposition} \label{refdiracener}
		Suppose that the bootstrap assumption \eqref{priorienergy}-\eqref{priorimax} hold, $N\geq 13$ and $|I|\leq N$. Then for all $0\leq t<T^*$, we have 
		\begin{align}
			\mathcal{E}^{\frac 12}_{D}(t,\widehat{\Gamma}^I\psi) & \leq C (\epsilon K_0^N + C_1^{\frac 32}K_0^{\frac 12}\epsilon^{\frac 12\delta}\epsilon^{1-\delta|I|}),  \label{ineq:rediracenerg} \\
			\big\|\langle r - t \rangle \chi(r-2t) \widehat{\Gamma}^I \psi \big\| & \leq C(\epsilon K_0^N + C_1^{\frac 32}K_0^{\frac 12}\epsilon^{\frac 12\delta} \epsilon^{1-\delta|I|}),  \label{ineq:outdiracener}
		\end{align}
		where $C_1 \epsilon <1, 0<\delta <\frac{1}{N+2}$ and $C$ is some  constant independent of $\epsilon$ and $K_0$.
	\end{proposition}
	\begin{proof}
		We first show \eqref{ineq:rediracenerg}.
		Recall that 
		\begin{align} \label{eq:derivdirac1}
			-i \gamma^{\mu}\partial_{\mu} \widehat{\Gamma}^I \psi = \widehat{\Gamma}^I(v\psi).
		\end{align}
		By \eqref{est:Energyphi2}, we have 
		\begin{align}\label{dice}
			\mathcal{E}_{D}(t, \widehat{\Gamma}^I \psi) &\lesssim \mathcal{E}_{D}(0, \widehat{\Gamma}^I \psi) + \int_0^t \|(\widehat{\Gamma}^I \psi)^* \gamma^0 \widehat{\Gamma}^I(v\psi)  - (\widehat{\Gamma}^I(v\psi))^* \gamma^0 \widehat{\Gamma}^I \psi\|_{L^1} \, \d s.
		\end{align}
		
		In case of $|I|=0$, one can find 
		\begin{align}
			\mathcal{E}_{D}(t, \psi) \lesssim  \mathcal{E}_{D}(0,\psi) \lesssim \epsilon^2.
		\end{align}
		Next, we take $|I| \geq 1$. Following \eqref{dice} and Lemma \ref{est:hatGfPhi}, one can obtain
		\begin{align} \label{ineq:diracene}
			\mathcal{E}_{D}(t,\widehat{\Gamma}^I\psi) & \lesssim \mathcal{E}_{D}(0, \widehat{\Gamma}^I \psi) + \sum_{\substack{|I_1|+|I_2|\leq |I|\\|I_2|<|I|}} \int_0^t \big\|(\Gamma^{I_1}v) (\widehat{\Gamma}^I\psi)^* \gamma^0 \widehat{\Gamma}^{I_2} \psi  \big\|_{L^1} \,\d s.
		\end{align}
		In case $|I|\geq 7$, one can show 
		\begin{align}
			&\sum_{|I_1|+|I_2|\leq |I|,|I_2|<|I|} \big\|(\Gamma^{I_1}v) (\widehat{\Gamma}^I\psi)^* \gamma^0 \widehat{\Gamma}^{I_2} \psi  \big\|_{L^1} \nonumber\\
			&\lesssim \|\widehat{\Gamma}^I \psi\| \sum_{|I_1| \leq 6,|I_2|<|I|} \|\Gamma^{I_1} v\|_{L^{\infty}} \|\widehat{\Gamma}^{I_2}\psi\|  \nonumber \\
			& + \|\widehat{\Gamma}^I \psi\| \sum_{|I_2| \leq |I|-7, |I_1|\leq |I|} \bigg\|\frac{\Gamma^{I_1}v}{\langle r -s \rangle^{1/2+\delta}}  \bigg\| \big\|\langle s-r \rangle^{\frac 12+\delta} \widehat{\Gamma}^{I_2}\psi  \big\|_{L^{\infty}} \nonumber\\
			&\lesssim C_1\epsilon^{1-\delta|I|} \epsilon^{1-\delta(|I|-1)} \bigg(C_1^2 K_0  \langle s \rangle^{-\frac 32}  +C_1 \langle s \rangle^{-\frac 34 +\delta} \sum_{|I_1| \leq |I|}\bigg\|\frac{\Gamma^{I_1}v}{\langle r -s \rangle^{1/2+\delta}}  \bigg\| \bigg), \nonumber
		\end{align}
		in which we used the  assumption \eqref{priorienergy}-\eqref{priorimax} and Corollaries~\ref{l2energy},\ref{maxiestimate}. Inserting the above estimate into \eqref{ineq:diracene}, we can derive 
		\begin{align}
			\mathcal{E}_{D}(t,\widehat{\Gamma}^I\psi) & \lesssim \mathcal{E}_{D}(0, \widehat{\Gamma}^I \psi) + C_1^3 K_0 \epsilon^{\delta} \epsilon^{2(1-\delta|I|)} \nonumber\\
			& + C_1^2\epsilon^{\delta} \epsilon^{2(1-\delta|I|)} \sum_{|I_1| \leq |I|} \int_0^t \langle s \rangle^{-\frac 34 +\delta}\bigg\|\frac{\Gamma^{I_1}v}{\langle r -s \rangle^{1/2+\delta}}  \bigg\|\, \  \d s  \nonumber\\
			& \lesssim (\epsilon K_0^N)^2 + C_1^3 K_0 \epsilon^{\delta} \epsilon^{2(1-\delta|I|)}, \label{diracbd1}
		\end{align}
		where we used \eqref{initidirac}, \eqref{initikg}  and  Corollary \ref{l2energy} in the last step.  \\
		On the other hand, in case of $1\leq |I| \leq 6$, by Corollaries \ref{l2energy},\ref{maxiestimate}, one can have 
		\begin{align}
			&\sum_{|I_1|+|I_2|\leq |I|,|I_2|<|I|} \big\|(\Gamma^{I_1}v) (\widehat{\Gamma}^I\psi)^* \gamma^0 \widehat{\Gamma}^{I_2} \psi  \big\|_{L^1} \nonumber\\
			& \lesssim \sum_{|I_1|+|I_2|\leq |I|,|I_2|<|I|} \|\widehat{\Gamma}^{I} \psi\| \|\Gamma^{I_1}v\|_{L^{\infty}} \|\widehat{\Gamma}^{I_2} \psi\| \nonumber\\
			& \lesssim C_1^3 K_0 \epsilon^{1-\delta|I|} \epsilon^{1-\delta(|I|-1)}\langle s\rangle^{-\frac{3}{2}}. \label{ineq:lowderdirac}
		\end{align}
		Now inserting \eqref{ineq:lowderdirac} into \eqref{ineq:diracene}, we get 
		\begin{align}
			\mathcal{E}_{D}(t,\widehat{\Gamma}^I\psi) & \lesssim (\epsilon K_0^N)^2 + C_1^3 K_0 \epsilon^{\delta} \epsilon^{2(1-\delta|I|)}. \label{diracbd2}
		\end{align}
		In any case, taking the square root of  \eqref{diracbd1} and \eqref{diracbd2}, we can get 
		\begin{align}
			\mathcal{E}_{D}^{\frac 12}(t,\widehat{\Gamma}^I\psi) \lesssim \epsilon K_0^N + C_1^{\frac 32}K_0^{\frac 12}\epsilon^{\frac 12\delta} \epsilon^{1-\delta|I|}.
		\end{align}
		This yields \eqref{ineq:rediracenerg}.  \\
		Next, we treat \eqref{ineq:outdiracener}. Applying energy estimates \eqref{est:cutenergyphi} to Dirac equation \eqref{eq:derivdirac1}, one can see
		\begin{align}
			& \|\langle r - t\rangle \chi(r-2t) \widehat{\Gamma}^I \psi\|^2 \lesssim \|\langle r \rangle\widehat{\Gamma}^{I} \psi(0) \|^2 \nonumber\\
			& \qquad   + \int_0^t \int_{\R^3}\langle r -s \rangle^2 \chi(r-2s)^2 |(\widehat{\Gamma}^I \psi)^* \gamma^0 \widehat{\Gamma}^I(v\psi)- (\widehat{\Gamma}^I(v\psi))^* \gamma^0 \widehat{\Gamma}^I \psi  |\,\d x \d s.
		\end{align} 
		Observing that if $|I|=0$, we have 
		\begin{align}
			\|\langle r - t\rangle \chi(r-2t)  \psi\| \lesssim \|\langle r \rangle\psi(0) \| \lesssim \epsilon.
		\end{align}
		Now we consider $|I| \geq 1$. By a direct calculation and \eqref{initidirac}, \eqref{initikg},  we find 
		\begin{align}
			& \|\langle r - t\rangle \chi(r-2t) \widehat{\Gamma}^I \psi\|^2    \nonumber\\
			&\lesssim  (\epsilon K_0^N)^2 + \sum_{\substack{|I_1|+|I_2|\leq |I|\\ |I_2|<|I|}} \int_0^t \int_{\R^3}\langle r -s \rangle^2 \chi(r-2s)^2 |\Gamma^{I_1}v (\widehat{\Gamma}^{I} \psi)^* \gamma^0 \widehat{\Gamma}^{I_2}\psi|\,\d x \d s, \nonumber\\
			&\lesssim (\epsilon K_0^N)^2 + \sum_{\substack{|I_1|+|I_2|\leq |I|\\ |I_2|<|I|}} \int_0^t \|\langle r -s \rangle\chi(r-2s) \widehat{\Gamma}^I \psi \| \|\langle r -s \rangle\chi(r-2s) \Gamma^{I_1}v \widehat{\Gamma}^{I_2} \psi \|\,\d  s \nonumber\\
			&\lesssim (\epsilon K_0^N)^2  + C_1 \epsilon^{1-\delta|I|}\sum_{\substack{|I_1|+|I_2|\leq |I|\\ |I_2|<|I|}} \int_0^t \ \|\langle r -s \rangle\chi(r-2s) \Gamma^{I_1}v \widehat{\Gamma}^{I_2} \psi \|\,\d s, \label{ineq:outener1}
		\end{align}
		where we used \eqref{priorienergy} in the last inequality.  To proceed, we consider two cases. \\
		\textbf{Case 1.} $|I| \geq 7$, by \eqref{priorienergy}, \eqref{priorimax} and \eqref{est:pointexterdir}, Corollary~\ref{maxiestimate}, one can see 
		\begin{align}
			& \sum_{\substack{|I_1|+|I_2|\leq |I|\\ |I_2|<|I|}} \|\langle r -s \rangle\chi(r-2s) \Gamma^{I_1}v \widehat{\Gamma}^{I_2} \psi \|  \nonumber\\
			& \lesssim \sum_{|I_1| \leq 6, |I_2|<|I|} \|\Gamma^{I_1}v\|_{L^{\infty}} \|\langle r -s \rangle\chi(r-2s) \widehat{\Gamma}^{I_2} \psi\| \nonumber\\
			& + \sum_{|I_2| \leq |I|-7, |I_1|\leq |I|} \|1_{r \geq 2s+1} \widehat{\Gamma}^{I_2}\psi \|_{L^{\infty}} \|\langle r -s \rangle\chi(r-2s) \Gamma^{I_1}v \| \nonumber\\
			& \lesssim C_1^2 K_0 \langle s \rangle^{-\frac 32} \epsilon^{1-\delta(|I|-1)} + C_1^2 K_0 \epsilon^{1-\delta(|I|-2)} \langle s \rangle^{-\frac{7}{4}+\delta} \nonumber\\
			& \lesssim C_1^2 K_0 \langle s \rangle^{-\frac 32} \epsilon^{1-\delta(|I|-1)}, \label{ineq:outener2}
		\end{align}
		where we used the fact $0<\delta < 1/4$. \\
		\textbf{Case 2.} $1\leq |I| \leq 6$.  Applying \eqref{priorienergy}, \eqref{priorimax} again, we see 
		\begin{align}
			&\sum_{\substack{|I_1|+|I_2|\leq |I|\\ |I_2|<|I|}} \|\langle r -s \rangle\chi(r-2s) \Gamma^{I_1}v \widehat{\Gamma}^{I_2} \psi \|  \nonumber\\ 
			& \lesssim \sum_{|I_1|+|I_2| \leq |I|, |I_2|<|I|} \|\Gamma^{I_1}v\|_{L^{\infty}} \|\langle r -s \rangle\chi(r-2s) \widehat{\Gamma}^{I_2} \psi\| \nonumber\\
			&\lesssim C_1^2 K_0 \langle s \rangle^{-\frac 32} \epsilon^{1-\delta(|I|-1)}.  \label{ineq:outener3}
		\end{align}
		In summary, for any $1 \leq |I| \leq N$, by \eqref{ineq:outener1}-\eqref{ineq:outener3}, we have 
		\begin{align}
			& \|\langle r - t\rangle \chi(r-2t) \widehat{\Gamma}^I \psi\|^2    \nonumber\\
			& \lesssim (\epsilon K_0^N)^2 + C_1^3 K_0 \epsilon^{\delta} \epsilon^{2(1-\delta|I|)}.  \label{ineq:outener4}
		\end{align}
		The desired result \eqref{ineq:outdiracener} then follows by taking the square root of \eqref{ineq:outener4}.
	\end{proof}

	\subsection{Refined estimates for Klein-Gordon solution $v$} We recall that $v= V^0 + V^1$, where $V^0$ solves
	\begin{equation} \label{eq:kge1}
		\left \{
		\begin{aligned}
			&-\Box V^0 + V^0 = 0, \\
			&V^0(0,x) = v_0, \ \partial_t V^0(0,x)= v_1,
		\end{aligned}	
		\right.
	\end{equation}
	and $V^1$ solves
	\begin{equation}\label{eq:kge2}
		\left \{
		\begin{aligned}
			&-\Box V^1 + V^1 = \psi^* \gamma^0 \psi, \\
			& V^1(0,x) =0, \ \partial_t V^1(0,x) = 0.
		\end{aligned}
		\right.
	\end{equation}
	Next, we perform weighted energy estimates  for $V^0$ and $V^1$. 
	\begin{proposition} \label{refkgener}
		Suppose that the bootstrap assumption \eqref{priorienergy}-\eqref{priorimax} hold and $N\geq 13$, then for any $|I|\leq N$ and $0\leq t<T^*$, we have 
		\begin{align}
			\mathcal{E}^{\frac 12}_{gst,1}(t, \Gamma^I V^1) & \leq C(\epsilon^2 K_0^N + C_1^2 \epsilon^{\delta} \epsilon^{2(1-\delta N)}),  \ \  \label{ineq:kgbound2} \\
			\|\langle r -t \rangle \chi(r-2t) \Gamma^I v\| & \leq C (K_0 + C_1^2 \epsilon^{\delta} \epsilon^{2(1-\delta N)}),
			\label{ineq:outenergnkg}
		\end{align}
		where $C$ is some constant independent of $\epsilon$ and $K_0$ and $K_0 \epsilon <1$.
	\end{proposition}
	\begin{proof}
		One can see 
		\begin{align} \label{eq:nlnkge2}
			-\Box \Gamma^I V^1 + \Gamma^I V^1 = \Gamma^I (\psi^* \gamma^0 \psi).
		\end{align}
		Applying \eqref{est:kgenergy}, we get 
		\begin{align}
			\mathcal{E}_{gst,1}^{\frac 12}(t,\Gamma^I V^1) &\lesssim \mathcal{E}_{gst,1}^{\frac 12}(0,\Gamma^I V^1) +\int_0^t \|\Gamma^I(\psi^*\gamma^0\psi)\|\, \d s  \label{ineq:energykgb}.
		\end{align}
		Using Leibniz rule and Lemmas \ref{est:derinull}, \ref{lem:hidden}, one can find 
		\begin{align}
			\|\Gamma^I (\psi^* \gamma^0\psi)\| &\lesssim \sum_{|I_1|+|I_2| \leq |I|} \|[\widehat{\Gamma}^{I_1} \psi]_{-} \widehat{\Gamma}^{I_2}\psi\|\nonumber \\
			& \lesssim \sum_{|I_1|\leq N-6, |I_2| \leq N}\|[\widehat{\Gamma}^{I_1}\psi]_{-}\|_{L^{\infty}} \|\widehat{\Gamma}^{I_2}\psi\|\nonumber\\
			& + \sum_{|I_1|\leq N,|I_2|\leq N-6}\bigg\|\frac{[\widehat{\Gamma}^{I_1}\psi]_{-}}{\langle s -r \rangle^{\frac 12 +\delta}}  \bigg\| \|\langle s -r \rangle^{\frac 12 +\delta}  \widehat{\Gamma}^{I_2} \psi\|_{L^{\infty}} \nonumber\\
			& \lesssim C_1^2 \epsilon^{\delta} \epsilon^{2(1-\delta N)} \langle s \rangle^{-\frac 74 +\delta} + C_1 \epsilon^{\delta} \epsilon^{1-\delta N} \langle s\rangle^{-\frac 34 +\delta} \sum_{|I_1|\leq N}\bigg\|\frac{[\widehat{\Gamma}^{I_1}\psi]_{-}}{\langle s -r \rangle^{\frac 12 +\delta}}  \bigg\|, \nonumber
		\end{align}
		in which we use \eqref{priorienergy} and \eqref{priorimax}. Inserting the above bound into \eqref{ineq:energykgb}, we can get 
		\begin{align}
			& \mathcal{E}_{gst,1}^{\frac 12}(t,\Gamma^I V^1) \nonumber \\
			&\lesssim \epsilon^2 K_0^N + C_1^2 \epsilon^{\delta} \epsilon^{2(1-\delta N)} + C_1 \epsilon^{\delta} \epsilon^{1-\delta N}\sum_{|I_1|\leq N}\int_0^t \langle s\rangle^{-\frac 34 +\delta}\bigg\|\frac{[\widehat{\Gamma}^{I_1}\psi]_{-}}{\langle s -r \rangle^{\frac 12 +\delta}}  \bigg\|\,\d s \nonumber\\
			& \lesssim \epsilon^2 K_0^N+ C_1^2 \epsilon^{\delta} \epsilon^{2(1-\delta N)},
		\end{align}
		where we use Cauchy-Schwarz inequality and \eqref{priorienergy} to derive the last inequality. Hence \eqref{ineq:kgbound2} is proved. \\
		Finally, we show \eqref{ineq:outenergnkg}.  Since 
		\begin{align}
			-\Box \Gamma^I v + \Gamma^I v = \Gamma^I(\psi^* \gamma^0 \psi).
		\end{align}
		By \eqref{est:cutenergyv} in Lemma \ref{lem:ghostKG}, one can obtain
		\begin{align}
			\|\langle r -t \rangle \chi(r-2t) \Gamma^I v\| & \lesssim \|\langle r \rangle \Gamma^I v(0,x)\|_{H^1} + \|\langle r \rangle \partial_t \Gamma^I v(0,x)\| \nonumber\\
			& + \int_0^t \|\langle r -s \rangle \chi(r-2s) \Gamma^I (\psi^* \gamma^0\psi)\|\, \d  s. \label{ineq:nkgoutene2}
		\end{align} 
		Noting that 
		\begin{align}
			& \|\langle r -s \rangle \chi(r-2s) \Gamma^I (\psi^* \gamma^0\psi)\| \nonumber\\
			&\lesssim \sum_{|I_1| \leq |I|, |I_2|\leq N-6} \|\langle r -s \rangle \chi(r-2s) \widehat{\Gamma}^{I_1}\psi\|\|\widehat{\Gamma}^{I_2} \psi 1_{r\geq 2s+1}\|_{L^{\infty}} \nonumber\\
			& \lesssim C_1^2 \epsilon^{\delta} \epsilon^{2(1-\delta N)} \langle s \rangle^{-\frac 74 +\delta}. \label{ineq:outpsiene}
		\end{align}
		Inserting \eqref{ineq:outpsiene} into \eqref{ineq:nkgoutene2} and using initial condition \eqref{initidirac}, \eqref{initikg}, we get
		\begin{align}
			\|\langle r -t \rangle \chi(r-2t) \Gamma^I v\| & \lesssim K_0 + C_1^2 \epsilon^{\delta} \epsilon^{2(1-\delta N)},
		\end{align}
		which implies \eqref{ineq:outenergnkg}. The proof is done.
	\end{proof}
	In order to close the argument, it suffices to refine the $L^{\infty}$ estimate for  $V^1$. To begin with, we perform a nonlinear transformation for $V^1$, which enables us to get a Klein-Gordon equation with fast decay nonlinearity.  Denote 
	\begin{align} \label{eq:transnon}
		\mathcal{N}_1(v,\psi) := 2v^2\psi^*\gamma^0 \psi, \quad  \mathcal{N}_{2}(\psi^*, \psi) = -2 Q_0(\psi^*, \gamma^0 \psi),
	\end{align}
	where $Q_0(v, w)= \partial_t v \partial_t w - \nabla v \cdot\nabla w$ represents the standard quadratic null term. 
	\begin{proposition}
		Let $v, \psi$ be the solution to the system \eqref{eq:DKG-L} and $V^1$ solve \eqref{eq:kge2}. Denote
		\begin{align}
			\widetilde{V}^1 = V^1 - \psi^*\gamma^0\psi.
		\end{align}
		Then we have 
		\begin{align} \label{eq:transnkgdec}
			-\Box \widetilde{V}^1 + \widetilde{V}^1 =  \mathcal{N}_1(v,\psi) + \mathcal{N}_2(\psi^*,\psi),
		\end{align}
	\end{proposition}
	where $\mathcal{N}_1(v,\psi)$, $\mathcal{N}_2(\psi^*,\psi)$ are given by \eqref{eq:transnon}.
	\begin{proof}
		This follows by a direct computation. Acting $i\gamma^{\mu}\partial_{\mu}$ to both sides of the Dirac equation 
		\begin{align}
			-i \gamma^{\mu}\partial_{\mu} \psi = v \psi,
		\end{align}
		we can obtain 
		\begin{align} \label{eq:wavedirac}
			-\Box \psi = i (\partial_{\mu}v) \gamma^{\mu} \psi -v^2 \psi.
		\end{align}
		On the other hand, 
		\begin{align}
			-\Box \widetilde{V}^1 + \widetilde{V}^1 &= (\Box \psi^*)\gamma^0 \psi + \psi^*\gamma^0(\Box \psi) - 2 Q_0(\psi^*, \gamma^0 \psi) \nonumber\\
			& = 2 v^2 \psi^*\gamma^0 \psi - i \partial_{\mu} v( \psi^* \gamma^0 \gamma^{\mu} \psi - \psi^* (\gamma^{\mu})^* \gamma^0 \psi) - 2 Q_0(\psi^*, \gamma^0 \psi) \nonumber \\
			& = 2 v^2 \psi^*\gamma^0 \psi -2 Q_0(\psi^*, \gamma^0 \psi),
		\end{align}
		in which we use \eqref{equ:gamma}. The proof is completed.
	\end{proof}
	
	\begin{proposition} \label{refkgpoint}
		Suppose that the bootstrap assumption \eqref{priorienergy}-\eqref{priorimax} hold and $N\geq 13$, then for any  $0\leq t<T^*$, we have 
		\begin{align}
			\sum_{|I|\leq N-7}|\Gamma^I V^1(t,x)| & \leq C\Big(\epsilon^2 K_0^N+C_1^2 \epsilon^{2\delta} \epsilon^{2(1-\delta N)} \nonumber \\
			&+C_1 \epsilon^{2(1-\delta N)}(\epsilon^{\delta (N+2)} K_0^N + C_1^3 K_0^2 \epsilon^{\delta})\Big)\langle t + r \rangle^{-\frac 32}, \label{ineq:maxinkg2}
		\end{align}where $C_1 \epsilon <1, 0<\delta <\frac{1}{N+2}$ and $C$ is some  constant independent of $\epsilon$ and $K_0$.
	\end{proposition}
	\begin{proof}
		We fix $|I|\leq N-7$  and  treat $\widetilde{V}^1$ first. Performing $\Gamma^I$ to both sides of  \eqref{eq:transnkgdec}, one can get
		\begin{align}
			-\Box \Gamma^I \widetilde{V}^1 + \Gamma^I \widetilde{V}^1 = \Gamma^I  \mathcal{N}_1(v,\psi) + \Gamma^I \mathcal{N}_2(\psi^*,\psi).
		\end{align} 
		By Corollaries \ref{l2energy}, \ref{maxiestimate},  we have  
		\begin{align}
			& \sum_{|I|\leq N}\|\langle s + r \rangle \Gamma^I \mathcal{N}_1(v,\psi)  \| \nonumber \\
			&\lesssim \sum_{\substack{|I_1|\leq N \\ |I_2|,|I_3|,|I_4|\leq N-4}} \|\Gamma^{I_1} v\| \|\langle s +r \rangle^{\frac 12} \Gamma^{I_2}v\|_{L^{\infty}} \|\langle s +r \rangle^{\frac 12} \widehat{\Gamma}^{I_3}\psi\|_{L^{\infty}}\|\widehat{\Gamma}^{I_4}\psi\|_{L^{\infty}} \nonumber\\
			&+ \sum_{\substack{|I_1|\leq N \\ |I_2|,|I_3|,|I_4|\leq N-4}} \|\widehat{\Gamma}^{I_1} \psi\| \|\langle s +r \rangle^{\frac 12} \widehat{\Gamma}^{I_2}\psi\|_{L^{\infty}} \|\langle s +r \rangle^{\frac 12} \Gamma^{I_3}v\|_{L^{\infty}}\|\Gamma^{I_4}v\|_{L^{\infty}} \\
			& \lesssim C^4_1 K_0^2 \epsilon^{\delta} \epsilon^{2(1-\delta N)} \langle s \rangle^{-\frac 54 }.
		\end{align}
		Owing to \eqref{est:Q0fg} and \eqref{eq:diffwave}, one can get
		\begin{align}
			& \sum_{|I|\leq N-3} \|\langle s + r \rangle \Gamma^I \mathcal{N}_2(\psi^*,\psi)  \| \nonumber \\
			& \lesssim \sum_{|I_1|+|I_2| \leq N-3} \|(|L_0 \widehat{\Gamma}^{I_1} \psi|+|\widehat{\Gamma} \widehat{\Gamma}^{I_1}\psi|) \widehat{\Gamma} \widehat{\Gamma}^{I_2} \psi\| \nonumber \\
			& \lesssim \sum_{|I_1|+|I_2| \leq N-3} \||L_0\Gamma^{I_1} \partial \Psi| |\widehat{\Gamma} \widehat{\Gamma}^{I_2} \psi|\| + \sum_{\substack{|I_1|\leq N-4\\|I_2|\leq N-3}} \|\widehat{\Gamma} \widehat{\Gamma}^{I_1} \psi\|_{L^{\infty}}\|\widehat{\Gamma} \widehat{\Gamma}^{I_2} \psi\| \nonumber\\
			&\lesssim \sum_{\substack{|I_1|\leq N-4\\|I_2|\leq N-3}} \|\widehat{\Gamma} \widehat{\Gamma}^{I_1} \psi\|_{L^{\infty}}\|\widehat{\Gamma} \widehat{\Gamma}^{I_2} \psi\|  +\sum_{\substack{|I_1|\leq N-3\\|I_2|\leq N-6}} \|L_0\Gamma^{I_1} \partial \Psi\|\|\widehat{\Gamma}\widehat{\Gamma}^{I_2} \psi\|_{L^{\infty}} \nonumber \\
			& \qquad \qquad \qquad + \sum_{\substack{|I_1|\leq N-6\\|I_2|\leq N-3}} \|L_0\Gamma^{I_1} \partial \Psi\|_{L^{\infty}}\|\widehat{\Gamma} \widehat{\Gamma}^{I_2} \psi\|  \nonumber\\
			& \lesssim C_1^2 \epsilon^{2\delta} \epsilon^{2(1-\delta N)} \langle s \rangle^{-\frac 34} + C_1 \epsilon^{1-\delta(N-2)}(\epsilon K_0^N + C_1^3 K_0^2 \epsilon^{1-\delta(N-1)}) \langle s \rangle^{-\frac 34 +\delta}, 
		\end{align}
		in which we used Corollaries \ref{l2energy}, \ref{maxiestimate} and \eqref{est:conforpsi}, Corollary \ref{cor:pointscal}. Now we can apply Corollary \ref{cor:KGpoint} to derive
		\begin{align}
			&\sum_{|I|\leq N-7} \langle t + r \rangle^{\frac 32}|\Gamma^I \widetilde{V}^1(t,x)| \nonumber\\
			& \lesssim \sum_{|J|\leq N-2} \|\langle x \rangle^{\frac 32}\log(2+|x|) \Gamma^J  \widetilde{V}^1(0,x)\|+C_1^2 \epsilon^{2\delta} \epsilon^{2(1-\delta N)} \nonumber\\
			& \qquad +C_1 \epsilon^{1-\delta(N-2)}(\epsilon K_0^N + C_1^3 K_0^2 \epsilon^{1-\delta(N+1)}) \nonumber\\
			&\lesssim \epsilon^2 K_0^N +C_1^2 \epsilon^{2\delta} \epsilon^{2(1-\delta N)}+C_1 \epsilon^{2\delta}\epsilon^{1-\delta N}(\epsilon K_0^N + C_1^3 K_0^2 \epsilon^{1-\delta(N+1)}).
		\end{align}
		Additionally, 
		\begin{align}
			&\sum_{|I|\leq N-7} \langle t+r \rangle^{\frac 32}|\Gamma^I(\psi^*\gamma^0 \psi)| \nonumber\\
			& \lesssim \sum_{|I_1|+|I_2|\leq N-7} \langle t+r \rangle^{\frac 32}|\widehat{\Gamma}^{I_1}\psi||\widehat{\Gamma}^{I_2}\psi| \nonumber\\
			& \lesssim C_1^2 \epsilon^{2(1-\delta(N-4))}.
		\end{align}
		Recall that $V^1 = \widetilde{V}^1 + \psi^*\gamma^0\psi $, we thus obtain
		\begin{align}
			&\sum_{|I|\leq N-7} \langle t + r \rangle^{\frac 32}|\Gamma^I V^1(t,x)| \nonumber\\
			& \lesssim \epsilon^2K_0^N +C_1^2 \epsilon^{2\delta} \epsilon^{2(1-\delta N)}+C_1 \epsilon^{2(1-\delta N)}(\epsilon^{\delta (N+2)} K_0^N + C_1^3 K_0^2 \epsilon^{\delta}).
		\end{align}
		This concludes the proof of \eqref{ineq:maxinkg2}.
	\end{proof}
	Last, we can refine all of the bootstrap assumption \eqref{priorienergy}-\eqref{priorimax} by selecting appropriate $C_1$ and $\epsilon$, thus concluding the proof of the main result.
	
	\begin{proof} [Proof of Theorem \ref{thm:DKG-L}] Let $C>1$ be the maximal constant appearing in Propositions \ref{refdiracpo}, \ref{refdiracener}, \ref{refkgener} and \ref{refkgpoint}, we first set $C_1$ by letting
		\begin{align} \label{choicec1}
			CK_0^N \leq \frac 18 C_1.
		\end{align}
		Then fix such $C_1$, and choose $\epsilon_0$ so small that  
		\begin{align}
			CC_1^3 K_0^2 \epsilon_0^{\delta} \leq \frac{1}{16}, \label{choiceep}
		\end{align}
		where we recall $K_0>1, 0<\delta< 1/(N+2)$ and $N\geq 13$. By making such choices of $C_1$ and $\epsilon_0$, for any $0\leq t<T^*$ and $\epsilon<\epsilon_0$, we get from Propositions \ref{refdiracpo}, \ref{refdiracener}, \ref{refkgener} and \ref{refkgpoint} that 
		\begin{equation} \label{priorienergy11}
			\left \{
			\begin{aligned}
				\mathcal{E}_D^{\frac{1}{2}}(t,  \widehat{\Gamma}^I \psi) & \leq \frac 12 C_1 \epsilon^{1-\delta |I|}, \ \ \ |I| \leq N,   \\
				\mathcal{E}_{gst,1}^{\frac{1}{2}}(t,  \Gamma^I V^1) & \leq \frac 12C_1 \epsilon^{2(1-N\delta)}, \ \ \  |I| \leq N,\\
				\|\langle r-t \rangle \chi(r-2t)  \widehat{\Gamma}^I \psi\| & \leq \frac 12C_1 \epsilon^{1-\delta |I|}, \qquad \ \ |I| \leq N, \\
				\|\langle r-t \rangle \chi(r-2t)  \Gamma^I v \|  & \leq \frac 12C_1 K_0, \qquad \ \ |I| \leq N, 
			\end{aligned}
			\right.
		\end{equation}
		and 
		\begin{equation} \label{priorimax22}
			\left \{
			\begin{aligned}
				|\widehat{\Gamma}^I \psi(t,x)| & \leq  \frac 12C_1 \epsilon^{1-\delta(|I|+5)} \langle t+r \rangle^{-\frac{3}{4}+\delta} \langle t-r \rangle^{-1}, \ \ |I| \leq N-5,  \\
				|[\widehat{\Gamma}^I \psi]_{-}| & \leq \frac 12C_1 \epsilon^{1-\delta (|I|+5)} \langle t+r \rangle^{-\frac{7}{4}+\delta},  \ \  |I| \leq N-5,  \\
				|\Gamma^I V^1(t,x)| & \leq \frac 12 C_1 \epsilon^{(1-\delta N)2} \langle t+r \rangle^{-\frac 32}, \ \  |I|  \leq N- 7.
			\end{aligned}
			\right.
		\end{equation}
		This improves the previous bootstrap assumption \eqref{priorienergy}-\eqref{priorimax}, and consequently we have $T^*= \infty$.  Now we can apply Lemma \ref{lem:mksdecay} to get
		\begin{align}
			|\psi(t,x)| & \lesssim \langle t +|x|\rangle^{-1} \sup_{0\leq s\leq 2t,|I|\leq 3} \|\widehat{\Gamma}^I \psi(s)\| \nonumber\\
			& \lesssim \langle t +|x|\rangle^{-1} \epsilon^{1-3\delta}.
		\end{align}
		To conclude the proof, it suffices to show the linear scattering for the system. For this purpose, we rely on Lemmas \ref{lem:scatdir}, \ref{lem:scatterKG}.  Indeed, fix $|I|\leq N$, one can see 
		\begin{align}
			\int_0^{\infty} \|\nabla^I(v\psi)(\tau) \|\, \d\tau & \lesssim \sum_{|I_1|+|I_2|\leq |I|} \int_0^{\infty}\|\nabla^{I_1} v \nabla^{I_2} \psi(\tau)\|\,\d\tau \nonumber\\
			& \lesssim \sum_{|I_1|\leq N-7,|I_2|\leq N} \int_0^{\infty}\|\nabla^{I_1} v\|_{L^{\infty}} \|\nabla^{I_2} \psi(\tau)\| \,\d\tau \nonumber \\
			& + \sum_{|I_1|\leq N,|I_2|\leq N-7}\int_0^{\infty} \left\|\frac{\nabla^{I_1} v}{\langle \tau -r \rangle^{\frac 12 +\delta}} \right\| \big\|\langle \tau -r \rangle^{\frac 12 +\delta} \nabla^{I_2} \psi \big\|_{L^\infty}\d\tau \nonumber\\
			& \lesssim C_1^2 K_0 \epsilon^{1-\delta N}, \label{sca:cond1}
		\end{align}
		where we used \eqref{priorienergy} and \eqref{priorimax} in the last step.
		Additionally, we can compute 
		\begin{align}
			\int_0^{\infty} \|\nabla^I (\psi^* \gamma^0 \psi)\| \,d\tau & \lesssim \sum_{|I_1|+|I_2|\leq |I|}\int_0^{\infty} \big \|[\nabla^{I_1} \psi]_{-} \nabla^{I_2}\psi \big\|\,\d\tau \nonumber\\
			& \lesssim \sum_{|I_1|\leq N-5,|I_2|\leq |I|} \int_0^{\infty}\|[\nabla^{I_1}\psi]_{-}\|_{L^{\infty}} \|\nabla^{I_2}\psi\| \,\d\tau \nonumber\\
			&+ \sum_{|I_2|\leq N-5,|I_1|\leq |I|} \int_0^{\infty} \left\|\frac{[\nabla^{I_1}\psi]_{-}}{\langle \tau -r \rangle^{\frac 12+\delta}} \right\| \|\langle \tau -r \rangle^{\frac 12+\delta}\nabla^{I_2}\psi\|_{L^\infty}\d\tau \nonumber\\
			& \lesssim C_1^2 \epsilon^{2(1-\delta N)}, \label{sca:cond2}
		\end{align}
		in which \eqref{priorienergy} and \eqref{priorimax} are used.  Following \eqref{sca:cond1} and \eqref{sca:cond2}, we can assert that $(\psi, v)$ scatters linearly.
		The proof is completed.
	\end{proof}

	%\section*{Acknowledgment}


\begin{thebibliography}{99}
		
		\bibitem{Alinhac01b} {S. Alinhac.}
		{ The null condition for quasilinear wave equations in two space dimensions {I}}.
		\emph{Invent. Math.} \textbf{145} (2001), no. 3, 597--618.
		
		
		%\bibitem{AlinhacIndiana} S. Alinhac.
		%Semilinear hyperbolic systems with blowup at infinity. 
		%\emph{Indiana Univ. Math. J.} \textbf{55} (2006), no. 3, 1209–1232. 
		
		\bibitem{AlinhacBook} S. Alinhac.
		\emph{Hyperbolic partial differential equations.} 
		Universitext. \emph{Springer, Dordrecht,} 2009. xii+150 pp. 
		
		%	 \bibitem{Wya} L. Andersson, P. Blue, Z. Wyatt and S.-T. Yau.
		%	 Global stability of spacetimes with supersymmetric compactifications.
		%	 Preprint, arXiv:2006.00824. 
		
		\bibitem{Bache} A. Bachelot.
		Problème de Cauchy global pour des systèmes de Dirac-Klein-Gordon. 
		\emph{Ann. Inst. H. Poincaré Phys. Théor.} \textbf{48} (1988), no. 4, 387–422.
		
		\bibitem{Bejenaru-Herr} {I. Bejenaru and S. Herr.}
		{ On global well-posedness and scattering for the massive {D}irac-{K}lein-{G}ordon system}.
		\emph{J. Eur. Math. Soc. (JEMS)} \textbf{19} (2017) no. 8, 2445--2467.	
		
		%								\bibitem{BMS} L.~Bieri, S.~Miao and S.~Shahshahani.
		%								Asymptotic properties of solutions of the Maxwell Klein Gordon equation with small data.
		%								\emph{Comm. Anal. Geom.} 25 (2017), no. 1, 25–96.
		
		%		\bibitem{Boura} N.~Bournaveas.
		%	Low regularity solutions of the Dirac Klein-Gordon equations in two space dimensions. 
		%	\emph{Comm. Partial Differential Equations} \textbf{26} (2001), no. 7-8, 1345–1366.
		
		\bibitem{Boura00} N.~Bournaveas.
		A new proof of global existence for the Dirac Klein-Gordon equations in one space dimension. 
		\emph{J. Funct. Anal.} \textbf{173} (2000), no. 1, 203--213.
		
		
		
		
		%\bibitem{Candy-Herr2} {T. Candy and S. Herr.}
		%{Conditional large initial data scattering results for the {D}irac-{K}lein-{G}ordon system}.
		%\emph{Forum Math. Sigma} \textbf{6} (2018), Paper No. e9, 55 pp.
		
		\bibitem{ChGl74} {J.~M. Chadam and R.~T. Glassey.} 
		On certain global solutions of the Cauchy problem for the (classical) coupled Klein-Gordon-Dirac equations in one and three space dimensions. 
		\emph{Arch. Ration. Mech. Anal.} \textbf{54} (1974), 223--237.
		
		
		
		\bibitem{Choquet} Y. Choquet-Bruhat.
		Solutions globales des équations de Maxwell-Dirac-Klein-Gordon (masses nulles). (French) 
		\emph{C. R. Acad. Sci. Paris Sér. I Math.} \textbf{292} (1981), no. 2, 153–158.
		
		\bibitem{Christodoulou86} D. Christodoulou,
		\textit{Global solutions of nonlinear hyperbolic equations for small initial data}.
		Comm. Pure Appl. Math. {39} (1986), no. 2, 267-282.
		
		
		
		
		\bibitem{Christodoulou09} D. Christodoulou,
		\textit{The formation of black holes in general relativity}.
		EMS Monogr. Math.
		European Mathematical Society (EMS), Zürich, 2009, x+589 pp.
		
		\bibitem{DFS-07} {P. D'Ancona, D. Foschi and S. Selberg.}
		{Null structure and almost optimal local regularity for the  {D}irac-{K}lein-{G}ordon system}.
		\emph{J. Eur. Math. Soc. (JEMS)} \textbf{9} (2007), no. 4, 877--899.
		
		\bibitem{De90}  R. O. Dendy, Plasma Dynamics, Oxford University Press, 1990.
		
		\bibitem{Dong2006} { S. Dong.}
		{Asymptotic behavior of the solution to the Klein-Gordon-Zakharov model in dimension two.}
		\emph{Comm. Math. Phys.} \textbf{384} (2021), no. 1, 587--607. 
		
		
		\bibitem{DLMY} S. Dong, K. Li, Y. Ma and X. Yuan.
		Global Behavior of Small Data Solutions for The 2D Dirac-Klein-Gordon Equations. Trans. Amer. Math. Soc.377(2024), no.1, 649--695.
		
		\bibitem{DoLiZh} S. Dong, K. Li, and J. Zhao.
		Cubic Dirac equations with a class of large data.
		Preprint arXiv:2312.07880. 
		
		
		
		
		
		\bibitem{DW-JDE} {S. Dong and Z. Wyatt.}
		Stability of a coupled wave-Klein-Gordon system with quadratic nonlinearities.
		\emph{J. Differential Equations} \textbf{269} (2020), no. 9, 7470--7497. 
		
		
		\bibitem{DongZoeDKG} S. Dong and Z. Wyatt.
		Hidden structure and sharp asymptotics for the Dirac--Klein-Gordon system in two space dimensions.
		Preprint, arXiv: 2105.13780.
		
		%								\bibitem{DM20} {S. Duan and Y. Ma.}
		%								Global Solutions of Wave-Klein–Gordon Systems in 2+1 Dimensional Space-Time with Strong Couplings in Divergence Form. 
		%								\emph{SIAM J. Math. Anal.} \textbf{54} (2022), no. 3, 2691--2726.
		
		
		
		
		\bibitem{FaWaYa} A. Fang, Q. Wang, and S. Yang.
		Global solution for massive Maxwell-Klein-Gordon equations with large Maxwell field.
		Ann. PDE 7 (2021), no. 1, Paper No. 3, 69 pp.
		
		
		\bibitem{Geor} V. Georgiev.
		Decay estimates for the Klein-Gordon equation.
		\emph{Comm. Partial Differential Equations} \textbf{17} (1992), no. 7-8, 1111--1139.
		
		\bibitem{Geor91} V. Georgiev.
		Small amplitude solutions of the Maxwell-Dirac equations.
		Indiana Univ. Math. J. 40 (1991), no. 3, 845--883.
		
		\bibitem{Geor90} V. Georgiev.
		Global solution of the system of wave and Klein-Gordon equations.
		Math. Z. 203 (1990), no. 4, 683--698.
		
		\bibitem{GrPe} {A. Gr\"{u}nrock and H. Pecher.}
		{Global solutions for the {D}irac-{K}lein-{G}ordon system in two space dimensions}.
		\emph{Comm. Partial Differential Equations} \textbf{35} (2010), no. 1, 89--112.
		
		\bibitem{Guo-N-W} { Z. Guo, K. Nakanishi, and S. Wang.} 
		{Small energy scattering for the Klein-Gordon-Zakharov system with radial symmetry.}
		\emph{Math. Res. Lett.} \textbf{21} (2014), no. 4, 733--755.
		
		
		\bibitem{HomanderBook} L. H\"ormander.
		\emph{Lectures on nonlinear hyperbolic differential equations.} 
		Mathématiques \& Applications (Berlin), 26. Springer-Verlag, Berlin, 1997. viii+289 pp.
		
		\bibitem{Zoe23} {C. Huneau, A. Stingo, and Z. Wyatt,}
  {The global stability of the Kaluza-Klein spacetime}.
  preprint arXiv:2307.15267.
		
		\bibitem{IoPa-WKG}  {A. D. Ionescu and B. Pausader.}
		On the global regularity for a wave-Klein-Gordon coupled system. 
		\emph{Acta Math. Sin. (Engl. Ser.)} \textbf{35} (2019), no. 6, 933–986.
		
		
		\bibitem{IoPa-EKG} {A. D. Ionescu and B. Pausader.}
		\emph{The Einstein-Klein-Gordon coupled system: global stability of the Minkowski solution.}
		\emph{Annals of Mathematics Studies}, Vol 403, Princeton University Press (2022).
		
		
		\bibitem{Kata} S. Katayama.
		Global existence for coupled systems of nonlinear wave and Klein-Gordon equations in three space dimensions.
		\emph{Math. Z.} \textbf{270} (2012), no. 1-2, 487–513.
		
		%	 \bibitem{Katayama-Kubo} {S. Katayama and H. Kubo.}
		%     An alternative proof of global existence for nonlinear wave equations in an exterior domain.
		%     \emph{J. Math. Soc. Jpn.} \textbf{60}, 1135--1170 (2008).
		
		
		\bibitem{Kl86} S. Klainerman, The null condition and global existence to nonlinear wave equations, Part 1 (Santa Fe, N.M., 1984). 
		Lectures in Appl. Math., vol. 23, Amer. Math. Soc., Providence, RI, 1986, pp. 293--326.
		
		
		\bibitem{KlainWave} S. Klainerman.
		Uniform decay estimates and the Lorentz invariance of the classical wave equation.
		\emph{Comm. Pure Appl. Math.} \textbf{38} (1985), no. 3, 321–332.
		
		
		\bibitem{Klainerman85} { S. Klainerman.}
		Global existence of small amplitude solutions to nonlinear Klein-Gordon equations in four spacetime dimensions.
		\emph{Comm. Pure Appl. Math.} \textbf{38} (1985), 631--641.
		
		\bibitem{Klainerman93} {S. Klainerman.}
		Remark on the asymptotic behavior of the Klein-Gordon equation in $\R^{n+1}$. 
		\emph{Comm. Pure Appl. Math.} \textbf{46} (1993), no. 2, 137–144.
		
		\bibitem{Klainerman-WY}{ S. Klainerman, Q. Wang, and S. Yang.}
		{Global solution for massive Maxwell-Klein-Gordon equations}.
		\emph{Comm. Pure Appl. Math.} \textbf{73} (2020), no. 1, 63–109.
		
		%								\bibitem{Kubota-Yokoyama} {K. Kubota and K. Yokoyama.}
		%								Global existence of classical solutions to systems of nonlinear wave equations with different speeds of propagation.
		%								\emph{Jpn. J. Math. (N. S.)} \textbf{27}, 113--202 (2001).
		
		
		\bibitem{PLF-YM-cmp} { P.G. LeFloch and Y. Ma.}
		The global nonlinear stability of Minkowski space for self-gravitating massive fields. 
		The wave-Klein-Gordon model. 
		\emph{Commun. Math. Phys.} \textbf{346} (2016), 603--665.  
		
		
		
		%								\bibitem{LindCMAP} H. Lindblad.
		%								On the lifespan of solutions of nonlinear wave equations with small initial data. 
		%								\emph{Comm. Pure Appl. Math.} \textbf{43} (1990), no. 4, 445–472.
		
		%								\bibitem{LindAnn} H. Lindblad and I. Rodnianski.
		%								The global stability of Minkowski space-time in harmonic gauge. 
		%								\emph{Ann. of Math.} (2) \textbf{171} (2010), no. 3, 1401–1477.
		
		
		\bibitem{MiPeYu} S. Miao, L. Pei and P. Yu.
		On classical global solutions of nonlinear wave equations with large data.
		Int. Math. Res. Not. IMRN (2019), no.19, 5859--5913.
		
		
		
		
		
		\bibitem{LuOh} J. Luk, S.  Oh. Global nonlinear stability of large dispersive solutions to the Einstein equations.
		Ann. Henri Poincaré 23 (2022), no. 7, 2391--2521.
		
		
		\bibitem{LuOhYa} J. Luk, S. Oh and S. Yang.
		Solutions to the Einstein-scalar-field system in spherical symmetry with large bounded variation norms.
		Ann. PDE 4 (2018), no. 1, Paper No. 3, 59 pp.
		
		%	 \bibitem{Ma2008} { Y. Ma.}
		%Global solutions of nonlinear wave-Klein-Gordon system in two spatial dimensions: A prototype of strong coupling case, 
		% J. Differential Equations 287 (2021), 236–294.	 
		
		\bibitem{OTT95} {T. Ozawa, K. Tsutaya, and Y. Tsutsumi,}
		{Normal form and global solutions for the Klein-Gordon-Zakharov equations}, 
		Ann. Inst. H. Poincaré Anal. Non Linéaire 12 (1995), no. 4, 459--503. 
		
		\bibitem{OTT} {T. Ozawa, K. Tsutaya and Y. Tsutsumi.} 
		Global existence and asymptotic behavior of solutions for the Klein-Gordon equations with quadratic nonlinearity in two space dimensions. 
		\emph{Math. Z.} \textbf{222} (1996), no. 3, 341--362. 
		
		\bibitem{OTT99} {T. Ozawa, K. Tsutaya, and Y. Tsutsumi,}
		Well-posedness in energy space for the Cauchy problem of the Klein-Gordon-Zakharov equations with different propagation speeds in three space dimensions. 
		Math. Ann. volume 313 (1999), 127--140.
		
		\bibitem{PsarelliA}
		{M. Psarelli}  
		Asymptotic behavior of the solutions of Maxwell-Klein-Gordon field equations in 4-dimensional Minkowski space.
		Comm. in Part. Diff. Equa., 24 (1999), 223--272. 
		
		
		
		
		\bibitem{Shatah} {J. Shatah.}
		{ Normal forms and quadratic nonlinear Klein--Gordon equations}.
		\emph{Comm. Pure Appl. Math.} \textbf{38} (1985), no. 5, 685--696. 
		
		
		%								\bibitem{SimTaf} {J. C. H.  Simon and E. Taflin.} 
		%								The Cauchy problem for nonlinear Klein-Gordon equations. 
		%								\emph{Comm. Math. Phys.} \textbf{152} (1993), no. 3, 433--478.
		
		
		\bibitem{Sogge} C.D. Sogge.
		\emph{Lectures on non-linear wave equations.}
		Second edition. International Press, Boston, MA, 2008. x+205 pp.
		
		\bibitem{WangQ} {Q. Wang.} 
		{An intrinsic hyperboloid approach for Einstein Klein-Gordon equations}.
		\emph{J. of Diff. Geom.} \textbf{115} (2020), 27--109.	
		
		
		\bibitem{WangX} {X. Wang.}
		{ On global existence of 3D charge critical Dirac-Klein-Gordon system}. 
		\emph{Int. Math. Res. Not. IMRN} 2015, no. 21, 10801--10846.
		
		\bibitem{WeYaYu} D. Wei, S. Yang and P. Yu. On the Global Dynamics of Yang-Mills-Higgs Equations. Comm. Math. Phys. 405 (2024), no.1, Paper No. 4, 54 pp.  
		
		
		\bibitem{Ya18} S. Yang. On the global behavior of solutions of the Maxwell-Klein-Gordon equations. 
		Adv. Math. 326 (2018), 490--520.
		
		\bibitem{YaYu19} S. Yang, P. Yu. On global dynamics of the Maxwell-Klein-Gordon equations. Camb. J. Math. 7(4), 365–467 (2019)
		
		\bibitem{Za72} V. E. Zakharov, Collapse of Langmuir waves. Sov. Phys. JETP 35 (1972) 908--914.
		
		
		
		
		
	\end{thebibliography}
\end{document}